\documentclass[lettersize,journal]{IEEEtran}

\usepackage[caption=false,font=normalsize,labelfont=sf,textfont=sf]{subfig}
\usepackage{textcomp}
\usepackage{stfloats}
\usepackage{verbatim}
\usepackage{graphicx}
\usepackage{cite}
\usepackage{array}
\usepackage{cite}
\usepackage{amsmath,amsthm,amssymb,amsfonts}

\usepackage{xpatch}
\makeatletter
\xpatchcmd{\@thm}{\thm@headpunct{.}}{\thm@headpunct{}}{}{}
\makeatother

\usepackage{bm}
\usepackage[T1]{fontenc}

\usepackage{algorithm,algorithmic}
\usepackage{textcomp}
\usepackage[table]{xcolor}
\usepackage{tabularray}
\newcommand{\cmark}{$\checkmark$}

\usepackage[short]{optidef}

\usepackage{caption}

\usepackage{aliascnt}

\usepackage{booktabs}  
\usepackage{tabularx}  
\newcolumntype{Y}{>{\centering\arraybackslash}X}                 
\newcolumntype{L}{>{\centering\arraybackslash}p{0.045\textwidth}}
\newcolumntype{W}{>{\centering\arraybackslash}p{0.053\textwidth}}
\newcolumntype{S}{>{\centering\arraybackslash}p{0.048\textwidth}}
\renewcommand{\tabcolsep}{3pt}

\newtheorem{theorem}{Theorem}
\newtheorem{assumption}{Assumption}

\newtheorem{lemma}{Lemma}

\newtheorem{definition}{Definition}
\newtheorem{remark}{Remark}

\DeclareMathOperator{\col}{col}

\DeclareMathOperator{\diag}{diag}

\def\bsg{{\mathbf{g}}}

\def\bsq{{\mathbf{q}}}

\def\bsv{{\mathbf{v}}}

\def\bsx{{\mathbf{x}}}
\def\bsy{{\mathbf{y}}}

\def\bsE{{\mathbf{E}}}
\def\bsF{{\mathbf{F}}}

\def\bsH{{\mathbf{H}}}

\def\bsL{{\mathbf{L}}}

\def\bsM{{\mathbf{M}}}

\usepackage{fancyhdr}
\allowdisplaybreaks

\usepackage{colortbl}

\definecolor{LightGray}{gray}{0.9}
\definecolor{LightCyan}{rgb}{0.4863    0.6157    0.5922}
\definecolor{LightGreen}{rgb}{0.8824    0.8000    0.6941}
\definecolor{LightOrange}{rgb}{0.8314    0.7294    0.6784}

\usepackage{multirow}
\usepackage{array}
\newcolumntype{M}[1]{>{\centering\arraybackslash}m{#1}}
\newcolumntype{N}{@{}m{0pt}@{}}

\usepackage{fancyhdr}
\usepackage[T1]{fontenc}     
\usepackage[scaled]{helvet}

\usepackage[hyphens]{url}

\usepackage[
    colorlinks=false,                          
    pdfborder={0 0 1},                     
    linkbordercolor={1 0.3 0.3},                
    citebordercolor={0 1 0},                   
    urlbordercolor={0 0 1},                    
    breaklinks=true                            
]{hyperref}

\hypersetup{pdflinkmargin=0.5pt}

\makeatletter
\ifpdftex
\xpatchcmd{\Hy@StartlinkName}
  {\pdfstartlink}
  {\pdfstartlink height 0.69em depth 0.23em}
  {}{\PackageError{reference-links}{Could not set uniform link dimensions}{}}
\fi
\ifxetex
\newcommand*{\UniformReferenceLinkStrut}{%
  \begingroup
    \setbox\z@=\hbox{%
      \rlap{\raise 0.69em\hbox{\XeTeXLink@space}}%
      \rlap{\lower 0.23em\hbox{\XeTeXLink@space}}%
    }%
    \ht\z@=0pt \dp\z@=0pt \wd\z@=0pt
    \box\z@
  \endgroup
}
\xapptocmd{\hyper@linkstart}{\UniformReferenceLinkStrut}
  {}{\PackageError{reference-links}{Could not set uniform XeLaTeX link dimensions}{}}
\fi
\makeatother

\usepackage{enumitem}

\begin{document}

\title{\fontsize{24pt}{28pt}\selectfont Unified Communication Compression\\ Beyond Global Error Bounds for Distributed Nonconvex Optimization}

\author{Haonan Wang, Minghui Liwang, \IEEEmembership{Senior Member, IEEE}, Yiguang Hong, \IEEEmembership{Fellow, IEEE},\\ Karl H. Johansson, \IEEEmembership{Fellow, IEEE}, 
and  Xinlei Yi, \IEEEmembership{Member, IEEE}        % <-this % stops a space
\thanks{This work was supported in part by the National Natural Science Foundation of China under Grants 62503365, 62573319, 62271424, and 62088101.
(Corresponding author: Xinlei Yi.)}
\thanks{H. Wang, M. Liwang, Y. Hong, and X. Yi are with the Department of Control Science and Engineering,
 College of Electronics and Information Engineering, Tongji University,
 Shanghai 201804, China.
 M. Liwang,  Y. Hong, and X. Yi are also with the State Key Laboratory of Autonomous Intelligent Unmanned Systems,
  and Frontiers Science Center for Intelligent Autonomous Systems, Ministry of Education, 
  and the Shanghai Institute of Intelligent Science and Technology, Tongji University, Shanghai 200092, China
%  the Shanghai Institute of Intelligent Science and Technology, State Key Laboratory of Autonomous Intelligent Unmanned Systems, 
% and Frontiers Science Center for Intelligent Autonomous Systems, Ministry of Education,
%  Beijing 100816, China 
 (e-mail: hnwang@tongji.edu.cn; minghuiliwang@tongji.edu.cn; yghong@iss.ac.cn; xinleiyi@tongji.edu.cn).}
\thanks{K. H. Johansson is with the Division of Decision and Control Systems, 
School of Electrical Engineering and Computer Science, KTH Royal Institute 
of Technology, and is also affiliated with Digital Futures, 10044 Stockholm, 
Sweden (e-mail: kallej@kth.se).}}

% \markboth{IEEE Journal of Selected Topics in Signal Processing, submitted for review}%


\maketitle

\begin{abstract}
% Communication compression has become a powerful tool for alleviating communication burdens in distributed optimization, 
% where the cost function is distributed across agents and minimized through their communication and collaboration.
In this paper, we propose a unified  compression algorithm
for  distributed nonconvex opitmization with both
%  that applies to two general compressor classes,
%  namely
the locally- and globally-bounded communication compressors, including 
% In particular, the locally-bounded class covers 
1-bit compressors, saturating quantizers, 
and
 the globally-bounded
%  class accommodates
  compressors with both relative and absolute compression errors, as well as additional arbitrary bounded noise.
% Together, these two classes form the most general compressor family in the literature,  to the best of our knowledge.
We provide a rigorous convergence analysis in nonconvex settings and establish linear convergence under the Polyak--\L{}ojasiewicz (P--L) condition.
Notably, we establish an $\mathcal{O}(1/\sqrt{T})$ convergence rate
% the first convergence guarantees 
for the locally-bounded class in the distributed nonconvex setting,
matching that achieved by the centralized algorithms with 1-bit compressors,
%  with an $\mathcal{O}(1/\sqrt{T})$ rate matches the centralized 1-bit special case,
  where $T$ denotes the total number of iterations.
Moreover, one initial uncompressed communication round further yields an order-wise improvement to $\mathcal{O}(1/T^{2/3})$.
For the P--L setting and the globally-bounded class, we recover state-of-the-art convergence rates.
% For the P--L setting and the globally-bounded class, our results recover state-of-the-art rates and match the uncompressed counterparts.
% Numerical experiments validate the theoretical results.
\end{abstract}


\section{Introduction}

Advances in communication and digital technologies have driven the development of distributed optimization over networked systems, with broad applications in large-scale learning, signal processing, and control \cite{qian2022distributed,li2020federated,Yang_Survey_2019}.
% Advances in communication and digital technologies have driven the development of distributed optimization over networked systems, with broad applications in large-scale learning and control \cite{Nedic_DistributedControl_2018,Yang_Survey_2019}.
% Communication constraints constitute a significant concern in practical distributed optimization, especially in large neural network training (\cite{tang2024fusionllm}), real-time decision systems (\cite{xu2025communication}), and resource-limited edge devices (\cite{li2021talk}).
% Consider $n$ agents communicating over a communication network and collaboratively solving the following optimization problem:
Such systems typically consist of $n$ agents that communicate and collaborate to solve the following optimization problem:
\begin{align}\label{prob}
    \min_{x\in \mathbb{R}^d} f(x)=\frac{1}{n}\sum_{i=1}^nf_i(x),
  \end{align}
where $x$ is the optimization variable that denotes, for example, the global model parameter in learning tasks, and
$f_i$ (possibly nonconvex) is a private local cost function held by agent $i$.
To solve~\eqref{prob}, numerious algorithms have been developed with
% , providing rigorous
 performance guarantees 
% in terms of convergence and convergence rates,
 from convex \cite{Nedic_Distributed_2009,chen2012diffusion,Shi_EXTRA_2015}
 to nonconvex settings \cite{di2016next,Scutari_PartI_2016}.

However, communication constraint remains a significant practical concern in distributed optimization. 
For example, in distributed training of large-scale neural networks, 
high-dimensional model parameters and frequent update transmissions make communication particularly costly \cite{minaee2024large}.
% the high dimensionality of model parameters and the frequent transmission of updates make communication particularly costly \cite{minaee2024large}. 
Moreover,  real-time decision systems, such as industrial automation, robotics, and autonomous driving, often require low-latency communication \cite{parker2016multiple}. 
Meanwhile, on
% resource-limited 
edge devices used in embedded AI applications, communication is often constrained by limited bandwidth \cite{zhu2020one}. 
To reduce communication overhead, a broad family of communication-efficient methods has been developed, 
including compression \cite{Rabbat_Quantized_2005}, event-triggered communication \cite{Lemmon_Event_2010}, delay-tolerant designs \cite{NIPS2011_f0e52b27}, and asynchronous updates \cite{srivastava2011distributed}.
Among them, compression is a widely used approach due to its simplicity and the finite-precision nature of digital hardware.
In practice, high-precision communication is costly and perfectly accurate communication is impossible.
This motivates the use of a compressor $\mathcal{C}:\mathbb{R}^d\to\mathbb{R}^d$ that reduces the number of communicated bits by mapping a vector to its compressed representation.
Accordingly, one major line of research has focused on developing different forms of compressors.
Representative examples studied in distributed convex optimization include quantizers~\cite{yi2014quantized},
 unbiased compressors~\cite{alistarh2017qsgd}, and contractive compressors~\cite{koloskova2019decentralized}, 
which were later extended to nonconvex settings \cite{tang2018communication}.

Most existing works consider compressors with a global compression property (hereafter, global compressors),
where the compression error $\|\mathcal{C}(x)-x\|$ admits a global error bound for all $x$ in $\mathbb{R}^d$.
In contrast, few works investigate compressors that satisfy only local properties (hereafter, local compressors).
A typical example is the 1-bit compressor, which transmits only one bit per coordinate and thus conveys extremely limited information.
Theoretical analysis for such aggressive compression is usually carried out in convex settings and often require relatively strong assumptions to ensure convergence.
For example, 
% Zhang et al.~
\cite{zhang2023innovation} and 
% Yi et al.~
\cite{Yi_CommunicationCompression_2023} established  convergence under strong convexity and the Polyak--\L{}ojasiewicz (P--L) conditions, respectively.
Although 
% Yi and Hong 
\cite{yi2014quantized} considered convex optimization, they established convergence without an explicit rate.
In nonconvex optimization, 1-bit methods are developed under a centralized server paradigm~\cite{zhu2020one,fan20221}, where the consensus problem does not arise due to a shared initialization and the server-broadcast gradient update.
They achieve an $\mathcal{O}(1/\sqrt{T})$ convergence rate and rely on relatively strong assumptions
 such as homogeneous problem settings and Lipschitz continuity, where $T$ denotes the total number of iterations.
Another representative local compressor is the saturating quantizer. However, many works neglect saturation, effectively rendering it a global compressor \cite{alistarh2017qsgd,Yi_CommunicationCompression_2023}.
Theoretical guarantees under a given or arbitrary saturation level are relatively few
and often rely on strong convexity or the P--L condition \cite{magnusson2020maintaining,xu2024quantized}.
In the convex setting, 
% Saha et al. 
\cite{saha2021decentralized} established  convergence with probability proportional to the saturation level and a rate of $\mathcal{O}(T^{-(1/2-\theta)})$ for some $\theta\in(0,1/2)$.

Given that local compressors have been less explored, designing a unified algorithm that handles both local and global compressors is even more challenging.
Beyond works on saturating quantizers that also cover the unsaturated case,
to the best of our knowledge, 
% only Yi et al. \cite{Yi_CommunicationCompression_2023} 
only 
% Yi et al.~
\cite{Yi_CommunicationCompression_2023} and
% Xu et al.~
\cite{xu2025compressed} developed similar algorithms
% provided a unified algorithm
 applicable to 
	local and global compressors with absolute compression error, albeit under the P--L condition.
For global compressors, there have been efforts toward unifying compression and algorithmic design.
% Richt{\'a}rik et al. 
\cite{richtarik2021ef21} proposed an algorithm that works with both unbiased and contractive compressors. 
% Liao et al. 
\cite{liao2022compressed} further considered a more general class of compressors that unifies unbiased, contractive, and certain biased non-contractive compressors.
% \cite{liao2022compressed} further considers a more general compression model, which unifies unbiased and contractive compressors and additionally accommodates certain biased non-contractive types. 
More recently, 
% Michelusi et al.~
\cite{michelusi2022finite}, 
% Nassif et al.~
\cite{10976577}, and 
% Liao et al.~
\cite{liao2024robust}  progressively 
	moved toward a unified class of compressors that accommodates both relative and absolute compression errors.

\begin{table*}[t]
\centering
\caption{Summary of existing results on distributed
optimization with communication compression.
The third and seventh columns show that compressors considered in
existing studies fall within our classes defined in
Definitions~\ref{ass:local} and~\ref{ass:global}, respectively, 
with a checkmark denoting full-class coverage.
SC: strong convexity;
``Uncomp.\ init.'': uncompressed initial communication.
}
\vspace{-2pt}
\label{tab:compressor_comparison}
\footnotesize
% Keep the original separator width so citation spacing and line breaks stay unchanged.
\newcommand{\tablecitecomma}{\begingroup\setbox0=\hbox{,}\hbox to\wd0{{\scriptsize,}\hss}\endgroup}
\begin{tblr}{
  width=\textwidth,
colspec={
  Q[l,wd=1.675cm] Q[c,wd=1.78cm] X[c] Q[c,wd=1.82cm]
  Q[c,wd=0.3cm]
  Q[l,wd=1.9cm] Q[c,wd=2.07cm] X[c] Q[c,wd=1.92cm]
},
  colsep=2.5pt,
  rowsep=2pt,
  cells={valign=m},
column{1,6}={halign=l,leftsep=0pt},
column{4}={rightsep=0pt},
column{5}={leftsep=0pt,rightsep=0pt },
column{6}={leftsep=0pt},
  row{1,2}={font=\mdseries},
  row{3-11}={ht=10pt,preto={\rule[-2pt]{0pt}{10pt}}},
  % Keep data-row spacing uniform across the two partial separators.
  row{8,9}={belowsep=1.6pt},
hline{1}={1-9}{0.8pt},
hline{2}={1-4,6-9}{0.4pt},
hline{3}={1-9}{0.8pt},
hline{9}={1-4}{wd= 0.4pt},
hline{10}={6-9}{wd=0.4pt},
hline{12}={1-9}{0.8pt},
% vline{5}={9-9}{0.04pt},
  cell{8}{1-5}={bg=black!3},
  cell{7-8}{6-9}={bg=black!3},
}
\SetCell[c=4]{c} Local compressor
&&&
&
&
\SetCell[c=4]{c} Global compressor
&&&
\\
Ref. (Alg.)
& Condition
& Compressor within\qquad\qquad our 
class (Def.~\ref{ass:local})
& Convergence rate
&
&
Ref. (Alg.)
& Condition
& Compressor within\qquad\qquad our 
class (Def.~\ref{ass:global})
& Convergence rate
\\
\hspace{-0.33em}
\cite{yi2014quantized}
& Convex
& One-bit
& No rate
&
&\hspace{-0.33em}
\cite{richtarik2021ef21}
& \mbox{Nonconvex/P--L}
&  Unbiased/contractive 
& \mbox{$\mathcal O(T^{-1})$/linear}
\\
\hspace{-0.33em}
\cite{zhu2020one}\tablecitecomma{} \hspace{-0.3em}\cite{fan20221}
& Centralized
& One-bit
& $\mathcal O(T^{-1/2})$
&
&\hspace{-0.33em}
\cite{liao2022compressed}
& SC
& $C=0$
& Linear
\\
\hspace{-0.33em}
\cite{magnusson2020maintaining}\tablecitecomma{} \hspace{-0.3em}\cite{xu2024quantized}
& SC/P--L
& \mbox{Saturating quantizer}
& Linear
&
&\hspace{-0.33em}
\cite{michelusi2022finite}
& SC
& $C\to0$
& Linear
\\
\hspace{-0.33em}
\cite{saha2021decentralized}
& Convex
& \mbox{Saturating quantizer}
& $\mathcal O(T^{-1/2+\theta})$
&
&\hspace{-0.28em}
\cite{10976577}
& SC
& \cmark
& Neighborhood
\\
\hspace{-0.33em}
\cite{zhang2023innovation}
& SC
& $C=r=1$
& Linear
&
&
\mbox{\hspace{-0.00005em}\cite{Yi_CommunicationCompression_2023}\tablecitecomma{} \hspace{-0.3em}\cite{xu2025compressed},\hspace{-0.2em} Alg. \hspace{-0.3em}1/2}
& \mbox{Nonconvex/P--L}
& $C=0$
& \mbox{$\mathcal O(T^{-1})$/linear}
\\
\mbox{\hspace{-0.00005em}\cite{Yi_CommunicationCompression_2023}\tablecitecomma{} \hspace{-0.3em}\cite{xu2025compressed}, \hspace{-0.2em}Alg. \hspace{-0.3em}3}
& P--L
& $C=r=1$
& Linear
&
&
\mbox{\hspace{-0.00005em}\cite{Yi_CommunicationCompression_2023}\tablecitecomma{} \hspace{-0.3em}\cite{xu2025compressed}, \hspace{-0.2em}Alg. \hspace{-0.3em}3}
& \mbox{Nonconvex/P--L}
& $\delta=1$
& \mbox{$\mathcal O(T^{-1})$/linear}
\\
\SetCell[r=3]{l}
{Ours, unified\\Alg.~\ref{nonconvex:algorithm-pdgd}}
& Nonconvex
& \cmark
& $\mathcal O(T^{-1/2})$
&
&\hspace{-0.33em}
\cite{liao2024robust}
& \mbox{Nonconvex/P--L}
& \cmark
& \mbox{$\mathcal O(T^{-1})$/linear}
\\
&
\mbox{Uncomp. init.}
& \cmark
& $\mathcal O(T^{-2/3})$
&
&
\SetCell[r=2]{l}{Ours, unified\\Alg.~\ref{nonconvex:algorithm-pdgd}}
& Nonconvex
& \cmark
& $\mathcal O(T^{-1})$
\\
&
P--L
& \cmark
& Linear
&
&
&
P--L
& \cmark
& Linear
\\
\end{tblr}
\vspace{-2pt}

\end{table*}

As discussed above, theoretical results for local compressors remain limited,
and a unified treatment of local and global compressors is even less developed.
In particular, theoretical guarantees are still lacking for
(i) several representative local compressors in nonconvex settings, including 1-bit compressors and saturating quantizers,
(ii) a unified treatment of local and global compressors in nonconvex settings, and
(iii) a unified algorithm applicable to local compressors, and to global compressors with both relative and absolute compression errors.
These gaps motivate this paper.

\subsection{Main Contributions}
This paper studies distributed nonconvex optimization with communication compression 
and addresses the theoretical gaps identified above by proposing a unified
algorithm that accommodates both local and global compressors.
%  and establishes convergence in nonconvex setting.
Our contributions are summarized as follows.
	\begin{enumerate}[label=(\roman*)]
		\item
		\looseness=-1
		We simultaneously consider two highly general classes of compressors, the locally-bounded class (Definition~\ref{ass:local}) and the globally-bounded class (Definition~\ref{ass:global}).
		In particular,
		 the novel locally-bounded class covers 1-bit compressors and saturating quantizers.
		The globally-bounded class accommodates compressors with both relative and absolute compression errors.
		% , as well as additional arbitrary bounded noise.
		It allows commonly used global compressors with only relative or absolute error to tolerate additional arbitrary bounded noise (Lemma~\ref{lemma:noise}), 
		and remain within this class under bounded-error composition (Lemma~\ref{lemma:composition}).
		Both the locally- and globally-bounded classes are, to the best of our knowledge, the most general in their respective local and global senses, 
		and together they form the most general compressor family in the distributed optimization literature,
		as summarized in Table~\ref{tab:compressor_comparison}.
		\item 
		We propose a unified  compression algorithm (Algorithm~\ref{nonconvex:algorithm-pdgd}) that applies to both the locally- and globally-bounded compressors.
		To the best of our knowledge, this is the first distributed optimization algorithm that simultaneously handles local compressors, and global compressors with both relative and absolute compression errors.
		% To the best of our knowledge, this is the first distributed optimization algorithm that simultaneously handles global relative error, global absolute error, and local compressors.
		In comparison, 
		% Yi et al. 
		\cite{Yi_CommunicationCompression_2023} considered local and global compressors with only absolute compression error under the P--L condition, 
			while our algorithm accommodates broader classes of compressors in the nonconvex setting.
		This unification exploits the structural parallels between the locally- and globally-bounded compressors,
		 enabling a unified treatment of several key technical challenges.
		\item
		For the locally-bounded compressors, we establish convergence in the nonconvex setting 
		and linear convergence under the P--L condition.
		% To the best of our knowledge, this is the first convergence guarantee for local compressors in the nonconvex setting.
		To the best of our knowledge, this is the first convergence guarantee for general local compressors. 
		 In contrast to distributed convex methods with only 1-bit compressors or saturating quantizers \cite{yi2014quantized,saha2021decentralized}, 
		 we extend the first analysis to general local compressors and distributed nonconvex settings.
		Notably, the achieved $\mathcal{O}(1/\sqrt{T})$ convergence rate matches that of centralized algorithms with 1-bit compressors~\cite{zhu2020one,fan20221}.
		% To the best of our knowledge, this is the first convergence guarantee for general local compressors,
		% with an $\mathcal{O}(1/\sqrt{T})$ nonconvex convergence rate
		% matching that achieved by centralized algorithms with 1-bit compressors~(\cite{zhu2020one,fan20221}).
		% In contrast to distributed convex methods that consider only 1-bit compressors or saturated quantizers (\cite{yi2014quantized,saha2021decentralized}), 
		% we extend the first analysis to the nonconvex setting and general local compressors.
		Moreover, with one initial uncompressed communication round, the rate improves to $\mathcal{O}(1/T^{2/3})$, yielding an order-wise improvement over the existing literature.
		These results are summarized in Theorems~\ref{nonconvex:thm-R1}--\ref{nonconvex:thm-liner}.
		\item
		For the globally-bounded compressors, we obtain an $\mathcal{O}(1/T)$ convergence rate in the nonconvex setting and a linear rate under the P--L condition, 
		matching the state-of-the-art guarantees in \cite{liao2024robust}.
		%  for global compressors with both relative and absolute compression errors.	
		% Moreover, compared with \cite{liao2024robust}, our method is more communication- and memory-efficient by  transmitting one fewer variable per iteration and using two fewer auxiliary variables.
		Moreover, our method is more communication- and memory-efficient than the method in \cite{liao2024robust}, 
		since it transmits only one compressed $d$-dimensional variable per iteration (one less than theirs)
		and uses three $d$-dimensional auxiliary variables (two less than theirs).
		These results are summarized in Theorems~\ref{nonconvex:thm-sm_quantization} and \ref{nonconvex:thm-ft_quantization}, 
		and Theorems~\ref{nonconvex:thm-R1}--\ref{nonconvex:thm-ft_quantization} together demonstrate the unification of our algorithm across the locally- and globally-bounded compressors.
	\end{enumerate}

\subsection{Organization and Notations}
	The rest of this paper is organized as follows.
	Section~\ref{zerosg:sec-preliminary} formulates the problem and introduces the locally- and globally-bounded compressors. 
	Then Section~\ref{section:algo}  presents the proposed unified  compression algorithm, 
	and Section~\ref{zerosg:sec-analysis} provides the convergence analysis.
	%  Section~\ref{zerosg:sec-main-random} states the main results,
	%   and Section~\ref{zerosg:sec-analysis} provides the convergence analysis.
	  Section~\ref{zerosg:sec-simulation} gives numerical simulations.
	% In Section~\ref{zerosg:sec-preliminary}, we formulate the problem and introduce the locally- and globally-bounded compressors.
	% Section~\ref{section:algo} presents the proposed unified  compression algorithm.
	% We then state the main results in Section~\ref{zerosg:sec-main-random} and provide the convergence analysis in Section~\ref{zerosg:sec-analysis}.
	% Numerical experiments are reported in Section~\ref{zerosg:sec-simulation}.
	 Finally, Section~\ref{zerosg:sec-conclusion} concludes the paper.
%  Due to space limitations, all proofs are provided in the supplemental material.

	{\bf Notations:} 
	For any positive integer $n$, denote $[n]=\{1,\ldots,n\}$.
	Let $\|\cdot\|_p$ denote the $\ell_p$-norm, and write $\|\cdot\|=\|\cdot\|_2$.
	$\col(x_1,\ldots,x_k)$ stacks $x_1,\ldots,x_k$ with $x_i\in\mathbb{R}^{d_i}$, and $\diag(t_1,\ldots,t_n)$ is the diagonal matrix with diagonal entries $t_i$.
	${\bf 0}_n$, ${\bf 1}_n$, and ${\bf I}_n$ are the $n$-dimensional all-zero vector, all-one vector, and the $n\times n$ identity matrix, respectively.
	% ${\bf 1}_n$ and ${\bf I}_n$ are the $n$-dimensional all-one vector and the $n\times n$ identity matrix, respectively.
	$\mathrm{Unif}(\cdot)$ denotes the uniform distribution.
	%  and sign$(\cdot)$ denotes the sign function. 
	For a differentiable function $f$, $\nabla f$ denotes its gradient.
	$\rho(\cdot)$ denotes the spectral radius, $\rho_2(\cdot)$ the smallest positive eigenvalue, and $A\otimes B$ the Kronecker product.
	% $A\otimes B$ denotes the Kronecker product.
	For a positive semideﬁnite matrix $A$, define $\|x\|_A=\sqrt{\langle x,Ax\rangle}$.
	The subscripts $i$ and $k$ index the agent and iteration, respectively.
	Let $x_{i,k}\in\mathbb{R}^d$ denote agent~$i$'s local iterate for solving~\eqref{prob} at iteration~$k$,
	 and $v_{i,k}\in\mathbb{R}^d$ the associated dual variable.
	Denote $\bsx_k=\col(x_{1,k},\ldots,x_{n,k})$, $\bsv_k=\col(v_{1,k},\ldots,v_{n,k})$, $\bar{x}_k=\tfrac{1}{n}({\bm 1}_n^\top\otimes{\bf I}_d)\bsx_k$, $\bar{\bsx}_k={\bm 1}_n\otimes\bar{x}_k$,
	$\tilde{f}(\bsx_k)=\sum_{i=1}^{n}f_i(x_{i,k})$, $g_{i,k}=\nabla f_i(x_{i,k})$, $\bsg_k=\nabla\tilde{f}(\bsx_k)$, $\bsg_k^0=\nabla\tilde{f}(\bar{\bsx}_k)$,
	$\bsH=\tfrac{1}{n}{\bm 1}_n{\bm 1}_n^\top\otimes{\bf I}_d$ and $\bar{\bsg}_k^0=\bsH\bsg_k^0={\bm 1}_n\otimes\nabla f(\bar{x}_k)$. 
	Without ambiguity, write $\mathcal{C}(\bsx_k)=\col(\mathcal{C}(x_{1,k}),\ldots,\mathcal{C}(x_{n,k}))$.
	We use the norm equivalence: for any $x\in\mathbb{R}^d$, $\|x\|_p \le \hat{d}\|x\|$ and $\|x\| \le \tilde{d}\|x\|_p$, where $\hat{d}=d^{\frac{1}{p}-\frac{1}{2}}$, $\tilde{d}=1$ for $p\in[1,2]$, and $\hat{d}=1$, $\tilde{d}=d^{\frac{1}{2}-\frac{1}{p}}$ for $p>2$.

\section{Problem Formulation}\label{zerosg:sec-preliminary}

% In this section, we introduce the communication compression methods
% % , the communication network, and the standard assumptions on the cost functions.
% with some standard assumptions.

To solve problem~\eqref{prob}, agents must exchange information with each other, and the presence of communication constraint motivates the use of compression.
% Specifically, we consider two general compressor classes, namely, the locally-bounded compressors (Definition~\ref{ass:local}) and the globally-bounded compressors (Definition~\ref{ass:global}).
% Together they form a general compressor family.

\subsection{Locally- and Globally-Bounded Compressors}

We consider two general compressor classes, namely, the locally-bounded compressors (Definition~\ref{ass:local}) and the globally-bounded compressors (Definition~\ref{ass:global}).
Together they form a general compressor family.

The proposed locally-bounded class guarantees compression accuracy only when the input lies within a local region,
thereby allowing compressors such as 1-bit compressors and saturating quantizers \cite{zhang2023innovation,Yi_CommunicationCompression_2023} that are not covered by standard global error bounds.
In contrast, the globally-bounded class impose error bounds for all inputs and have been more widely studied and used in practice.
Most existing global compressors consider either relative or absolute compression error \cite{alistarh2017qsgd,tang2018communication,koloskova2019decentralized,richtarik2021ef21,zhang2023innovation}, 
 whereas Definition~\ref{ass:global} unifies both types of errors \cite{liao2024robust}.

% We now present the locally-bounded compressor,
% which ensures acceptable compression accuracy only when the input lies within a local region.
\begin{definition}\label{ass:local}
  The compressor $\mathcal{C}:\mathbb{R}^d\mapsto\mathbb{R}^d$ satisfies
  \begin{align}
       &\Big\|\frac{\mathcal{C}(x)}{r}-x\Big\|_p  \le C(1-\delta),
      ~\forall x\in\{x:\|x\|_p\le C\}, \label{nonconvex:ass:compression_non1b}
  \end{align}
where $p \ge 1$, $r>0$, $C>0$ specifies the local region, and $\delta \in (0,1]$ measures the contraction.
  \end{definition}
  \begin{definition}\label{ass:global}
	The compressor $\mathcal{C}:\mathbb{R}^d\mapsto\mathbb{R}^d$ satisfies
	\begin{align}\label{nonconvex:ass:compression_equ_scaling}
		\mathbb{E}_{\mathcal{C}}\Big[\Big\|\frac{\mathcal{C}(x)}{r}-x\Big\|^2\Big]\le (1-\delta)\|x\|^2 + C,~\forall x\in\mathbb{R}^d,
	\end{align}
	where $r>0$, $C>0$ specifies the absolute compression error, $\delta\in(0,1]$ measures the contraction of the relative compression error, 
		and $\mathbb{E}_{\mathcal{C}}[\cdot]$ denotes expectation over
	all compressor randomness, which is independent across agents
	and iterations.
	% and $\mathbb{E}_{\mathcal{C}}[\cdot]$ denotes the expectation over the internal randomness of the compressor $\mathcal{C}$.
\end{definition}
\begin{remark}
	In the local sense, Definition~\ref{ass:local} is the most general form of compressors in the literature, to the best of our knowledge.
	When \(C=1\) and \(r=1\), it reduces to the lcoal compressors with absolute compression error studied in \cite{zhang2023innovation,Yi_CommunicationCompression_2023}. 
	Moreover, Definition~\ref{ass:local} is weaker than many existing global compressor, since it only requires an error guarantee on the local region \(\{x:\|x\|_p\le C\}\).
	% When \(C\) is allowed to vary over all positive values,
	% % When \(C\) is allowed to vary over \((0,\infty)\),
	% % When Definition~\ref{ass:local} holds for any \(C>0\),
	% Definition~\ref{ass:local} reduces to deterministic global compressors with relative compression error as in~\cite{Yi_CommunicationCompression_2023}.
	% If one removes the domain restriction and sets \(r=1\), Definition~\ref{ass:local} reduces to deterministic global compressors with absolute compression error as in~\cite{Yi_CommunicationCompression_2023}.
\end{remark}
\begin{remark}
In the global sense, Definition~\ref{ass:global} is the most general form of compressors in the literature, to the best of our knowledge.
When \(C=0\), it reduces to global compressors with relative compression error as in~\cite{liao2022compressed},
while setting \(\delta=1\) and \(r=1\) yields global compressors with absolute compression error as in~\cite{Yi_CommunicationCompression_2023}.
Compared with the locally-bounded class, this globally-bounded class
 has the distinctive feature of accommodating stochasticity and arbitrary additional bounded noise, 
which can be absorbed into the constant $C$ in \eqref{nonconvex:ass:compression_equ_scaling}.
% Since global compressors with only relative or absolute compression error are widely used in practice,
%  we interpret \(C\) in~\eqref{nonconvex:ass:compression_equ_scaling} as an additional bounded noise term imposed on such compressors,
%  which may arise from practical sources such as communication noise and finite-precision quantization.
% This viewpoint brings a broader set of practical compressors into our compressor family.
% Moreover, composing a global compressor with relative compression error and one with absolute compression error, in either order, also preserves Definition~\ref{ass:global}.
% This compositional preservation continues to hold
%   in the presence of an additional bounded noise term, 
%  irrespective of whether it is applied to the inner compressor, the outer compressor, or both.
%  Such compositions can 
% dramatically increase the variety of compressors
% and reduce communication costs.
\end{remark}

Next, we provide several examples that satisfy Definition~\ref{ass:local},
including
% This class includes 
the 1-bit compressors and saturating quantizers, both of which transmit relatively limited information.
\begin{align*} 
% &\hspace{-0.6em}(\text{1-bit compressor:}~C^\prime>0) \\
&\hspace{-2.45em}(\text{1-bit compressor}) \\
% &\hspace{-0.6em}(\text{1-bit compressor:}~p=\infty,~r=1,~C>0,~\delta\in(0,1/2])  \\
&~[\mathcal{C}_1(x)]_j=
\begin{cases}
+\frac{C^\prime}{2}, & x_j  \ge 0,\\[2mm]
-\frac{C^\prime}{2}, & x_j  < 0,
\end{cases}
~\forall j\in[d],
~~b_v=d,
% \\
% 	& \text{where}~j\in[d]~\text{indexes the coordinates}, \xi_j\in[-1,1] \text{models}\\
% 	& \text{a} \text{sign-flip} \text{noise on the} j\text{-th} \text{coordinate, and} b_v \text{denotes the number of transmitted bits per vector}.
\end{align*} 
% $j\in[d]$ indexes the coordinates, 
% where $\xi_1\in\mathbb{R}^d$ denotes a sign-flip noise with $\xi_{1,j}\in[-C(1/2-\delta),\,C(1/2-\delta)]$ on the $j$-th coordinate, and 
where $C'>0$ is the quantization level, and $b_v$ denotes the number of transmitted bits per vector.
The 1-bit compressors satisfy Definition~\ref{ass:local} with $p=\infty$, $r=1$, $C=C^\prime$, and $\delta\in(0,1/2]$.
% $\xi_{1,j}\in[-C\big(\tfrac{1}{2}-\delta\big),C\big(\tfrac{1}{2}-\delta\big)]$ models a sign-flip perturbation on the $j$th coordinate, and $b_v$ denotes the number of transmitted bits per vector.
\begin{flalign*}
&(\text{Saturating quantizer})&&\\
% &(\text{arbitra19ry saturating quantizers:}~\Delta>0,~C^\prime>0)\\
	% &(\text{arbitrary saturating quantizers:}\hspace{0.1em}~p=\infty,~r=1,~C>0,\\
% &~\delta\in(0,1]) \\
% &(\text{arbitrary saturating quantizers:}\\
% &\hspace{9.5em}p=\infty,~r=1,~C>0,~\delta\in(0,1]) \\
% &[\mathcal{C}_2(x)]_j =
% \Delta\,\min\Big\{ 
% \max\big\{\Big\lfloor \frac{x_j}{\Delta}+\tfrac12+\xi_j\Big\rfloor,\lfloor \frac{-C}{\Delta}\rfloor\big\}, 
% \lfloor \frac{C}{\Delta}\rfloor\Big\},\\
&\quad\hspace{0.4em}\mathcal{C}_2(x) = \min\bigg\{ \max\Big\{\Delta\Big\lfloor \frac{x}{\Delta}+\frac{{\bf 1}_d}{2} \Big\rfloor,-C^\prime \Big\}, 
 C^\prime\bigg\},&&\\
%  & [\mathcal{C}_2(x)]_j= 
%  \begin{cases}
% \Delta\lfloor \frac{C}{\Delta} \rfloor, & \frac{x_j}{\Delta}+\tfrac12+\xi_j>\frac{C}{\Delta},\\[2mm]
% \Delta\lfloor\frac{x_j}{\Delta}+\tfrac12+\xi_j\rfloor, & \frac{x_j}{\Delta}+\tfrac12+\xi_j \in[-\frac{C}{\Delta},\frac{C}{\Delta}], \\[2mm]
% \Delta\lfloor -\frac{C}{\Delta}\rfloor, & \frac{x_j}{\Delta}+\tfrac12+\xi_j<-\frac{C}{\Delta},
% \end{cases}\\
&\quad\hspace{0.4em} b_v = d\Big\lceil \log_2\Big(2\Big\lceil \frac{C^\prime}{\Delta}\Big\rceil+1\Big)\Big\rceil,&&
% ~\forall j\in[d],
\end{flalign*}
\looseness=-1
where $C'>0$ is the saturation level, $\Delta\in(0,2C)$ is the quantization stepsize, 
% where $\Delta>0$ denotes the quantization interval,
 with $\lfloor\cdot\rfloor$ and $\lceil\cdot\rceil$ denoting the floor and ceiling operators, respectively.
% $b_v$ denotes the number of transmitted bits per vector,
When the input magnitude $|x_j|$ exceeds $C'$, the output remains within $[-C^\prime,C^\prime]$.
Note that the saturating quantizers satisfy Definition~\ref{ass:local} with $p=\infty$, $r=1$, $C=C^\prime$, and $\delta\in\allowbreak(0,1-\Delta/(2C^\prime)]$,
but not the local compressors with absolute compression error considered in \cite{zhang2023innovation,Yi_CommunicationCompression_2023}.

Definition~\ref{ass:local} also encompasses many global compressors, as it only requires the compression
%  property to hold 
 within a local region.
% , and imposing an additional saturation (clipping) range---often due to communication constranit---preserves this assumption.
\begin{flalign*}
&(\text{Top-$k$}~\text{compressor})&&\\
	% &(\text{Top-$k$}~\text{compressor:}\hspace{0.1em}~p=2,\hspace{0.1em}~r=1,\hspace{0.1em}~C>0,\hspace{0.1em}~\delta=k/p)\\
% &(\text{Saturated}~\hspace{0.1em} \text{Top-$k$}~\hspace{0.1em} \text{compressor:}\hspace{0.1em}~p=2,\hspace{0.1em}~r=1,\hspace{0.1em}~C>0,\hspace{0.1em}~\delta\\
% &\hspace{0.1em})\\
% &(\text{Saturated Top-$k$ compressor:}\\
% &\hspace{10.8em} p=2,~r=1,~C>0,~\delta=k/p)\\
 &\quad\hspace{0.4em}\mathcal{C}_3(x) =\sum\nolimits_{j=1}^{k}\hspace{-0.1em}[x]_{t_j}e_{t_j},~
 b_v=kb_1\hspace{-0.1em}+\hspace{-0.1em}\min\{k\lceil\log_2d\rceil,d\}.&&\\
% \end{align*}
% where the index $t_j$ refers to one of the $j$-th largest-magnitude components of $x$, and $b_1=64$ corresponds to 64-bit floating-point representation, ensuring exact precision.
% \begin{align*}
&(\text{Norm-sign}~\text{compressor})&&\\
% &(\text{Norm-sign}\hspace{0.1em} ~\text{compressor:}\hspace{0.1em}~p=2,\hspace{0.1em}~r=d/2,\hspace{0.1em}~C>0,~\delta=1/d^2\\
% &\hspace{0.1em}>0,~\delta=1/d^2~\text{and}~p=\infty,~r=1,~C>0,~\delta=1/2)\\
%  &\hspace{0.15em}~~~~~\mathcal{C}_4(x)=\frac{\|x\|_\infty}{2}\text{sign}(x),~\forall x\in\{x\in\mathbb{R}^d:\|x\|_p\le C\},\\
%  &\hspace{0.15em}~~~~~b_v=kb_1.
&\quad\hspace{0.4em}\mathcal{C}_4(x)=\frac{\|x\|_\infty}{2}\text{sign}(x),~b_v=b_1 + \hspace{0.05em}d.&&
\end{flalign*}
Here, $k$ is the number of selected coordinates,
$t_j$ indexes the $j$-th largest magnitude entry of $x$ and $[x]_{t_j}$ denotes the corresponding value,
$b_1=32$ corresponds to 32-bit floating-point representation ensuring exact precision, 
and $\text{sign}(\cdot)$ is applied elementwise.
% The Top-$k$ compressors satisfy Definition~\ref{ass:local} with $p=2$, $r=1$, $C>0$, and $\delta=k/d$. 
% The norm-sign compressors satisfy Definition~\ref{ass:local} with either $p=2$, $r=d/2$, $C>0$, and $\delta=1/d^2$, or $p=\infty$, $r=1$, $C>0$, and $\delta=1/2$.
The top-$k$ compressors satisfy Definition~\ref{ass:local} with $p=2$, $r=1$, $C>0$, and $\delta=1-\sqrt{1-k/d}$, while the norm-sign compressors satisfy Definition~\ref{ass:local} with $p=\infty$, $r=1$, $C>0$, and $\delta=1/2$.
% but they do not fit the local compressors with absolute compression error in \cite{zhang2023innovation,Yi_CommunicationCompression_2023} when $p=2$.

We now turn to the globally-bounded compressors. 
Since global compressors with only relative or absolute compression error are widely used in practice,
 we interpret \(C\) in~\eqref{nonconvex:ass:compression_equ_scaling} as an additional bounded noise term imposed on such compressors,
 which may arise from practical sources such as communication noise and finite-precision quantization.
This viewpoint brings a broader set of practical compressors into our compressor family.
Moreover, composing a global compressor with relative compression error and one with absolute compression error, in either order, also preserves Definition~\ref{ass:global}.
This compositional preservation continues to hold
  in the presence of an additional bounded noise term, 
 irrespective of whether it is applied to the inner compressor, the outer compressor, or both.
 Such compositions can 
dramatically increase the variety of compressors
and reduce communication costs.

\begin{lemma}\label{lemma:noise}
	Suppose that a bounded noise $\xi$ satisfies $\|\xi\|\le C_\xi$ almost surely, 
	and define the perturbed compressors $\tilde{\mathcal{C}}_r(x)=\mathcal{C}_r(x)+\xi$ and $\tilde{\mathcal{C}}_a(x)=\mathcal{C}_a(x)+\xi$. 
	Here, $\mathcal{C}_r,\mathcal{C}_a:\mathbb{R}^d\to\mathbb{R}^d$ are global compressors with relative and absolute compression errors, respectively, and satisfy
	\begin{align}
		\mathbb{E}_{\mathcal{C}_{r}}\Big[\Big\|\frac{\mathcal{C}_r(x)}{r_r}-x\Big\|^2\Big]&\le (1-\delta_r)\|x\|^2,~\forall x\in\mathbb{R}^d, \label{ass:relative}\\
		\mathbb{E}_{\mathcal{C}_{a}}\Big[\Big\|\frac{\mathcal{C}_a(x)}{r_a}-x\Big\|^2\Big]&\le C_a,~\forall x\in\mathbb{R}^d, \label{ass:absolute}
	\end{align}
where $r_r,r_a>0$, $\delta_r\in(0,1]$, and $C_a>0$; 
$\mathbb{E}_{\mathcal{C}_{r}}[\cdot]$ and $\mathbb{E}_{\mathcal{C}_{a}}[\cdot]$ 
denote the expectations over the internal randomness of $\mathcal{C}_r$ and $\mathcal{C}_a$, respectively, 
with 
these two kinds of randomness being independent.
Then $\tilde{\mathcal{C}}_r(x)$ satisfies Definition~\ref{ass:global} with
\begin{align*}
	r=r_r,~
C=C_{\tilde{r}}=(2-\delta_r)C_\xi^2/(\delta_r r_r^2),~
\delta=\delta_r/2,
\end{align*}
and $\tilde{\mathcal{C}}_a(x)$ satisfies Definition~\ref{ass:global} with 
\begin{align*}
	r=r_a,~C=C_{\tilde{a}}=2C_a+2C_\xi^2/r_a^2,~\delta=1.
\end{align*}
\end{lemma}
\begin{proof}
	See Appendix \ref{appendix:bn_co:1}.
\end{proof}
\begin{lemma}\label{lemma:composition}
With $\tilde{\mathcal{C}}_r$ and $\tilde{\mathcal{C}}_a$ defined as in Lemma~\ref{lemma:noise},
the composition $\tilde{\mathcal{C}}_r(\frac{\tilde{\mathcal{C}}_a(x)}{r_a})$ satisfies Definition~\ref{ass:global} with
\begin{align*}
	r\hspace{-0.1em}=\hspace{-0.1em}r_r,~C\hspace{-0.1em}=\hspace{-0.1em}\frac{(4\hspace{-0.1em}-\hspace{-0.1em}\delta_r)C_{\tilde{r}}}{4\hspace{-0.1em}-\hspace{-0.1em}2\delta_r} + \frac{(4\hspace{-0.1em}-\hspace{-0.1em}\delta_r)(12\hspace{-0.1em}-\hspace{-0.1em}\delta_r)C_{\tilde{a}}}{4\delta_r},~\delta\hspace{-0.1em}=\hspace{-0.1em}\frac{\delta_r}{8},
	% r=r_r,~C=(2-\delta_r)(6-\delta_r)\tilde{d}^{2}C_a/(2\delta_r),~\delta=\delta_r/8,
\end{align*}
and $\tilde{\mathcal{C}}_a(\tilde{\mathcal{C}}_r(x))$ satisfies Definition~\ref{ass:global} with
\begin{align*}
	r=r_rr_a,~C=\frac{(4-\delta_r)C_{\tilde{r}}}{4-2\delta_r}+\frac{(4-\delta_r)C_{\tilde{a}}}{\delta_rr_r^2},~\delta=\frac{\delta_r}{4}.
\end{align*}
\end{lemma}
\begin{proof}
	See Appendix \ref{appendix:bn_co:2}.
\end{proof}

Based on Lemmas~\ref{lemma:noise} and~\ref{lemma:composition}, we present several examples that satisfy Definition~\ref{ass:global} but not Definition~\ref{ass:local}. 
The associated parameters in Definition~\ref{ass:global} can be derived from Lemmas~\ref{lemma:noise} and~\ref{lemma:composition}, and are omitted for brevity.
Note that $\mathcal{C}_5$, $\mathcal{C}_6$, $\mathcal{C}_7$, and $\mathcal{C}_9$ do not fall into the class of global compressors with relative or absolute compression error only.
% For compressors with only relative or absolute compression error, adding a bounded noise $\xi_i$ with $\|\xi_i\|^2\le C_0$ for $i=6$--$9$ still preserves Definition~\ref{ass:global}.
\begin{align*} 
	&\hspace{0.em}(\text{Unbiased}~ \text{$k$-bit}~ \text{quantizer}~ \text{with}~ \text{bounded}~ \text{noise})\\
	&\hspace{3.5em}\mathcal{C}_5(x)=\frac{\|x\|_\infty}{2^{k-1}}\text{sign}(x)\circ\hspace{-0.15em}
	\Big\lfloor\frac{2^{k-1}\left|{x}\right|}{\|x\|_\infty}\hspace{-0.06em}+\hspace{-0.05em}\zeta_5\Big\rfloor \hspace{-0.05em}+\hspace{-0.em} \xi_5,\\
	&\hspace{3.5em}b_v=(k + 1)d + b_1.\\
	&\hspace{0.0em}(\text{Rand-$k$ compressor with bounded noise})\\
	&\hspace{3.5em}\mathcal{C}_6(x)=\sum\nolimits_{j=1}^{k}[x]_{r_j}e_{r_j}\hspace{-0.15em}+\xi_6,\\
	&\hspace{3.5em}b_v=kb_1+\min\{k\lceil\log_2d\rceil,d\}.\\
	&\hspace{0.0em}(\text{Scalarization compressor with bounded noise})\\
	&\hspace{3.5em}\mathcal{C}_{7}(x)=\psi_k\hspace{0.1em}(\hspace{0.1em}\psi_k^\top x\hspace{0.1em}) + \xi_{7},~\hspace{0.1em}b_v\hspace{0.1em}=\hspace{0.1em}b_1.\\
	&\hspace{0.0em}(\text{Standard uniform quantizer with bounded noise})\\
	&\hspace{1em}\mathcal{C}_{8}(x)=\Delta\Big\lfloor \frac{x}{\Delta}\hspace{-0.2em}+\hspace{-0.2em}\frac{{\bf 1}_d}{2}\Big\rfloor 
		\hspace{-0.1em}+\hspace{-0.1em} \xi_{8},\hspace{0.2em}b_v = d\,\Big\lceil\hspace{-0.1em}\log_2\!\big(2\Big\lfloor\hspace{-0.2em} \frac{\|x\|_\infty}{\Delta}\hspace{-0.2em}\Big\rfloor\hspace{-0.1em}+\hspace{-0.11em}1\big)\hspace{-0.1em}\Big\rceil.\\
	&\hspace{0.0em}(\text{Composition of Unbiased $k$-bit and standard uniform quant-}\\
	&~\text{izer, both with bounded noise)}\\
	&\hspace{3.5em}\mathcal{C}_{9}(x)=C_5(C_8(x)),~\hspace{0.1em}\mathcal{C}_{10}(x)=C_8(C_5(x)).
\end{align*}
Here, $\xi_i$ for $i=5,\dots,8$ denote bounded noise terms with $\|\xi_i\|\le C_{\xi_i}$ almost surely.
$\zeta_{5}\sim\mathrm{Unif}([0,1]^d)$ is a user-generated random vector used to ensure unbiasedness.
$\circ$ and $|\cdot|$ denote the Hadamard product and absolute value, respectively.
%  $r_j$ denotes a randomly selected non-repetitive coordinate index.
The periodic direction $\psi_k$ is known to all agents in advance, 
% so each round only transmits the scalar $\psi_k^\top x$. 
so only the scalar $\psi_k^\top x$ is transmitted at each round.
$\psi_k$ lies on the Stiefel manifold and satisfies $\mathbb{E}[\psi_k\psi_k^\top]=\mathbf{I}_d/d.$ 
The noise-free version of $\mathcal{C}_{7}$ was given in~\cite{Wang_Scalarized_2025}.
% Let $\hat{C}_5(\cdot)$ and $\hat{C}_8(\cdot)$ denote the corresponding compressors without additive bounded noise.

 \subsection{Communication Network and Cost Functions}

This subsection introduces the communication network and  several standard assumptions on the cost functions.

The communication network among agents is represented by a graph $\mathcal{G}=(\mathcal{V},\mathcal{E},A)$.
Here $\mathcal{V}=[n]$ is the agent set, $\mathcal{E}\subseteq \mathcal{V}\times\mathcal{V}$ is the edge set, and
$A=[A_{ij}]$ is the (weighted) adjacency matrix.
$A_{ij}>0$ if $(j,i)\in\mathcal{E}$, meaning that agent $i$ can receive information from agent $j$, and $A_{ij}=0$ otherwise.
The neighbor set of node $i$ is $\mathcal{N}_i=\{j\in\mathcal{V}:(j,i)\in\mathcal{E}\}$.
Let $D=\mathrm{diag}(d_1,\dots,d_n)$ with $d_i=\sum_{j=1}^n A_{ij}$ and define the Laplacian $L=D-A$.
The graph is undirected if $(i,j)\in\mathcal{E}$ implies $(j,i)\in\mathcal{E}$, in which case $A_{ij}=A_{ji}$.
It is connected if any two agents are linked by a path, i.e., a finite sequence of edges connecting them.
  \begin{assumption}\label{nonconvex:ass:graph}
    The underlying communication network is modeled by an undirected and connected graph $\mathcal G$.
  \end{assumption}

  Next, we present the standard assumptions on the cost functions.
%    used throughout the paper.
  \begin{assumption}\label{nonconvex:ass:fiu}
    Each  local cost function $f_i(x)$ is smooth with constant $\ell>0$, i.e., 
    \begin{align}\label{nonconvex:smooth}
      \|\nabla f_i(x)-\nabla f_i(y)\|\le \ell\|x-y\|,~\forall x,y\in \mathbb{R}^d.
    \end{align}

  \end{assumption}
  \begin{assumption}\label{nonconvex:ass:optset}
    The optimal value of the optimization problem~\eqref{prob} is finite, i.e.,
        $f^\ast = \inf_{x\in\mathbb{R}^d} f(x) > -\infty.$
  \end{assumption}

  \begin{assumption}\label{nonconvex:ass:known_lowerbound}
    A lower bound $f_{\text{low}}$ of the optimal value $f^\ast$ is known, where $f_{\text{low}}\le f^\ast$.
  \end{assumption}
  \begin{remark}
	Assumption~\ref{nonconvex:ass:known_lowerbound} is easily satisfied in a wide range of applications.	
		For instance, most machine learning loss functions (e.g., squared, logistic, and hinge losses) are nonnegative, so one can simply take $f_{\mathrm{low}}=0$.
		Similarly, in engineering and control applications, cost functions are frequently quadratic, which admit an explicit lower bound.
		More generally, even if the cost function can take negative values, a non-tight but valid lower bound is sufficient.
% Assumption~\ref{nonconvex:ass:known_lowerbound} is introduced only to derive an explicit theoretical choice of the algorithm parameters.
Assumption~\ref{nonconvex:ass:known_lowerbound} is introduced only for the locally-bounded case to derive explicit theoretical choices of the algorithm parameters.
Without this assumption, the main convergence results (Theorems~\ref{nonconvex:thm-R1} and~\ref{nonconvex:cor-R1} in Section~\ref{zerosg:sec-analysis}) can still be established, albeit without explicit parameter specifications.
In practice, however, the algorithm can be implemented without knowledge of $f_{\mathrm{low}}$ by tuning the parameters empirically.
% Assumption~\ref{nonconvex:ass:known_lowerbound} is introduced only for the locally-bounded case to derive an explicit theoretical choice of the algorithm parameters.
% All main convergence results (Theorems~\ref{nonconvex:thm-R1} and \ref{nonconvex:cor-R1} in Section \ref{zerosg:sec-analysis}) remain valid without this assumption.
% In practice, one can implement the algorithm without $f_{\mathrm{low}}$ and tune the parameters empirically.
  \end{remark}

	\begin{assumption}\label{nonconvex:ass:fil}
	The global cost function $f$ satisfies the Polyak--\L{}ojasiewicz (P--L) condition with constant $\nu>0$, i.e.,
	\begin{align}\label{nonconvex:equ:plc}
		\frac{1}{2}\|\nabla f(x)\|^2 \ge \nu \big(f(x)-f^\ast\big), 
		~\forall x\in \mathbb{R}^d.
	\end{align}
	\end{assumption}

  \begin{remark}
	The P--L condition relaxes strong convexity and does not imply convexity.
	% , in contrast to the convexity-based assumption in \cite{yi2014quantized}.
% It ensures that every stationary point is globally optimal and allows faster convergence rates.
% It is used to establish faster convergence rates than those in the nonconvex setting,
% and is known to hold for a broad class of nonconvex problems,
Moreover, a broad class of nonconvex problems is known to satisfy the P--L condition,
notably the loss functions of wide neural networks~\cite{liu2022loss}.
% It is used to establish faster convergence rates than in the nonconvex setting.
	% It should be highlighted that P--L condition does not imply convexity, unlike the convexity-based assumptions in \cite{yi2014quantized}.
	% It can be viewed as a nonconvex analogue of strong convexity, and it guarantees that every stationary point is globally optimal and allows faster convergence.
	% The P--L condition is known to hold for many nonconvex problems, including the loss of wide neural networks~\cite{liu2022loss}.
  \end{remark}

	Since Definitions~\ref{ass:local} and~\ref{ass:global} capture the general local and global compressors, respectively,
	 together they form the most general compressor family in the distributed optimization literature.
	Let us solve
	 the distributed optimization problem~\eqref{prob} with the locally- and globally-bounded compressors to reduce communication costs.
	 While existing works typically treat local and global cases separately,  
	 we provide a unified framework
	  and develop a single algorithm that applies to both.

\section{Unified  Compression Algorithm} \label{section:algo}

In this section, we propose a unified  compression algorithm that accommodates both the locally- and globally-bounded compressors.
Then we highlight the key design that exploits the structural parallels between them,
 allowing a unified treatment of several key technical challenges.

\subsection{Algorithm Design}

To solve problem~\eqref{prob} with communication compression, Yi et al. \cite{Yi_CommunicationCompression_2023} proposed the following compressed communication algorithm:
	\begin{subequations}\label{algorithm:Yi}
		\begin{align}
			&\hspace{-0.3em} x_{i,k+1} = x_{i,k}-\alpha\big(\beta\sum\nolimits_{j=1}^{n}L_{ij}\hat{x}_{j,k}+\gamma v_{i,k}+g_{i,k}\big), \label{nonconvex:kia-algo-dc-x-compress}\\
			&\hspace{-0.3em} v_{i,k+1} =v_{i,k}+ \alpha\gamma\sum\nolimits_{j=1}^{n}L_{ij}\hat{x}_{j,k},  \label{nonconvex:kia-algo-dc-v-compress}
		\end{align}
	\end{subequations}
	where $\hat{x}_{j,k}\in\mathbb{R}^d$ is a compression surrogate used in place of $x_{j,k}$, 
	% $g_{i,k}=\nabla f_i(x_{i,k})$ is the local gradient,
	and $\alpha$, $\beta$, and $\gamma$ are positive algorithm parameters with $\alpha$ being the stepsize.
The surrogate $\hat{x}_{j,k}$ admits different updating forms for different compressors.
Specifically, under global compressors with relative compression error, 
one uses
\begin{align*}
	&\hat{x}_{j,k}=a_{j,k}+\mathcal{C}(x_{j,k}-a_{j,k}),\\
	&a_{j,k}=a_{j,k-1}+\omega\,\mathcal{C}(x_{j,k}-a_{j,k-1})
\end{align*}
% $\hat{x}_{j,k}=a_{j,k}+\mathcal{C}(x_{j,k}-a_{j,k})$ with the update
% $a_{j,k}=a_{j,k-1}+\omega\,\mathcal{C}(x_{j,k}-a_{j,k-1})$, 
where $a_{j,k}\in\mathbb{R}^d$ is an auxiliary variable, and $\omega$ is a parameter to regulate the relative compression error.
Under compressors with absolute compression error, a dynamic scaled-difference scheme is adopted as
\begin{align*}
	\hat{x}_{j,k}=\hat{x}_{j,k-1}+s_k\,\mathcal{C}((x_{j,k}-\hat{x}_{j,k-1})/s_k),
\end{align*}
% $\hat{x}_{j,k}=\hat{x}_{j,k-1}+s_k\,\mathcal{C}((x_{j,k}-\hat{x}_{j,k-1})/s_k)$, 
where
$s_k$ is a diminishing scaling factor to regulate the absolute compression error.

	To accommodate more general compressors, our innovation
	% main departure from~\cite{Yi_CommunicationCompression_2023} 
	lies in bringing the updating forms above
	%  several techniques
	  into a unified framework 
	%   to construct 
	 and constructing 
	a new compression surrogate $\hat{x}_{j,k}$.
	% , together with time-varying parameters $\alpha$, $\beta$, and $\gamma$.	
	Specifically, $\hat{x}_{i,k}$ is updated as
	% we approximate $x_{i,k}$ by the compression surrogate $\hat{x}_{i,k}$ updated as
	\begin{align}
		&\hat{x}_{i,k}=\hat{x}_{i,k-1}+\omega s_k q_{i,k}, \label{nonconvex:kia-algo-dc-a}\\
		&q_{i,k}=\mathcal{C}((x_{i,k}-\hat{x}_{i,k-1})/s_k), \label{nonconvex:kia-algo-dc-q}
	\end{align}
	where $q_{i,k}\in\mathbb{R}^d$ is the compressed variable for communication, 
	with $\omega\in (0,1/r]$ being a fixed parameter and $s_k$ a diminishing scaling factor 
	to  regulate the relative and absolute compression error, respectively.

		\looseness=-1
	Note that having $\sum_{j=1}^n L_{ij}\hat{x}_{j,k}$ in \eqref{algorithm:Yi}
	% $\sum_{j=1}^n L_{ij}\hat{x}_{j,k}$ 
	% in \eqref{algorithm:Yi} implies transmitting $\hat{x}_{i,k}$, which 
	does not require sending $\hat{x}_{i,k-1}$ at every iteration in addition to the compressed variable $q_{i,k}$.
	Instead, each agent maintains an auxiliary variable
	\begin{align} \label{nonconvex:kia-algo-dc-b}
		y_{i,k}=y_{i,k-1}+\omega s_k\sum\nolimits_{j=1}^{n}L_{ij}q_{j,k}.
	\end{align} 
	By induction, it directly follows that 
	$y_{i,k}=\sum_{j=1}^{n}L_{ij}\hat{x}_{j,k}$.
	\begin{algorithm}[!tb]
      %\caption{Distributed Algorithm with Stochastic Compressed Communication}
      \caption{Unified  Compression Algorithm}
      \label{nonconvex:algorithm-pdgd}
      \begin{algorithmic}[1]
        \STATE Input: positive parameters  $\alpha$, $\beta$, $\gamma$ and $\omega$; positive sequence $\{ s_{k}\}$. 
        \STATE Initialize: $ x_{i,0}\in\mathbb{R}^d$ and $v_{i,0}=\hat{x}_{i,-1}=y_{i,-1}={\bf 0}_d$.
		% , and $q_{i,0}=\mathcal{C}((x_{i,0}-\hat{x}_{i,-1})/s_0),~\forall i\in[n]$.
        \FOR{$k=0,1,\dots$}
        \FOR{$i=1,\dots,n$ \textbf{in parallel}} 
		% \STATE \textbf{Compression}: \eqref{nonconvex:kia-algo-dc-q}			
		\STATE Compression $q_{i,k}$ via~\eqref{nonconvex:kia-algo-dc-q}.
		\begin{subequations}\label{zerosg:algorithm-random-pd} 
        %   \begin{align}		  
        %     q_{i,k}&=\mathcal{C}((x_{i,k}-\hat{x}_{i,k-1})/s_{k}). \label{nonconvex:kia-algo-dc-q}
        %   \end{align} \vspace{-1.15em} 
        % \STATE  \textbf{Communication}: Send $q_{i,k}$ to $\mathcal{N}_i$ and receive $q_{j,k}$ from $j\in\mathcal{N}_i$.
		\STATE  Send $q_{i,k}$ to $\mathcal{N}_i$ and receive $q_{j,k}$ from $j\in\mathcal{N}_i$.
        % \STATE \textbf{Stochastic zeroth-order gradient}:\\
        %  (i) Sample $\xi_{i,k}$ from the distribution $\mathcal{D}_i$;\\
        %  (ii) Sample $\zeta_{i,k} \in \mathbb{S}^{p}$ from the uniform distribution;\\
        % (iii) Sample $F_i(x_{i,k},\xi_{i,k})$, $F_i(x_{i,k}+\mu _{i,k}\zeta_{i,k},\xi_{i,k})$;\\
        %  (iv) Compute $g^z_{i,k}$ using \eqref{dbco:gradient:model2-st}.
        % \STATE \textbf{Update auxiliary variable}:
		\STATE Update auxiliary variable $\hat{x}_{i,k}$ and $y_{i,k}$ via \eqref{nonconvex:kia-algo-dc-a} and \eqref{nonconvex:kia-algo-dc-b}, respectively.
        %   \begin{align}
        %     \hat{x}_{i,k}&=\hat{x}_{i,k-1}+\omega s_{k}q_{i,k}, \label{nonconvex:kia-algo-dc-a}\\
        %     y_{i,k}&=y_{i,k-1} + \omega s_k\sum_{j=1}^{n}L_{ij}q_{j,k}. \label{nonconvex:kia-algo-dc-b}
        %   \end{align} \vspace{-1.15em} 
        \STATE Update primal and dual variables
        % \vspace{-0.45em} 				
        \begin{align}
            x_{i,k+1} &= x_{i,k}-\alpha(\beta y_{i,k} + \gamma v_{i,k}+g_{i,k} 	), \label{nonconvex:kia-algo-dc-x}\\
            v_{i,k+1} &=v_{i,k}+ \alpha\gamma y_{i,k}.  \label{nonconvex:kia-algo-dc-v}
          \end{align} %\vspace{-1.15em} 
        % \vspace{-0.5em} 
        %   \begin{align}
        %     q_{i,k+1}&=\mathcal{C}((x_{i,k+1}-\hat{x}_{i,k})/s_{k+1}). \label{nonconvex:kia-algo-dc-q}
        %   \end{align} %\vspace{-1.15em} 
        \end{subequations}
        \ENDFOR
        \ENDFOR
        \STATE  Output: $\{x_{i,k}\}$.
      \end{algorithmic}
    \end{algorithm}
By combining \eqref{algorithm:Yi}--\eqref{nonconvex:kia-algo-dc-b},
%  with the time-varying parameters $\alpha$, $\beta$, and $\gamma$, 
 we propose the unified  compression algorithm (Algorithm~\ref{nonconvex:algorithm-pdgd}).

Next, we explain why the same algorithm applies to both the locally- and globally-bounded compressors.
Despite their differences, the two classes share 
% key structures 
the same updating framework \eqref{nonconvex:kia-algo-dc-a} and \eqref{nonconvex:kia-algo-dc-q}
that enables a unified treatment of several technical challenges.
We highlight two key ingredients: contraction handling and absolute error handling.

\subsection{Unified Contraction and Absolute Error Handling}
% 	In both Definition~\ref{ass:local} and \ref{ass:global}, the constant terms $C(1-\delta)$ and $C$ can be viewed as absolute compression error term.
% This motivates a dynamic scaling scheme with a diminishing scaling factor $s_k$, which approximates $x_{i,k}$ by the rescaled compressed signal
% $x_{i,k} \approx s_k\mathcal{C}(x_{i,k}/s_k)$.
% Under this scheme, absolute compression error term is eventually weighted by $s_k$ and driven to zero.
 
% However, the absolute error in the locally-bounded setting can only be controlled when the compressed signal remains in the local region, i.e., $\|x_{i,k}/s_k\|_p\le C$.
% Moreover, the globally-bounded setting includes an additional relative error term.
% To address these issues, we introduce two techniques for unified contraction handling.
% To ensure that the compressed signal stays within the prescribed local region in the locally-bounded setting and to control the relative error term in the globally-bounded setting, we introduce two techniques: 
The locally-bounded compressors are valid only when the input remains in the local region, 
and the globally-bounded compressors include an additional relative error term.
To address these challenges, we introduce two techniques for unified contraction handling.
	\begin{enumerate}[label=(\roman*)]
\item We use a dynamic scaled-difference scheme in \eqref{nonconvex:kia-algo-dc-q}, where the difference $x_{i,k}-\hat{x}_{i,k-1}$ is scaled by $s_k$ and then compressed.
Since $\hat{x}_{i,k-1}$ tracks $x_{i,k}$, the difference remains small.
As a result, in the locally-bounded case, this scheme helps keep the input within the local region, i.e., $\|\hspace{-0.025em}(x_{i,k}\hspace{-0.025em}-\hspace{-0.025em}\hat{x}_{i,k-1})\hspace{-0.05em}/\hspace{-0.15em}s_k\hspace{-0.05em}\|_p\hspace{-0.015em}$$\le\hspace{-0.2em}C$.
In the globally-bounded case, the difference turns the relative error term into $\|x_{i\hspace{-0.05em},k}\hspace{-0.05em}-\hspace{-0.05em}\hat{x}_{i\hspace{-0.05em},k-1}\|\hspace{-0.05em}^2$\hspace{-0.05em}, which vanishes as $\hat{x}_{i\hspace{-0.05em},k\hspace{-0.05em}-\hspace{-0.05em}1}$ approaches $x_{i\hspace{-0.05em},k}$.

\item We use the parameter $\omega\in(0,1/r]$ to exploit the $(1-\delta)$-contraction property in both the locally- and globally-bounded cases,
% of the compressor in both cases, 
thereby driving $\|x_{i,k}-\hat{x}_{i,k-1}\|$ to zero.
The locally-bounded case is more delicate: since $s_k$ is diminishing, it is challenging to ensure that the scaled difference always stays in the local region, i.e.,
$\|x_{i,k}-\hat{x}_{i,k-1}\|_p \le C s_k,~\forall k$.
To this end, we combine the contraction property with a mathematical induction argument to establish this guarantee.
\end{enumerate}
In addition, Definitions~\ref{ass:local} and~\ref{ass:global} involve constant terms $C(1-\delta)$ and $C$, respectively, 
 both of which can be viewed as the absolute compression error terms.
% To handle the absolute error, 
% the diminishing $s_k$ in the scaled-difference scheme plays a key role by driving its contribution to zero.
Our unified absolute error handling relies on the diminishing scaling factor $s_k$ 
in the scaled-difference scheme to drive their contribution to zero.

Let us explain the unified contraction and absolute error handling. Note that we
approximate $x_{i,k}$ by the compression surrogate $\hat{x}_{i,k}$ updated via \eqref{nonconvex:kia-algo-dc-a} and~\eqref{nonconvex:kia-algo-dc-q}.
Under the locally-bounded compressors (Definition~\ref{ass:local}), if $\|(x_{i,k}-\hat{x}_{i,k-1})/s_k\|_p\le C$, 
then the Cauchy--Schwarz inequality and $\omega\in(0,1/r]$ gives
\begin{align} 
	&\|x_{i,k}-\hat{x}_{i,k}\|_p^2 \nonumber\\
	&=\|x_{i,k}-\hat{x}_{i,k-1}-\omega s_k\mathcal{C}((x_{i,k}-\hat{x}_{i,k-1})/s_k)\|_p^2\nonumber\\
	&=s_k^2\Big\|\omega r\Big(\frac{x_{i,k}-\hat{x}_{i,k-1}}{s_k}-\frac{\mathcal{C}((x_{i,k}-\hat{x}_{i,k-1})/s_k)}{r}\Big) \nonumber\\
		&\quad + (1-\omega r)\frac{x_{i,k}-\hat{x}_{i,k-1}}{s_k}\Big\|_p^2 \nonumber\\
	% &\le \omega rs_k^2 \Big\|\frac{x_{i,k}-\hat{x}_{i,k-1}}{s_k}-\frac{\mathcal{C}((x_{i,k}-\hat{x}_{i,k-1})/s_k)}{r}\Big\|_p^2 \nonumber\\
	&\le s_k^2\omega rC^2(1-\delta)^2
		+ (1-\omega r)\|x_{i,k}-\hat{x}_{i,k-1}\|_p^2 \nonumber\\
	&
	\le \big(1-\omega r(2\delta-\delta^2)\big)C^2s_k^2. \label{local:contraction}
\end{align}
% \begin{subequations}
% 	\begin{align} 
% 	&\|x_{i,k}-\hat{x}_{i,k}\|_p^2 \nonumber\\
% 	&=\|x_{i,k}-\hat{x}_{i,k-1}-\omega s_k\mathcal{C}((x_{i,k}-\hat{x}_{i,k-1})/s_k)\|_p^2\nonumber\\
% 	&=s_k^2\Big\|\omega r\Big(\frac{x_{i,k}-\hat{x}_{i,k-1}}{s_k}-\frac{\mathcal{C}((x_{i,k}-\hat{x}_{i,k-1})/s_k)}{r}\Big) \nonumber\\
% 		&\quad + (1-\omega r)\frac{x_{i,k}-\hat{x}_{i,k-1}}{s_k}\Big\|_p^2 \nonumber\\
% 	% &\le \omega rs_k^2 \Big\|\frac{x_{i,k}-\hat{x}_{i,k-1}}{s_k}-\frac{\mathcal{C}((x_{i,k}-\hat{x}_{i,k-1})/s_k)}{r}\Big\|_p^2 \nonumber\\
% 	&\le s_k^2\omega rC^2(1-\delta)^2
% 		+ (1-\omega r)\|x_{i,k}-\hat{x}_{i,k-1}\|_p^2 \label{local:contraction:a}\\
% 	&
% 	\le \big(1-\omega r(2\delta-\delta^2)\big)C^2s_k^2. \label{local:contraction}
% \end{align}
% \end{subequations}
Similarly, under the globally-bounded compressors (Definition~\ref{ass:global}), we have
\begin{align} \label{global:contraction}
	&\mathbb{E}_{\mathcal{C}}[\|x_{i,k}-\hat{x}_{i,k}\|^2]\le \mathbb{E}_{\mathcal{C}}\Big[ \omega r(1-\delta)\|x_{i,k}-\hat{x}_{i,k-1}\|^2 \nonumber\\
	&\quad + \omega rCs_k^2 
	% \hspace{-0.55em}\mathbb{E}_{\mathcal{C}}\hspace{-0.1em}[\|x_{i,k}\hspace{-0.1em}-\hspace{-0.1em}\hat{x}_{i,k}\|^2]\hspace{-0.1em}&\le\hspace{-0.1em} \omega r(1\hspace{-0.1em}-\hspace{-0.1em}\delta)\|x_{i,k}\hspace{-0.1em}-\hspace{-0.1em}\hat{x}_{i,k-1}\|^2 \hspace{-0.1em}+\hspace{-0.1em} \omega rCs_k^2 \nonumber\\
	 + (1-\omega r)\|x_{i,k}-\hat{x}_{i,k-1}\|^2    \Big]   \nonumber\\
	&= (1-\omega r \delta)\mathbb{E}_{\mathcal{C}}[\|x_{i,k}-\hat{x}_{i,k-1}\|^2] + \omega rCs_k^2.
\end{align}
Both $(1-\omega r(2\delta-\delta^2))$ and $(1-\omega r\delta)$ are strictly smaller than $1$ for $\omega\in (0,1/r]$ and $\delta\in(0,1]$, 
revealing the crucial contraction property of the compression error in~\eqref{local:contraction} and~\eqref{global:contraction}.
Meanwhile, a diminishing $s_k$ suppresses the absolute compression error terms $(1-\omega r(2\delta-\delta^2))C^2s_k^2$ and $\omega rCs_k^2$.
In summary, the dynamic scaled-difference scheme and the parameter $\omega$ enable the unified contraction and absolute error handling, 
yielding a single algorithm applicable to both the locally- and globally-bounded compressors.

\section{Convergence Analysis}\label{zerosg:sec-analysis}

Building on the designed algorithm and the insights developed in Section~\ref{section:algo}, we now present the convergence analysis and the main results.
Subsection~\ref{subsec:Lyapunov} provides a preliminary analysis for both the locally- and globally-bounded cases 
using standard Lyapunov techniques~\cite{Yi_CommunicationCompression_2023}.
Our main analytical contributions are developed in the subsequent
subsections.
Specifically,
Subsections~\ref{section:proof:local}
and~\ref{subsec:mainre:local} present the analysis and main results
for the locally-bounded case, respectively,
where the innovation lies in the careful design of the stepsize $\alpha$
and scaling factor $s_k$, together with mathematical induction
to establish that the compressed inputs remain within the local region.
Subsections~\ref{section:proof:global}
and~\ref{subsec:mainre:global} follow the same structure
for the globally-bounded case,
where the innovation lies in combining contraction analysis
for relative compression error with the diminishing scaling factor $s_k$
to control absolute error, thereby handling both simultaneously.
For readability, the constants used in this section are collected in the appendix, including the $\kappa$, $\varepsilon$, $\epsilon$, $\psi$, and $\varphi$ families with arbitrary subscripts, superscripts, or arguments.

\subsection{Lyapunov Functions} \label{subsec:Lyapunov}

To begin, we consider two different Lyapunov functions
for the locally- bounded and globally-bounded cases, respectively:
\begin{align*}
	\mathcal{L}_{1,k}= \sum\nolimits_{i=1}^{4}e_{i,k},~
	\mathcal{L}_{2,k}= \mathcal{L}_{1,k} +  e_{5,k}.
	\end{align*}
	 The Lyapunov component terms are defined as follows. 
	\begin{align*}
	&{\rm (Consensus~error)}  &e_{1,k}&\hspace{-0.05em}=\hspace{-0.05em}\tfrac{1}{2}\|\bsx_{k}\|^2_{\bsE}
									\hspace{-0.05em}=\hspace{-0.05em}\tfrac{1}{2}\hspace{-0.05em}\hspace{-0.1em}\textstyle\sum\nolimits_{i=1}^{n}\hspace{-0.3em}\|x_{i,k}\hspace{-0.2em}-\hspace{-0.2em}\bar{x}_k\|^2\hspace{-0.25em},
									\\
	&{\rm (Dual~term)}        &e_{2,k}&\hspace{-0.05em}=\hspace{-0.05em}\tfrac{1}{2}\Big\|\bsv_{k}+\tfrac{1}{\gamma}\bsg_{k}^0\Big\|^2_{\frac{\beta+\gamma}{\gamma}\bsF},\\
	&{\rm (Cross~term)}       &e_{3,k}&\hspace{-0.05em}=\hspace{-0.05em}\bsx_{k}^\top\bsE\bsF\Big(\bsv_{k}+\tfrac{1}{\gamma}\bsg_{k}^0\Big),\\
	&{\rm (Optimality~gap)}       &e_{4,k}&\hspace{-0.05em}=\hspace{-0.05em}nf(\bar{x}_k)-nf^*,\\
	&{\rm (Compression~error)}\hspace{-0.07em}&e_{5,k}&\hspace{-0.05em}=\hspace{-0.05em}\|\bsx_{k}-\hat{\bsx}_{k}\|^2.
	\end{align*}
	Here, $\hat{\bsx}_k=\col(\hat{x}_{1,k},\dots,\hat{x}_{n,k})$, $\bsE=E\otimes{\bf I}_d$ and $\bsF=F\otimes{\bf I}_d$. 
$E={\bf I}_n-\tfrac{1}{n}{\bm1}_n{\bm1}_n^{\top}$, 
% is the Laplacian of the complete graph,
and $F=[q\ Q]\diag(\lambda_{n+1}^{-1},\Lambda_1^{-1})[q\ Q]^\top$ is motivated by the eigendecomposition of the Laplacian matrix $L$,
where $L=[q\ Q]\diag(0,\Lambda_1)[q\ Q]^\top$, $\Lambda_1=\diag(\lambda_2,\ldots,\lambda_n)$ with $0<\lambda_2\le\cdots\le\lambda_n$, $q={\bm1}_n/\sqrt{n}$, and $\lambda_{n+1}\in[\lambda_2,\lambda_n]$.

	% In Section~\ref{section:proof:local}, we first use mathematical induction to ensure that the iterates remain in the local region, and then establish convergence via a Lyapunov analysis.
	% In Section~\ref{section:proof:global}, we construct a different Lyapunov function and directly prove convergence.

	\begin{lemma} \label{Lemma: Lyapunov Difference}
	Under Assumptions~\ref{nonconvex:ass:graph}--\ref{nonconvex:ass:optset}, take $\alpha<\hat{\kappa}_0$, $\beta=\tau_1\gamma$, $\gamma>\kappa_2$, and $\tau_1\ge \kappa_1$,
	% Under Assumptions~\ref{nonconvex:ass:graph}--\ref{nonconvex:ass:optset}, take $\alpha<\kappa_0$, $\beta=\tau_1\gamma$, $\gamma>\kappa_2$, and $\tau_1\ge \kappa_1$,
	where $\kappa_0$,~$\kappa_1$, and $\kappa_2$ are some positive constants.
			Let $\{x_{i,k}\}$ be generated by Algorithm~\ref{nonconvex:algorithm-pdgd}.
			Then 		
	\begin{align}
		% &\mathcal{L}_{1,k+1}	\le (1-\alpha\varepsilon_4)\mathcal{L}_{1,k} + \alpha\psi_1\|\bar{\bsg}_{k}^0\|^2 \nonumber\\
		% 	&\quad + \alpha n\tilde{d}^2\psi_2\max_{i\in[n]}\|x_{i,k}-\hat{x}_{i,k}\|^2_p, \label{nonconvex:vkLya4_determin}\\
		&\mathcal{L}_{1,k+1}\le \mathcal{L}_{1,k} - \frac{\alpha}{4}\|\bar{\bsg}_{k}^0\|^2
			 +  \alpha n\tilde{d}^2\psi_2\max\nolimits_{i\in[n]}\hspace{-0.1em}\|x_{i,k}-\hat{x}_{i,k}\|^2_p \nonumber\\
			 &\quad -\alpha\varepsilon_3\big(\|\bsx_{k}\|^2_{\bsE} + \big\|\bsv_{k}+\tfrac{1}{\gamma}\bsg_{k}^0\big\|^2_{\bsF}\big), \label{L1:iteration}\\
		&\max\nolimits_{i\in[n]}\hspace{-0.1em}\|x_{i,k+1}-\hat{x}_{i,k}\|^2_p \nonumber\\
			&\le(1+\varepsilon_5+\alpha^2 n\tilde{d}^2\psi_3)
			\max\nolimits_{i\in[n]}\|x_{i,k}-\hat{x}_{i,k}\|_p^2+\alpha^2\psi_4\mathcal{L}_{1,k}. \label{c:iteration}
	\end{align}	
	% where $\psi_2$--$\psi_4$, $\varepsilon_3$ and $\varepsilon_5$ are given in Appendix~\ref{appendix:lemmas}. 
	\end{lemma}
	\begin{proof}
		See Appendix~\ref{appendix:lemmas:c}.
		% The proof is similar to that of 
		% \cite[Eqs.~(85) and~(87)]{Yi_CommunicationCompression_2023}. 
		% Since the derivation is lengthy, we collect the constants used in the proof in
		% Appendix~\ref{appendix:lemmas:c} and provide the full proof in
		% Appendix~\ref{appendix:lemmas}.
	\end{proof}
	% \begin{proof}
	% 	The derivation follows from bounding the Lyapunov components $e_{1,k}$--$e_{4,k}$ separately.		% The proof analyzes the Lyapunov components $e_{1,k}$--$e_{4,k}$ separately.
	% 	We note that the term $\max_{i\in[n]}\|x_{i,k}-\hat{x}_{i,k}\|_p^2$ in~\eqref{L1:iteration} arises from converting an $\ell_2$-norm bound to an $\ell_p$-norm bound.
	% 	The detailed proof is deferred to Appendix~\ref{appendix:lemmas}.
	% \end{proof} \vspace{-0.5em}

	Lemma~\ref{Lemma: Lyapunov Difference} provides a preliminary one-step bound for the Lyapunov function $\mathcal{L}_{1,k}$.
	The compression error $\|x_{i,k}-\hat{x}_{i,k}\|_p^2$ in~\eqref{L1:iteration} depends on $\|x_{i,k}-\hat{x}_{i,k-1}\|_p^2$ for both the locally- and globally-bounded compressors, as shown in~\eqref{local:contraction} and~\eqref{global:contraction}.
	Therefore,  we further analyze the one-step difference of $\|x_{i,k}-\hat{x}_{i,k-1}\|_p^2$ in~\eqref{c:iteration}.

\subsection{Local Region Guarantee} \label{section:proof:local}

This subsection analyzes the convergence under the locally-bounded compressors, where the key step is to show by mathematical induction that the compressed input stays in the local region,
 i.e., 	
\begin{align} \label{ass:induction}
	 \|x_{i,k}-\hat{x}_{i,k-1}\|_p \le C s_k.
\end{align}
% and that $s_k\to 0$.
We first analyze the special case under the P--L condition, and then extend to the nonconvex setting.
% In this section, we analyze the locally-bounded case, where the compression error can be controlled only within the local region.
% In this section, we analyze the locally-bounded case using two key supporting lemmas.
% We introduce two key supporting lemmas.
%  Lemma~\ref{Lemma: Lyapunov Difference} provides a preliminary one-step analysis of $\mathcal{L}_{1,k}$ and the associated compression error, 
% while Lemma~\ref{Lemma:s0} ensures that the transmitted signal stays in the local region,
% which is crucial for controlling the compression error.	

\looseness=-1
The P--L setting generalizes Theorem~7 in~\cite{Yi_CommunicationCompression_2023}, 
where their compressor is a special case of our locally-bounded class with $C=1$ and $r=1$. 
Under the P--L condition and by definition, $\mathcal{L}_{1,k}$ can be upper bounded by 
		$\|\bar{\bsg}_{k}^0\|^2+\|\bsx_{k}\|_{\bsE}^2+\|\bsv_{k}+\frac{1}{\gamma}\bsg_{k}^0\|_{\bsF}^2$.
		Substituting this relation, the induction hypothesis~\eqref{ass:induction}, and~\eqref{local:contraction} 
		into Lemma~\ref{Lemma: Lyapunov Difference} yields the following symmetric bounds.
	% two symmetric contraction inequalities.	
	\begin{lemma} \label{Lemma: PL}
	% Assume~\eqref{ass:induction} holds at iteration $k$.
	%  Under Assumptions~\ref{ass:local}, \ref{nonconvex:ass:graph}--\ref{nonconvex:ass:optset} and \ref{nonconvex:ass:fil}, 
	Under Definition~\ref{ass:local} and Assumptions~\ref{nonconvex:ass:graph}--\ref{nonconvex:ass:optset} and~\ref{nonconvex:ass:fil},
	take $\alpha\in(0,\tilde{\kappa}_0),~\beta=\tau_1\gamma,~\gamma>\kappa_2,~\omega\in(0,1/r],$ and $\tau_1>\kappa_1$,
	 where $\tilde{\kappa}_0$ is a positive constant.
	Let $\{x_{i,k}\}$ be generated by Algorithm~\ref{nonconvex:algorithm-pdgd}. 
	If \eqref{ass:induction} holds at iteration $k$, then	
	\begin{align}
		% &\mathcal{L}_{1,k+1}	\le (1-\alpha\varepsilon_4)\mathcal{L}_{1,k} + \alpha\psi_1\|\bar{\bsg}_{k}^0\|^2 \nonumber\\
		% 	&\quad + \alpha n\tilde{d}^2\psi_2\max\nolimits_{i\in[n]}\|x_{i,k}-\hat{x}_{i,k}\|^2_p, \label{nonconvex:vkLya4_determin}\\
		&\mathcal{L}_{1,k+1}\le (1-\alpha\varepsilon_6)\mathcal{L}_{1,k} + \alpha n\tilde{d}^2(1-2\varepsilon_5)\psi_2C^2s_k^2, \label{L1:iteration:induction}\\* 
		&\max\nolimits_{i\in[n]}\hspace{-0.1em}\|x_{i,k+1}\hspace{-0.1em}-\hspace{-0.05em}\hat{x}_{i,k}\|^2_p \le 
		\hspace{-0.1em}(1\hspace{-0.1em}-\hspace{-0.1em}\varepsilon_7)C^2s_k^2\hspace{-0.1em}+\hspace{-0.05em}\alpha^2\psi_4\mathcal{L}_{1,k}.\hspace{-0.1em} \label{c:iteration:induction}
	\end{align}	
	% where $\psi_2$--$\psi_4$, $\varepsilon_3$ and $\varepsilon_5$ are given in Appendix~\ref{appendix:lemmas}. 
	\end{lemma}
	\begin{proof} 
				% The detailed proof is provided in
				See Appendix~\ref{appendix:lemma:localpl}.
\end{proof}
	% \begin{proof}
	% 	See Appendix~\ref{appendix:lemma:localpl}.
	% \end{proof}
	Note that the symmetric contraction structure in~\eqref{L1:iteration:induction} and~\eqref{c:iteration:induction} 
	enables the induction and implies 
	linear convergence of $\mathcal{L}_{1,k}$ and $\max_{i\in[n]}\hspace{-0.05em}\|x_{i,k}\hspace{-0.1em}-\hspace{-0.05em}\hat{x}_{i,k-1}\|_p^2$, 
	and hence the algorithm. 
		% \begin{proof}
	% 	The derivation follows from bounding the Lyapunov components $e_{1,k}$--$e_{4,k}$ separately.		% The proof analyzes the Lyapunov components $e_{1,k}$--$e_{4,k}$ separately.
	% 	We note that the term $\max_{i\in[n]}\|x_{i,k}-\hat{x}_{i,k}\|_p^2$ in~\eqref{L1:iteration} arises from converting an $\ell_2$-norm bound to an $\ell_p$-norm bound.
	% 	The detailed proof is deferred to Appendix~\ref{appendix:lemmas}.
	% \end{proof} \vspace{-0.5em}

	\looseness=-1
	However, without the P--L condition, a contraction structure of
	$\mathcal{L}_{\hspace{-0.1em}1\hspace{-0.05em},\hspace{-0.05em}k}$ may not follow directly from~\eqref{L1:iteration}.
	Instead, it reduces to
	\begin{align}\label{accumulate}
		\mathcal{L}_{1,k}\le \mathcal{L}_{1,0} 
			+  \alpha n\tilde{d}^2\psi_2\sum\nolimits_{t=0}^{k-1}\max\nolimits_{i\in[n]}\hspace{-0.1em}\|x_{i,t}-\hat{x}_{i,t}\|^2_p, 
	\end{align} 
	where we assume that the induction hypothesis~\eqref{ass:induction} holds for all $t\in[0,k]$.
	Thus, previous compression errors accumulate in $\mathcal{L}_{1,k}$ and 
	% can
	 amplify the next compressed input $\max_{i\in[n]}\|x_{i,k+1}-\hat{x}_{i,k}\|_p^2$ in~\eqref{c:iteration}.
% These difficulties makes the induction nontrivial in the nonconvex setting.
These difficulties preclude a direct induction in the nonconvex setting, 
	where existing analysis does not establish convergence for local compressors.
% To address this issue, we design a special stepsize to regulate the accumulation of previous compression errors, 
% % To address this issue, we combine mathematical induction
% % 	with a special stepsize design to regulate the accumulation of previous compression errors, 
% 	which appear 
% 	% as $\alpha n\tilde{d}^2\psi_2\sum\nolimits_{t=1}^{k}\max_{i\in[n]}\|x_{i,k}-\hat{x}_{i,k}\|_p^2$ in~\eqref{accumulate} 
% 	% and $\alpha^2\psi_4\mathcal{L}_{1,k}$ in~\eqref{c:iteration}.
% 	in the last terms of both \eqref{accumulate}  and  \eqref{c:iteration}.
To address this issue, we design a special stepsize $\alpha$ for the last terms of both \eqref{accumulate} and \eqref{c:iteration}
 to regulate the accumulation of previous compression errors.
Moreover, \eqref{local:contraction} and \eqref{c:iteration}  still yield a contraction structure for the compressed input $\|x_{i,k}-\hat{x}_{i,k-1}\|_p^2$ in~\eqref{c:iteration:induction}.
% Finally, with a special stepsize design, combining~\eqref{accumulate} and~\eqref{c:iteration:induction} closes the induction and yields the convergence rate.
Finally, combining~\eqref{accumulate} and~\eqref{c:iteration:induction} completes the induction 
for $\|x_{i,k}-\hat{x}_{i,k-1}\|_p \le Cs_k$. 
	% with the specially constructed scaling factor $s_k$.

	\begin{lemma} \label{Lemma:s0}
	Under Definition~\ref{ass:local} and Assumptions~\ref{nonconvex:ass:graph}--\ref{nonconvex:ass:known_lowerbound},  
    let $\{x_{i,k}\}$ be generated by Algorithm~\ref{nonconvex:algorithm-pdgd} with
    \begin{align}
		&\hspace{0.0em}\alpha\in(0,\tilde{\kappa}^{\prime}_0(T)),~\beta=\tau_1\gamma,~\gamma>\kappa_2,~\omega\in(0,1/r]\nonumber\\
		&~~s_0\ge \max\nolimits_{i\in[n]}\hspace{-0.1em}\|x_{i,0}\|_p/C,~\tau_1\ge \kappa_1,~T>\tilde{\kappa}_3, 
		% &\hspace{0.1em}\beta=\tau_1\gamma,~\gamma>\kappa_2,~\alpha\in(0,\kappa_6(T)),~\omega\in(0,1/r]\nonumber\\
		% &~~~~s_0\ge \max_{i\in[n]}\|x_{i,0}\|_p^2/C,~\tau_1\ge \kappa_1,~T>\kappa_T, 
    \end{align}
	where $\tilde{\kappa}_3$ and $\tilde{\kappa}^{\prime}_0(T)$ 
	% $\kappa_T$ and $\kappa_6(T)$
	 are some positive constants. Then, there exists a selectable sequence $\{s_k\}$ such that 
	% where $\kappa_T$ and $\kappa_6(T)$ are positive constants given in Appendix~\ref{appendix:lemma2}. Then, there exists a selectable sequence $\{s_k\}$ such that 
	the compressed input remains bounded, i.e.,
	% , and moreover,
		\begin{align}
			% ~~\max_{i\in[n]}\|x_{i,k}-\hat{x}_{i,k-1}\|_p^2 \le s_0^2, 
			\hspace{-0.2em}\max\nolimits_{i\in[n]}\hspace{-0.1em}\|x_{i,k}\hspace{-0.1em}-\hspace{-0.05em}\hat{x}_{i,k-1}\|_p^2 
			\hspace{-0.05em}\le\hspace{-0.1em} C^2 s_k^2 \hspace{-0.05em}\le\hspace{-0.1em} C^2 s_0^2,
			~\forall\,k \in [0,T]\hspace{-0.1em}.\hspace{-0.2em}
			\label{ineq:skbound} 
		\end{align} 
	Furthermore, take
		\begin{align}
				&~~\hspace{0.2em}\alpha=\frac{1}{n^{\frac{1}{4}}\tilde{d}\sqrt{T}},~\beta=\tau_1\gamma,~\gamma>\kappa_2,~\omega\in(0,\frac{1}{r}],\nonumber\\
			&\quad \tau_1\ge \kappa_1,~T>\tilde{\kappa}_3,~s_0\ge \max\nolimits_{i\in[n]}\hspace{-0.1em}\|x_{i,0}\|_p/C, \nonumber\\
			&~~~~~s_k^2=(1-\varepsilon_8)s_{k-1}^2 + \tilde{\kappa}_4(T),~\forall k\in[1,T], \label{thm1:set}
			% &~~~\alpha=\frac{1}{n^{\frac{1}{4}}\tilde{d}\sqrt{T}},~\beta=\tau_1\gamma,~\gamma>\kappa_2,~\omega\in(0,\frac{1}{r}],\nonumber\\
			% &~~\quad \tau_1\ge \kappa_1,~T>\kappa_T,~s_0\ge \max_{i\in[n]}\|x_{i,0}\|_p/C, \nonumber\\
			% &~~~~~s_k^2=(1-\varepsilon_8)s_{k-1}^2 + \kappa_s(T),~\forall k\in[1,T], \label{thm1:set}
		\end{align}
	where
	%  $\kappa_1$ and $\kappa_2$ are given in Appendix~\ref{appendix:lemmas}, and $\kappa_T$,
	$\varepsilon_8\in(0,1)$ and $\tilde{\kappa}_4(T)$
	% $\varepsilon_8$ and $\kappa_s(T)$ 
	are some positive constants.
	% $\varepsilon_8$ and $\kappa_s(T)$ are positive constants given in Appendix~\ref{appendix:lemma2}.
	% the same setting as 
	% Theorem~\ref{nonconvex:thm-R1}, 
	Then 
	% $s_k$ converges to zero, and it holds that
		% \begin{align}
		% 	s_k^2 = (1-\varepsilon_8)s_{k-1}^2 + \kappa_s(T),
		% 	\quad \forall k\in[1,T].
		% \end{align}
		% Then 
	\begin{align} \label{bound:sk}
		&\max\nolimits_{i\in[n]}\hspace{-0.1em}\|x_{i,k}-\hat{x}_{i,k-1}\|_p^2 \le C^2s_k^2 \le 
		(1-\varepsilon_8)^k C^2s_0^2 \nonumber\\
		&\quad + \frac{\psi_4\mathcal{L}_{1,0}}{\varepsilon_8n} \frac{n^{\frac{1}{2}}}{\tilde{d}^2T} 
			+ \frac{(1-2\varepsilon_5)\psi_2\psi_4}{\varepsilon_8 }C^2s_0^2  \frac{n^{\frac{1}{4}}}{\tilde{d}\sqrt{T}}.
	\end{align}
	\end{lemma}
	% \begin{proof}
	% 	See Appendix~\ref{appendix:lemma2}.
	% \end{proof}
	\begin{proof}
		% \mbox{}
	% \par
	\noindent {\bf (i)}
			% We first show that, with suitable parameters and a scaling sequence $\{s_k\}$,
			% %  satisfying $ \max_{i\in[n]} \|x_{i,k}-\hat{x}_{i,k-1}\|_p^2/C \le s_k^2  \le s_0^2$, 
			% % the compression error $\max_{i\in[n]}\|x_{i,k}-\hat{x}_{i,k-1}\|_p^2$ 
			% the compressed signal $\max_{i\in[n]}\|x_{i,k}-\hat{x}_{i,k-1}\|_p^2$ remains bounded.
			% %  by $C^2s_0^2$. 
			% % Consequently, such a sequence always exists—one may simply take $s_k \equiv s_0$, which ensures boundedness.
			% Using mathematical induction, we prove that there exists a sequence $\{s_k\}$ such that
			% $
			% \max_{i\in[n]} \|x_{i,k}-\hat{x}_{i,k-1}\|_p^2 \le C^2s_k^2 \le C^2s_0^2
			% $
			% for all $k \in [0,T]$.  
			We first show by induction that, with suitable parameters, there exists a scaling sequence $\{s_k\}$ satisfying
			\[
			\max\nolimits_{i\in[n]}\|x_{i,k}-\hat{x}_{i,k-1}\|_p^2 \le C^2 s_k^2 \le C^2 s_0^2,\quad \forall k\in[0,T],
			\]
			which guarantees bounded inputs to the compressors.

			For $k=0$, the inequality holds trivially.  
			Suppose that the statement holds for $k = 0,1,\ldots,\tau~(\tau\le T-1)$, meaning that there exist $s_0,\ldots,s_\tau$ such that
			\begin{align} 
				% \label{assume:induction}
				\hspace{-0.5em}\max\nolimits_{i\in[n]}\|x_{i,k}-\hat{x}_{i,k-1}\|_p^2 \le C^2s_k^2 \le C^2s_0^2,~\forall k\in[0,\tau].
				\nonumber
			\end{align}
			% Here, $s_k$ denotes an arbitrary nonnegative number satisfying the above inequality.
			We next show that the inequality also holds for $k = \tau+1$ by properly choosing $s_{\tau+1}$.

			% Based on the induction hypothesis, we have
			Based on the induction hypothesis and following the same arguments as in~\eqref{local:contraction}, we have
			\begin{align}\label{c:1}
				&\|x_{i,k}-\hat{x}_{i,k}\|_p^2 
				% \nonumber\\
				% 	&
				% =\|x_{i,\tau}-\hat{x}_{i,\tau-1}-\omega s_\tau\mathcal{C}((x_{i,\tau}-\hat{x}_{i,\tau-1})/s_\tau)\|_p^2\nonumber\\
				% &=\Big\|\omega rs_\tau\Big(\frac{x_{i,\tau}-\hat{x}_{i,\tau-1}}{s_\tau}-\frac{\mathcal{C}((x_{i,\tau}-\hat{x}_{i,\tau-1})/s_\tau)}{r}\Big) \nonumber\\
				% 	&\quad + (1-\omega r)(x_{i,\tau}-\hat{x}_{i,\tau-1})\Big\|_p^2 \nonumber\\
				% &\le \omega rs_\tau^2\Big\|\frac{x_{i,\tau}-\hat{x}_{i,\tau-1}}{s_\tau}-\frac{\mathcal{C}((x_{i,\tau}-\hat{x}_{i,\tau-1})/s_\tau)}{r}\Big\|_p^2 \nonumber\\
				% 	&\quad + (1-\omega r)\|x_{i,\tau}-\hat{x}_{i,\tau-1}\|_p^2 \nonumber\\
				\le (1-2\varepsilon_5)C^2s_k^2
					\le (1-2\varepsilon_5)C^2s_0^2, \nonumber\\
					&\hspace{15em} \forall k\in[0,\tau].
			\end{align}
			% where the first equality holds due to the Minkowski inequality and the Cauchy--Schwarz inequality;
			% and the last inequality holds due to \eqref{nonconvex:ass:compression_non1b} and \eqref{assume:induction}.
			Building on this, unrolling the inequality \eqref{L1:iteration} gives
			\begin{align} \label{bound:L2_induction}
				&\mathcal{L}_{1,\tau} 
				\le \mathcal{L}_{1,0}  + (1-2\varepsilon_5)\psi_2n\tilde{d}^2C^2s_0^2\alpha \tau. 
			\end{align}

			Substituting \eqref{c:1} and \eqref{bound:L2_induction}  into \eqref{c:iteration} leads to
			\begin{align}
				&\max\nolimits_{i\in[n]}\|x_{i,\tau+1}-\hat{x}_{i,\tau}\|^2_p \nonumber\\
					&\le(1+\varepsilon_5+\alpha^2 n\tilde{d}^2\psi_3)(1-2\varepsilon_5)C^2s_0^2  \nonumber\\
					&\quad + \alpha^2\psi_4\big(\mathcal{L}_{1,0}  + (1-2\varepsilon_5)\psi_2n\tilde{d}^2C^2s_0^2\alpha \tau \big) \nonumber\\
				&\le\big(1-(\varepsilon_5+2\varepsilon_5^2)+(1-2\varepsilon_5)\psi_3 n\tilde{d}^2\alpha^2\big)C^2s_0^2  \nonumber\\
					&\quad	+ \psi_4\mathcal{L}_{1,0}\alpha^2 
						+ (1-2\varepsilon_5)\psi_2\psi_4n\tilde{d}^2C^2s_0^2 \alpha^3T,  
						\nonumber
			\end{align}
			where the last inequality holds due to $\tau\le T$.

			Noting $\alpha<\tilde{\kappa}^{\prime}_0(T)\le\kappa_{8}(T)$, we have
			\begin{align} 
				&\max\nolimits_{i\in[n]}\|x_{i,\tau+1}-\hat{x}_{i,\tau}\|^2_p\le (1-\varepsilon_8)C^2s_0^2 \nonumber\\
				&\quad	+ \psi_4\mathcal{L}_{1,0}\alpha^2 
						+ (1-2\varepsilon_5)\psi_2\psi_4n\tilde{d}^2C^2s_0^2 \alpha^3T \label{mid:s0}\\
				&\le (1-\varepsilon_8)C^2s_0^2 + \varepsilon_8 C^2s_0^2 \nonumber\\
				&\le C^2s_0^2. \label{sk:iteration:0}
			\end{align}
			Therefore, there exists a number $s_{\tau+1}$ satisfying
			$
			\max_{i\in[n]}\|x_{i,\tau+1}-\hat{x}_{i,\tau}\|_p^2 
			\le C^2s_{\tau+1}^2 \le C^2s_0^2,
			$
			which completes the induction.

	\noindent {\bf (ii)} 
	% To further control the compression error, we show that the sequence $\{s_k\}$ constructed in Theorem~\ref{nonconvex:thm-R1} converges to zero, and  prove by mathematical induction that 
	% $\max_{i\in[n]}\|x_{i,k}-\hat{x}_{i,k-1}\|_p^2 \le C^2 s_k^2$ 
	% for all $k\in[0,T]$.
	To further control the compression error, we prove by mathematical induction that the sequence $\{s_k\}$ constructed in Theorem~\ref{nonconvex:thm-R1} 
	decay to $s_T=\mathcal{O}(T^{-1/4})$, and
	$\max\nolimits_{i\in[n]}\hspace{-0.1em}\|x_{i,k}-\hat{x}_{i,k-1}\|_p^2 \le C^2 s_k^2$ 
	for all $k\in[0,T]$.

		We first show that the sequence $\{s_k\}$ remains bounded by its initial value $s_0^2$. 
	Given $\alpha=1/(n^{\frac{1}{4}}\tilde{d}\sqrt{T})$ and $T>\tilde{\kappa}_3$, 
	we have $\alpha<\kappa_{8}(T)$ and 
	\begin{align}
		\tilde{\kappa}_4(T)&= \psi_4\mathcal{L}_{1,0}  \alpha^2/C^2   
		+  (1-2\varepsilon_5)\psi_2\psi_4n\tilde{d}^2s_0^2 \alpha^3T  
		\le \varepsilon_8s_0^2. \nonumber
	\end{align}
	Together with the definition of $s_k$, we know that
	\begin{align} \label{matters}
		s_{k+1}^2 &=  (1-\varepsilon_8)s_k^2 + \tilde{\kappa}_4(T) \nonumber\\
		&\le (1-\varepsilon_8)s_k^2 + \varepsilon_8 s_0^2 \le s_0^2, 
		~\forall k\in[0,T-1].
	\end{align}
	\looseness=-1
Moreover, $s_k$ is nonincreasing since
$s_{k+1}^2-s_k^2=(1-\varepsilon_8)^k\allowbreak\times(s_1^2-s_0^2)\le0$.

			Next, we prove by mathematical induction that the constructed sequence $s_k^2$ satisfies $\max_{i\in[n]}\|x_{i,k} - \hat{x}_{i,k-1}\|^2_p \le C^2s_k^2$.
			For $k=0$, the inequality holds trivially.
			Suppose that the statement holds for $k = 0,1,\ldots,\tau$, namely,
			\begin{align} 
				% \label{assume:induction:k}
				\max\nolimits_{i\in[n]}\|x_{i,k}-\hat{x}_{i,k-1}\|_p^2 \le C^2s_k^2.
				\nonumber
			\end{align}
			% We next show that the inequality also holds for $k = \tau+1$.
			% Similar to the proof of \eqref{sk:iteration:0}, and 
			% \begin{align}
			% 	s_{k+1}^2 \le (1-\varepsilon_8)s_k^2 + \varepsilon_8s_0^2,
			% \end{align}
			% which implies that $s_k^2\le s_0^2$ for all $k\in[0,T]$.
			For $k = \tau+1$, similar to the proof of \eqref{mid:s0}, we have
			\begin{align} \label{bound:sk00}
				&\max\nolimits_{i\in[n]}\|x_{i,\tau+1}-\hat{x}_{i,\tau}\|^2_p 
						\le \big(1-\varepsilon_8\big)C^2s_\tau^2 \nonumber\\
					&\quad	+ \psi_4\mathcal{L}_{1,0}\alpha^2 
						+ (1-2\varepsilon_5)\psi_2\psi_4n\tilde{d}^2C^2s_0^2 \alpha^3T \nonumber\\
				&=C^2\big((1-\varepsilon_8)s_\tau^2 + \tilde{\kappa}_4(T) \big) =C^2s_{\tau+1}^2,
			% 	&s_{k+1}^2\ge \big(1-(\varepsilon_5+2\varepsilon_5^2)+(1-2\varepsilon_5) n\tilde{d}^2\psi_3\alpha^2\big)s_k^2 + \nonumber\\
			% 		&\quad +\psi_4\mathcal{L}_{1,0}\alpha^2
			% 			+ \frac{2\ell\psi_1\psi_4}{\varepsilon_4} \mathcal{L}_{1,0} \alpha^2 
			% 			+ \frac{ \psi_2\psi_4}{\varepsilon_4}n\tilde{d}^2s_0^2 \alpha^2 \nonumber\\
			% 		&\quad
			% 			+ \frac{2\ell\psi_1\psi_4}{\varepsilon_4} n\tilde{d}^2 s_0^2 \alpha^3 T.
			\end{align}
			which completes the induction.

			Finally, by unrolling the recursion of $s_k^2$ in~\eqref{thm1:set}, together with~\eqref{bound:sk00} and $\alpha=1/(n^{\frac{1}{4}}\tilde{d}\sqrt{T})$, we obtain
			\begin{align} 
				&\hspace{-0.8em} \max\nolimits_{i\in[n]}\hspace{-0.1em}\|x_{i,k}-\hat{x}_{i,k-1}\|_p^2/C^2\le s_k^2 \nonumber\\
				&\hspace{-0.8em} \le \hspace{-0.1em} \big(1 \hspace{-0.1em} - \hspace{-0.1em} \varepsilon_8\big)^k s_0^2 \hspace{-0.06em} 
					+ \hspace{-0.1em}  \frac{\psi_4\mathcal{L}_{1,0}}{C^2\varepsilon_8n} n\alpha^2 \hspace{-0.2em}+\hspace{-0.18em} \frac{(1\hspace{-0.1em} -\hspace{-0.1em} 2\varepsilon_5)\psi_2\psi_4}{\varepsilon_8 } n\tilde{d}^2 s_0^2 \alpha^3 T 
						  \label{bound:sk0} \\
				&\hspace{-0.8em} =\big(1 \hspace{-0.1em}-\hspace{-0.1em}\varepsilon_8\big)^k s_0^2 \hspace{-0.06em} +\hspace{-0.1em} \frac{\psi_4\mathcal{L}_{1,0}}{C^2\varepsilon_8n} \frac{n^{\frac{1}{2}}}{\tilde{d}^2T} 
					\hspace{-0.2em}+\hspace{-0.18em} \frac{(1\hspace{-0.1em}-\hspace{-0.1em}2\varepsilon_5)\psi_2\psi_4}{\varepsilon_8 }s_0^2  \frac{n^{\frac{1}{4}}}{\tilde{d}\sqrt{T}},  \nonumber
					% \label{bound:sk:max}
			\end{align}
			which implies \eqref{bound:sk}.
	\end{proof}
	\begin{remark}
		% It should be highlighted that
	Here we specially design the stepsize $\alpha$ and the scaling factor $s_k$ in \eqref{thm1:set} to 
	regulate the compression error in \eqref{sk:iteration:0}, \eqref{matters}, and~\eqref{bound:sk00}.
This keeps the compression input in the local region 
and guarantees an $\mathcal{O}(1/\sqrt{T})$ decay of the compression error via~\eqref{bound:sk} and \eqref{local:contraction}.
Then summing~\eqref{L1:iteration} over $k\in[0,T-1]$ yields the convergence rate in terms of the stationarity measure $\|\bar{\bsg}_{k}^0\|^2$.
% This ensures that the compressed input stays in the local region, 
		% and further establishes that the compression error decays at the rate $\mathcal{O}(1/\sqrt{T})$ using \eqref{local:contraction}.
		% By combining Lemmas~\ref{Lemma: Lyapunov Difference} and~\ref{Lemma:s0}, we can establish the convergence of the Lyapunov function $\mathcal{L}_{1,k}$ and derive the desired convergence guarantees under locally-bounded compressors.
	\end{remark}
	% \begin{proof}
	% 	It should be highlighted that we specially design the stepsize $\alpha$ and the scaling sequence $\{s_k\}$ to control the compression error,
	% 	and use mathematical induction based on Lemma~\ref{Lemma: Lyapunov Difference} to ensure that the transmitted signal always stays in the local region.
	% 	The detailed proof is deferred to Appendix~\ref{appendix:lemma2}.
	% \end{proof} 

% Lemma~\ref{Lemma:s0} ensures that the transmitted signal stays in the local region, 
% and further establishes that the compression error decays at the rate $\mathcal{O}(1/\sqrt{T})$.
% By combining Lemmas~\ref{Lemma: Lyapunov Difference} and~\ref{Lemma:s0}, we can establish the convergence of the Lyapunov function $\mathcal{L}_{1,k}$ and derive the desired convergence guarantees under locally-bounded compressors.

	\subsection{Locally-Bounded Case} \label{subsec:mainre:local}

	This subsection gives the main results under the locally-bounded compressors.
	%  which include 1-bit compressors and saturating quantizers. 	
	We start with the nonconvex setting, where Theorem~\ref{nonconvex:thm-R1} provides the first convergence guarantee for local compressors.
	% , including the locally-bounded class considered here.	% Under locally-bounded compressors, we start with the nonconvex case.

	\begin{theorem}\label{nonconvex:thm-R1}
	Under Definition~\ref{ass:local} and Assumptions~\ref{nonconvex:ass:graph}--\ref{nonconvex:ass:known_lowerbound},
	let $\{x_{i,k}\}$ be generated by Algorithm~\ref{nonconvex:algorithm-pdgd} with the same parameter settings as in~\eqref{thm1:set}.
	Then
	% Suppose Assumptions~\ref{ass:local} and~\ref{nonconvex:ass:graph}--\ref{nonconvex:ass:known_lowerbound} hold.
	% Let $\{x_{i,k}\}$ be generated by Algorithm~\ref{nonconvex:algorithm-pdgd} with the same parameter settings as in~\eqref{thm1:set}.
	% Then
	% 	\begin{align}
	% 		&~~~\alpha=\frac{1}{n^{\frac{1}{4}}\tilde{d}\sqrt{T}},~\beta=\tau_1\gamma,~\gamma>\kappa_2,~\omega\in(0,\frac{1}{r}],\nonumber\\
	% 		&~~\quad \tau_1\ge \kappa_1,~T>\kappa_T,~s_0\ge \max_{i\in[n]}\|x_{i,0}\|_p/C, \nonumber\\
	% 		&~~~~~s_k^2=(1-\varepsilon_8)s_{k-1}^2 + \kappa_s(T),~\forall k\in[1,T], \label{thm1:set}
	% 	\end{align}
	% where $\kappa_1$ and $\kappa_2$ are given in Appendix~\ref{appendix:lemmas}, and $\kappa_T$, $\kappa_s(T)$ and $\varepsilon_8$ are given in Appendix~\ref{appendix:lemma2}.
	% where $\kappa_1$ and $\kappa_2$ are positive constants given in \ref{appendix:lemmas},$~\kappa_T,~\kappa_s(T),~\varepsilon_8$ are positive constants given in Appendix~\ref{appendix:lemma2}.
		\begin{subequations}
		\begin{align}
			&\frac{1}{T}\sum_{k=0}^{T-1}\Big(\|\nabla f(\bar{x}_k)\|^2+\frac{1}{n}\sum_{i=1}^{n}\|x_{i,k}-\bar{x}_k\|^2\Big)=\mathcal{O}(\frac{n^{\tfrac{1}{4}}\tilde{d}}{\sqrt{T}}),\label{nonconvex:thm-sm-equ1}\\
			&f(\bar{x}_{T})-f^\ast=\mathcal{O}(1).\label{nonconvex:thm-sm-equ2}
		\end{align}
		\end{subequations}
		% where $\bar{x}_k=\frac{1}{n}\sum_{i=1}^{n}x_{i,k}$.
	\end{theorem}
	
\begin{proof} 
	See Appendix~\ref{appendix:thm1}.
	% Under the setting of Theorem~\ref{nonconvex:thm-R1}, 
	% all conditions of Lemmas~\ref{Lemma: Lyapunov Difference} and~\ref{Lemma:s0} are satisfied.
	% Lemma~\ref{Lemma:s0} first shows that the compressed input remains within the local region and gives the compression error bound~\eqref{bound:sk}.
	% Substituting~\eqref{bound:sk} into the one-step difference of the Lyapunov function $\mathcal{L}_{1,k}$ in~\eqref{L1:iteration} 
	% and summing over $k\in[0,T-1]$ yields~\eqref{nonconvex:thm-sm-equ1}.
	% The bound~\eqref{nonconvex:thm-sm-equ2} follows similarly from the one-step difference of $e_{4,k}$.
\end{proof}
	% \begin{proof}
	% 	See Appendix~\ref{appendix:thm1}
	% \end{proof}
	Note that the explicit expressions for $\tilde{\kappa}_3$ and $\tilde{\kappa}_4(T)$ 
	% $\kappa_T$ and $\kappa_s(T)$ 
	used in the parameter settings~\eqref{thm1:set}
	depend on the optimality gap $f(\bar{x}_0)-f^\ast$,
	% We note that deriving explicit expressions for $\kappa_T$ and $\kappa_s(T)$ involves the optimality gap $f(\bar{x}_0)-f^\ast$,
	 for which we use a known lower bound $f_{\mathrm{low}}\le f^\ast$ in Assumption~\ref{nonconvex:ass:known_lowerbound}.
	This assumption is mild and holds in many applications as remarked after Assumption~\ref{nonconvex:ass:known_lowerbound}. 
	% For example, most machine learning loss functions are nonnegative and thus satisfy it naturally.
	% It is introduced only for an explicit theoretical parameter choice, while in practice the parameters can be tuned empirically.

\begin{remark}
To the best of our knowledge, this is the first compressed optimization algorithm that establishes convergence guarantees for general local compressors. %in the nonconvex setting.
	Compared with existing 1-bit methods for nonconvex optimization~\cite{zhu2020one,fan20221}, which focus on the centralized setting, 
	our result provides the first extension to the distributed setting and general local compressors, 
	and matches their $\mathcal{O}(1/\sqrt{T})$ convergence rate in terms of $T$-dependence.
Importantly, our extension does not require the 
strong assumptions underlying these centralized convergence results, such as homogeneous problem settings  and symmetric and
unimodal gradient noise~\cite{zhu2020one},
as well as Lipschitz
continuity~\cite{fan20221}.
	Compared with existing results for distributed convex optimization,
% Yi and Hong 
\cite{yi2014quantized} studied 1-bit compressors and established convergence without an explicit rate,
 while
%   Saha et al. 
\cite{saha2021decentralized} considered saturating quantizers, established convergence in probability, and derived a slower theoretical rate of $\mathcal{O}(T^{-(1/2-\theta)})$ for some $\theta\in(0,1/2)$.
% \cite{yi2014quantized} studies 1-bit methods and establishes asymptotic convergence without an explicit rate.
% Regarding saturating quantizers, \cite{saha2021decentralized} studies distributed convex optimization, establishes convergence in probability, and attains a slower theoretical rate $\mathcal{O}(T^{-(1/2-\theta)})$ for some $\theta\in(0,1/2)$.
	% Regarding saturating quantizers, \cite{saha2021decentralized} considers distributed convex problems, proves convergence in probability, and attains a slower order rate of $\mathcal{O}(T^{-(1/2-\theta)})$ for some $\theta\in(0,1/2)$.
	%  whereas most other guarantees rely on strong convexity or the P--L condition \cite{magnusson2020maintaining,xiong2022quantized,xu2024quantized}.
	% 	% Moreover, existing distributed finite-bit compressed methods~(\cite{magnusson2020maintaining,rikos2024distributed}) are analyzed only for strongly convex problems.
\end{remark}

Next, we strengthen the previous result by using one initial uncompressed (exact) communication round, which yields a faster convergence.
% Next, we strengthen the previous result by requiring the first round to use uncompressed (exact) communication, 
% which leads to a faster convergence.
\begin{theorem}\label{nonconvex:cor-R1}
Under Definition~\ref{ass:local} and Assumptions~\ref{nonconvex:ass:graph}--\ref{nonconvex:ass:known_lowerbound},
let $\{x_{i,k}\}$ be generated by Algorithm~\ref{nonconvex:algorithm-pdgd} with the initial inter-agent communication replaced by an uncompressed round, 
and set 
	% Under the same conditions as Theorem~\ref{nonconvex:thm-R1},
	% suppose that the first inter-agent communication in Algorithm~\ref{nonconvex:algorithm-pdgd} is uncompressed,
	% which is equivalent to setting
	% \begin{align}
	% 	\hat{x}_{i,-1} = x_{i,0},~y_{i,-1}=\sum\nolimits_{j=1}^{n} L_{ij} x_{j,0}.  
	% \end{align}
	% Set the algorithm parameters as
	\begin{align} \label{cor1:set}
		&~~~~\alpha=\frac{\tau_0}{n^{\frac{1}{3}}\tilde{d}^{\frac{2}{3}}T^{\frac{1}{3}}},~\beta=\tau_1\gamma,~\gamma>\kappa_2, 
			~\omega\in(0,\frac{1}{r}], \nonumber\\
		&~T>\kappa_3,~s_0^2=\tau_4n\alpha^2,~\tau_0\le\kappa_0,~\tau_1\ge\kappa_1,\tau_4\ge \kappa_4, \nonumber\\
		&~~~~~~~s_k^2=(1-\varepsilon_8)s_{k-1}^2 + \tilde{\kappa}_4(T),~\forall k\in[1,T],
		% &~~~~\alpha=\frac{\tau_\alpha}{n^{\frac{1}{3}}\tilde{d}^{\frac{2}{3}}T^{\frac{1}{3}}},~\beta=\tau_1\gamma,~\gamma>\kappa_2, 
		% 	~\omega\in(0,\frac{1}{r}], \nonumber\\
		% &~T>\tilde{\kappa}_T,~s_0^2=\tau_sn\alpha^2,~\tau_1\ge\kappa_1,~\tau_\alpha\le\kappa_\alpha,\tau_s\ge \tilde{\kappa}_s, \nonumber\\
		% &~~~~~~~s_k^2=(1-\varepsilon_8)s_{k-1}^2 + \kappa_s(T),~\forall k\in[1,T],
	\end{align}
	where $\kappa_0$, $\kappa_3$, and $\kappa_4$
	% $\kappa_\alpha$, $\tilde{\kappa}_s$ and $\tilde{\kappa}_T$ 
	are some positive constants.
	Then
	\begin{subequations}\label{nonconvex:cor-sm-eq}
	\begin{align}
		&\frac{1}{T}\sum_{k=0}^{T-1}\Big(\|\nabla f(\bar{x}_k)\|^2+\frac{1}{n}\sum_{i=1}^{n}\|x_{i,k}-\bar{x}_k\|^2\Big) = \mathcal{O}\Big(\frac{n^{\frac{1}{3}}\tilde{d}^\frac{2}{3}}{T^{\frac{2}{3}}}\Big), \label{nonconvex:cor-sm-eq1}\\
		&f(\bar{x}_{T})-f^\ast = \mathcal{O}(1). \label{nonconvex:cor-sm-eq2}
	\end{align}
	\end{subequations}
	\end{theorem}
	\begin{proof}
		See Appendix \ref{appendix:cor1}.
		% The locally-bounded class is sensitive to the first-round compression error.
		% As discussed in the previous subsection, this error accumulates through~\eqref{accumulate} and~\eqref{c:iteration}, 
		% and dominates the final bound as reflected in the $s_0$-dependent last term on the right-hand side of~\eqref{bound:sk}.
		% % Thus, using one initial uncompressed communication round to drive this error to zero can improve the convergence rate.
		% 	The initial uncompressed communication round drives this error to zero and gives the refined bound~\eqref{bound:sk}.
		% The rest follows as in Theorem~\ref{nonconvex:thm-R1}.
	\end{proof}
	% 	\begin{proof}
	% 	See Appendix~\ref{appendix:cor1}.
	% \end{proof}

		Note that deriving explicit expressions for $\tilde{\kappa}_4(T)$ and $\kappa_4$
		% $\kappa_s(T)$ and $\tilde{\kappa}_s$ 
		involves the optimality gap $f(\bar{x}_0)-f^\ast$, 
		for which we use a known lower bound $f_{\mathrm{low}}\le f^\ast$ in Assumption~\ref{nonconvex:ass:known_lowerbound}.
		\begin{remark}
			The locally-bounded class is sensitive to the first-round compression error.
			As discussed in the previous subsection, this error accumulates through~\eqref{accumulate} and~\eqref{c:iteration}, 
			and dominates the final bound as reflected in the $s_0$-dependent last term on the right-hand side of~\eqref{bound:sk}.
			Thus, using one initial uncompressed communication round to drive this error to zero can improve the convergence rate.
			This initialization can be avoided simply by initializing all agents with the same vector, i.e., $x_{i,0}=x_0$ for all $i\in[n]$, where $x_0$ is arbitrary in $\mathbb{R}^d$.
		\end{remark}
	\begin{remark}
		% The locally-bounded setting is sensitive to the first-round compression error, which can propagate through the subsequent iterations and dominate the final bound. Thus, an uncompressed first round that drives this error close to zero can improve the convergence rate.
		% To the best of our knowledge, the
		The obtained $\mathcal{O}(1/T^{2/3})$ convergence rate is the fastest provably available 
		for 
		% distributed optimization under 
		local compressors.
		% and this is also the first to consider acceleration via an uncompressed communication round at the first iteration.
		This rate exhibits an order-wise improvement over the existing literature from $\mathcal{O}(1/\sqrt{T})$ to $\mathcal{O}(1/T^{2/3})$ in its dependence on $T$, 
		 and is even theoretically faster than the centralized results in the 1-bit special case \cite{zhu2020one,fan20221}.
		% This initialization can be avoided simply by initializing all agents with the same vector, i.e., $x_{i,0}=x_0$ for all $i\in[n]$, where $x_0$ is arbitrary in $\mathbb{R}^d$.
		%  albeit at the cost of one initial uncompressed communication round.
		% The uncompressed communication in the first round can be avoided simply by coordinate initialization, i.e., $x_{i,0}=x_0, \forall i\in[n]$, where $x_0$ is an arbitrary vector in $\mathbb{R}^d$.
	\end{remark}
% 	\begin{remark}
% 		From the parameter setting, we have 
% 		$s_0 = \mathcal{O}\!\big(\tfrac{1}{\tilde{d}^2}\big)$ 
% 		and, when $p=\infty$, equivalently 
% 		$s_0 = \mathcal{O}\!\big(\tfrac{1}{d}\big)$. 
% For the expression $\mathcal{C}(0/s_0)$, 
% it suffices that the smallest positive number 
% distinguishable from zero be smaller than $s_0 = \mathcal{O}(1/d)$,
% which is easily satisfied with standard 64-bit floating-point precision.
% 		In IEEE~754 double precision, 
% 		the machine epsilon $\varepsilon_{\mathrm{mach}} = 2^{-52} \!\approx\! 2.22\times10^{-16}$ 
% 		bounds the \emph{relative} rounding error of floating-point operations:
% 		$|\hat{x}-x| \le \varepsilon_{\mathrm{mach}} |x|$, 
% 		where $x$ is the exact real value and $\hat{x}$ its floating-point representation.  
% 		In addition, the smallest positive representable number is 
% 		$x_{\min} = 2^{-1074} \!\approx\! 4.94\times10^{-324}$, 
% 		which is many orders of magnitude below $1/d$ 
% 		even for large-scale models with $d = 10^9$--$10^{12}$.
% 	Hence, both the representability and relative-accuracy conditions 
% 	for $\mathcal{C}(0/s_0)$ are easily satisfied, 
% 	ensuring that the proposed algorithm is numerically valid. 
% 	\end{remark}

	When the global cost function further satisfies the P--L condition, linear convergence follows.

  \begin{theorem}\label{nonconvex:thm-liner}
    Under Definition~\ref{ass:local} and Assumptions~\ref{nonconvex:ass:graph}--\ref{nonconvex:ass:optset} and \ref{nonconvex:ass:fil}, 
    let $\{x_{i,k}\}$ be generated by Algorithm~\ref{nonconvex:algorithm-pdgd} with
 \begin{align}
        &~\hspace{0.3em}\alpha\in(0,\kappa_0^\prime),~\beta=\tau_1\gamma,~\gamma>\kappa_2,~\omega\in(0,1/r],~\tau_1>\kappa_1, \nonumber\\
		&s_0\ge\max\Big\{\sqrt{\frac{\kappa_\nu}{n\tilde{d}^2\psi_5C^2}},~\frac{\max_{i\in[n]}\|x_{i,0}\|_p}{C}\Big\},~s_k=s_0\epsilon^k,
		% &~\hspace{0.1em}\alpha\in(0,\tilde{\kappa}_0^\prime),~\beta=\tau_1\gamma,~\gamma>\kappa_2,~\omega\in(0,1/r],~\tau_1>\kappa_1, \nonumber\\
		% &s_0\ge\max\{\sqrt{\frac{\kappa_\nu}{n\tilde{d}^2\psi_5C^2}},~\frac{\max_{i\in[n]}\|x_{i,0}\|_p}{C}\},~s_k=s_0\epsilon^k,
    \end{align}
  where $\kappa_0^\prime$,~$\psi_5$,~$\kappa_\nu$,  and $\epsilon\in(0,1)$
%   $\tilde{\kappa}_0^\prime$ ,$\kappa_\nu$, $\psi_5$ and $\epsilon\in(0,1)$ 
  are some positive constants. 
  Then  
    \begin{align}\label{nonconvex:thm-ft-equ1_determin}
      f(\bar{x}_k)-f^\ast + \frac{1}{n}\sum\nolimits_{i=1}^{n}\|x_{i,k}-\bar{x}_k\|^2
      =\mathcal{O}(\epsilon^{2k}).
    \end{align}
  \end{theorem}
%   \begin{proof}
% 		See Appendix~\ref{appendix:thm2}. 
%   \end{proof}

  	\begin{proof}
		See Appendix \ref{appendix:thm2}.
		% The proof follows from Lemma~\ref{Lemma: PL} by induction, 
		% with details provided in the supplemental material.
	\end{proof}

	Note that Assumption~\ref{nonconvex:ass:known_lowerbound} is not needed under the P--L condition,
	since the optimality gap $f(\bar{x}_0)-f^\ast$ can be bounded by $\|\nabla f(\bar{x}_0)\|^2$ via~\eqref{nonconvex:equ:plc}, 
	making the choice of $s_0$ computable.
	% \begin{remark}
	% 	In the case where the P--{\L} constant $\nu$ is known in advance, 
	% 	Assumption~\ref{nonconvex:ass:known_lowerbound} becomes unnecessary. 
	% 	Choosing 
	% 	$s_0 \ge \max\{\sqrt{\frac{\kappa_4(\nu)}{n\tilde{d}^2\psi_5C^2}},~\max_{i\in[n]}\|x_{i,0}\|_p/C\}$ 
	% 	yields the same $\mathcal{O}(\epsilon^{2k})$ linear convergence rate as in 
	% 	Theorem~\ref{nonconvex:thm-liner}, 
	% 	where $\kappa_4(\nu)$ is a positive constant depending on $\nu$ 
	% 	as defined in Appendix~\ref{appendix:thm2}.
	% \end{remark}
	\begin{remark}
	Under the P--L condition, we achieve linear convergence as in the uncompressed counterpart.
	%  while allowing the locally-bounded compressors.
% Notably, under the P--L condition, our linear convergence rate matches the uncompressed counterpart, while using highly compressed communication via locally-bounded compressors.		
% It is also comparable to existing linearly convergent compressed distributed optimization methods, while being established under a more general compression framework.
Compared with~\cite{zhang2023innovation,Yi_CommunicationCompression_2023}, which also established  linear convergence, the compressor therein is a special case of ours with $C=1$ and $r=1$.
Moreover, 
% Zhang et al. 
\cite{zhang2023innovation} assumed strong convexity, which is stronger than the P--L condition used here.
Linear convergence for 1-bit compressors and saturating quantizers is also available under strong convexity or the P--L condition~\cite{magnusson2020maintaining,xu2024quantized}, 
 and these compressors are covered by our locally-bounded compressors.
% In particular, the compressor adopted in~\cite{Yi_CommunicationCompression_2023} is covered as a special case of our model by setting $C=1$ and $r=1$.
		% Moreover,~\cite{zhang2023innovation} can be viewed as another special case; however, its linear convergence is derived under strong convexity, which is stronger than the P--L condition.
% For finite-bit compressors,~\cite{magnusson2020maintaining,rikos2024distributed} also assume strong convexity to establish linear convergence.
% In addition,~\cite{rikos2024distributed} does not guarantee exact convergence.

% We note that several other compressors with global compression properties can also achieve linear convergence, which are covered by our globally-bounded compression model and will be discussed later.
	\end{remark}

\subsection{Contraction Guarantee} \label{section:proof:global}

This subsection analyzes the convergence under the globally-bounded compressors,
which incurs both relative and absolute compression errors.
The key step is to incorporate the compression error
$e_{5,k}=\|\bsx_k-\hat{\bsx}_k\|^2$
into the Lyapunov function $\mathcal{L}_{2,k}$,
establish their contraction properties to handle relative compression error,
and control absolute error through the diminishing scaling factor $s_k$.
We first bound the one-step difference of $e_{5,k}$ in Lemma~\ref{lemma:5} and that of $\mathcal{L}_{2,k}$ in Lemma~\ref{lemma:L2}.

	\begin{lemma} \label{lemma:5}	
		Under Definition~\ref{ass:global} and Assumptions~\ref{nonconvex:ass:graph} and \ref{nonconvex:ass:fiu},
		%  take $\alpha<\hat{\kappa}_0$, $\beta=\tau_1\gamma$, $\gamma>\kappa_2$, $\omega\in(0,1/r]$, and $\tau_1\ge \kappa_1$.
		let $\{x_{i,k}\}$ be generated by Algorithm~\ref{nonconvex:algorithm-pdgd} with $\omega\in(0,1/r]$.
		Then
		\begin{align}
			&\mathbb{E}_{\mathcal{C}}[\|\bsx_{k+1}-\hat{\bsx}_{k+1}\|^2]\le\mathbb{E}_{\mathcal{C}}\Big[(1-\varepsilon_{11})\|\bsx_{k}-\hat{\bsx}_{k}\|^2 \nonumber\\
			&\quad + n\omega r C s_k^2 
				+\|\bsx_k\|^2_{4\varepsilon_{10}\alpha^2(\beta^2\rho^2(L)+\ell^2)\bsE}\nonumber\\
			&\quad+\Big\|\bm{v}_k+\frac{1}{\gamma}\bsg_k^0\Big\|^2_{4\varepsilon_{10}\alpha^2\gamma^2\rho(L)\bsF}\Big].
			\label{nonconvex:xminush}
		\end{align}
		% where $\varepsilon_9$  and $\varepsilon_{11}$ is given in Appendix~\ref{appendix:lemma3}.
	\end{lemma}

	\begin{proof}
		
	% We denote the following constant: 
	% \begin{align*}
	% 	&\varepsilon_9=\omega r\delta/2,\\
	% 	&\varepsilon_{10}=(1-2\varepsilon_9)(1+\varepsilon_9^{-1}),\\
	% 	&\varepsilon_{11}=\varepsilon_9 + 2\varepsilon_9^2 - 4\varepsilon_{10}\alpha^2\beta^2\rho^2(L).
	% \end{align*}	
		Denote $\varepsilon_9=\omega r\delta/2$.
		% To analyze the one-step difference of $\mathcal{L}_{2,k}$, we study its additional component term $\|\bsx_{k}-\hat{\bsx}_{k}\|^2$.	
		From \eqref{global:contraction}, the Cauchy--Schwarz inequality gives
			\begin{align}
			&\mathbb{E}_{\mathcal{C}}[\|\bsx_{k+1}-\hat{\bsx}_{k+1}\|^2] 
				=\sum\nolimits_{i=1}^{n} \mathbb{E}_{\mathcal{C}}[\|x_{i,k+1}-\hat{x}_{i,k+1}\|^2] \nonumber\\
			&\le (1-2\varepsilon_9)\mathbb{E}_{\mathcal{C}}[\|\bsx_{k+1}-\hat{\bsx}_{k}\|^2] + n\omega rCs_k^2 \nonumber\\
			&= (1-2\varepsilon_9)\mathbb{E}_{\mathcal{C}}[\|\bsx_{k+1}-\bsx_{k}+\bsx_{k}-\hat{\bsx}_{k}\|^2] + n\omega rCs_k^2 \nonumber\\
			&\le \varepsilon_{10}\mathbb{E}_{\mathcal{C}}[\|\bsx_{k+1}-\bsx_{k}\|^2] \nonumber\\
				& \quad + (1-\varepsilon_9-2\varepsilon_9^2)\mathbb{E}_{\mathcal{C}}[\|\bsx_{k}-\hat{\bsx}_{k}\|^2] +  n\omega rCs_k^2.
				\label{nonconvex:xminush_compress}
		\end{align}

		For the first term on the right-hand side of \eqref{nonconvex:xminush_compress}, we have
		\begin{align}\label{nonconvex:xkoneminusx}
			&\|\bsx_{k+1}-\bsx_k\|^2 = \alpha^2\|\beta\bsL\hat{\bsx}_k+\gamma\bm{v}_k+\bsg_k\|^2\nonumber\\
			&=\alpha^2\|\beta\bsL(\hat{\bsx}_k-\bsx_k)+\beta\bsL\bsx_k+\gamma\bm{v}_k
			+\bsg_k^0+\bsg_k-\bsg_k^0\|^2\nonumber\\
			&\le4\alpha^2(\beta^2\|\bsx_k-\hat{\bsx}_k\|^2_{\bsL^2}+\beta^2\|\bsx_k\|^2_{\bsL^2} \nonumber\\
				&\quad + \|\gamma\bm{v}_k +\bsg_k^0\|^2+\|\bsg_k-\bsg_k^0\|^2)\nonumber\\
			&\le4\alpha^2\Big(\beta^2\rho^2(L)\|\bsx_k-\hat{\bsx}_k\|^2+\|\bsx_k\|^2_{(\beta^2\rho^2(L)+\ell^2)\bsE} + \nonumber\\
				&\quad \Big\|\bm{v}_k+\frac{1}{\gamma}\bsg_k^0\Big\|^2_{\gamma^2\rho(L)\bsF}\Big),
		\end{align}
		where the first equality holds due to \eqref{nonconvex:kia-algo-dc-x} and $y_{i,k}=\sum_{j=1}^{n}L_{ij}\hat{x}_{j,k}$; the first inequality holds due to the Cauchy--Schwarz inequality; 
		and the second inequality holds due to the matrix inequalities 
		$L\le\rho(L)E$ and $\rho^{-1}(L){\bf I}_n\le F$ in~\cite{Yi_CommunicationCompression_2023}, 
		together with Assumption~\ref{nonconvex:ass:fiu}.

		Then substituting \eqref{nonconvex:xkoneminusx} into \eqref{nonconvex:xminush_compress} yields \eqref{nonconvex:xminush}.
	\end{proof}

	% \begin{proof}
	% 	See Appendix~\ref{appendix:lemma3}.
	% 	% The claim follows directly from~\eqref{global:contraction}.
	% 	% % Under globally-bounded compressors, the compression error can be bounded directly using~\eqref{global:contraction}.
	% 	% The detailed proof is deferred to Appendix~\ref{appendix:lemma3}.
	% \end{proof}

	\begin{remark}
		The coefficient $1-\varepsilon_{11}$ strictly smaller than $1$ 
		% on the right-hand side of~\eqref{nonconvex:xminush} 
		is key to driving the compression error $e_{5,k}$ to zero.
		% The coefficient $1-\varepsilon_{11}$ strictly smaller than $1$ on the right-hand side of~\eqref{nonconvex:xminush} is key to driving the compression error $e_{5,k}$ to zero.
		% Moreover, the absolute compression error of the globally-bounded compressors appears in~\eqref{nonconvex:xminush} through the term $n\omega r C s_k^2$ and can be controlled by the diminishing scaling factor $s_k$.
		Compared with previous work~\cite{Yi_CommunicationCompression_2023}, which considered global compressors with only relative compression error, i.e., the special case $C=0$ in Definition~\ref{ass:global}, we additionally handle absolute compression error via a diminishing scaling factor $s_k$.
		Specifically, the absolute compression error of the globally-bounded compressors appears in~\eqref{nonconvex:xminush} through the term $n\omega r C s_k^2$ and is controlled by the decay of $s_k$.
	\end{remark}

		Combining Lemmas~\ref{lemma:5} and~\ref{Lemma: Lyapunov Difference} yields
		the one-step difference of $\mathcal{L}_{2,k}$.
	\begin{lemma} \label{lemma:L2}	
		Under Definition~\ref{ass:global} and Assumptions~\ref{nonconvex:ass:graph}--\ref{nonconvex:ass:optset},
		take $\alpha\in(0,\hat{\kappa}^{\prime}_0)$, $\beta=\tau_1\gamma$, $~\gamma>\kappa_2$, $\omega\in(0,1/r]$, and $\tau_1\ge \kappa_1$,
		where $\hat{\kappa}^{\prime}_0$ is a constant.
		% where $\hat{\kappa}_0$ is a positive constant.
		% where $\hat{\kappa}_0$ is a positive constant given in Appendix~\ref{appendix:lemma3}.
		Let $\{x_{i,k}\}$ be generated by Algorithm~\ref{nonconvex:algorithm-pdgd}.
		Then 		
		\begin{align}
			&\mathbb{E}_{\mathcal{C}}[\mathcal{L}_{2,k+1}]\le \mathbb{E}_{\mathcal{C}}\Big[\mathcal{L}_{2,k} 
			-\alpha\varepsilon_3^{\prime}\big(\|\bsx_{k}\|^2_{\bsE} + \big\|\bsv_{k}+\tfrac{1}{\gamma}\bsg_{k}^0\big\|^2_{\bsF}\big) \nonumber\\
			% &\quad - \frac{\alpha}{4}\|\bar{\bsg}_{k}^0\|^2 - (2\varepsilon_{12} - \alpha\varphi_7- \alpha^2\varphi_8^{\prime})\|\bsx_{k}-\hat{\bsx}_{k}\|^2 \nonumber\\
			% &\quad + n\omega r C s_k^2 \Big].\label{L2:iteration}
			&\hspace{-0.1em}-\hspace{-0.1em} \frac{\alpha}{4}\|\bar{\bsg}_{k}^0\|^2 \hspace{-0.1em}-\hspace{-0.1em} (2\varepsilon_{12} \hspace{-0.1em}-\hspace{-0.1em} \alpha\varphi_7\hspace{-0.1em}-\hspace{-0.1em} \alpha^2\varphi_8^{\prime})\|\bsx_{k}\hspace{-0.1em}-\hspace{-0.1em}\hat{\bsx}_{k}\|^2
			\hspace{-0.1em}+\hspace{-0.1em} n\omega r C s_k^2 \Big]\hspace{-0.1em}.\label{L2:iteration}
		\end{align}
		% where $\varphi_7$ is given in Appendix~\ref{appendix:lemmas} and $\varepsilon_3^{\prime}$ and $\varphi_8^{\prime}$ are given in Appendix~\ref{appendix:lemma4}.
	\end{lemma}
	\begin{proof}
		See Appendix~\ref{appendix:lemma4}.
		% The detailed proof is provided in the supplemental material.
	\end{proof}
	% \begin{proof}
	% 	See Appendix~\ref{appendix:lemma4}.
	% 	% Combining Lemmas~\ref{lemma:5} and~\ref{Lemma: Lyapunov Difference} directly yields the result.
	% 	% % Based on Lemma~\ref{lemma:5}, the proof follows by analogy with~\eqref{L1:iteration}.
	% 	% The detailed proof is deferred to Appendix~\ref{appendix:lemma4}.
	% \end{proof}

% Lemma~\ref{lemma:L2} provides a bound on the one-step difference of the Lyapunov function $\mathcal{L}_{2,k}$, showing a clear descent since the absolute error can be controlled by the diminishing sequence $\{s_k\}$.
	Lemma~\ref{lemma:L2} bounds the one-step difference of the Lyapunov function $\mathcal{L}_{2,k}$.
	By choosing a proper stepsize $\alpha$ so that all quadratic terms in~\eqref{L2:iteration} have negative coefficients, and using an exponentially diminishing $s_k$ to control the absolute compression error,
	$\mathcal{L}_{2,k}$ satisfies a descent inequality up to a vanishing term and hence converges.
	This further yields the desired convergence guarantees of the algorithm under the globally-bounded compressors.

%   \subsection{\hspace{-0.23em}Convergence under Globally-Bounded Compressors}
  \subsection{Globally-Bounded Case} \label{subsec:mainre:global}
	This subsection gives the main results under the globally-bounded compressors.
		% which accommodate both relative and absolute compression errors
		% as well as randomness and arbitrary bounded noise.
		% have been more extensively studied and used, and they
		% have distinctive advantages of accommodating randomness and arbitrary bounded noise compared with the locally-bounded compressors.
		% whereas they require a global compression error bound.
% 	Global compressors have been more extensively studied and used, 
% 	where one typically assumes either relative or absolute compression error.
% 	By contrast, we consider a class of globally-bounded compressors that unifies these two cases.
% 	% Compared with locally-bounded compressors, the globally-bounded model is slightly stronger, yet it has distinctive advantages in that it can accommodate randomness and arbitrary bounded noise.
% but they require a global compression error bound.
% This class can yield a modest improvement in convergence behavior.
% This enhanced model leads to a moderately improved convergence behavior compared with the locally-bounded case.
	We start with the nonconvex setting.
	% , where the resulting convergence behavior is modestly improved compared with the locally-bounded case.
  	% Under globally-bounded compressors, we start with the nonconvex case.

  \begin{theorem}\label{nonconvex:thm-sm_quantization}
    Under Definition~\ref{ass:global} and Assumptions~\ref{nonconvex:ass:graph}--\ref{nonconvex:ass:optset},
    let $\{x_{i,k}\}$ be generated by Algorithm~\ref{nonconvex:algorithm-pdgd} with
    \begin{align}
          &\quad~~~~~\alpha\in(0,\hat{\kappa}^{\prime}_0),~\beta=\tau_1\gamma,~\gamma>\kappa_2,~\omega\in(0,1/r],\nonumber\\
		&s_0\ge \max_{i\in[n]}\|x_{i,0}\|_p,~s_k=s_0\varepsilon^k,~\tau_1\ge\kappa_1,~\varepsilon\in(0,1).
		% &\quad~~~~~\alpha\in(0,\hat{\kappa}_0),~\beta=\tau_1\gamma,~\gamma>\kappa_2,~\omega\in(0,1/r],\nonumber\\
		% &s_0\ge \max_{i\in[n]}\|x_{i,0}\|_p^2,~s_k=s_0\varepsilon^k,~\tau_1\ge\kappa_1,~\varepsilon\in(0,1).
    \end{align}	
    % where $\hat{\kappa}_0$ is given in Appendix~\ref{appendix:lemma4}. 
	Then 
    \begin{subequations}
      \begin{align}
        &\frac{1}{T}\sum_{k=0}^{T-1}\mathbb{E}_{\mathcal{C}}\Big[\|\nabla f(\bar{x}_k)\|^2+\frac{1}{n}\sum_{i=1}^{n}\|x_{i,k}-\bar{x}_k\|^2\Big]=\mathcal{O}\Big(\frac{1}{T}\Big),\label{nonconvex:thm-sm-equ1_quantization}\\
        &\mathbb{E}_{\mathcal{C}}[f(\bar{x}_{T})-f^\ast]=\mathcal{O}(1).\label{nonconvex:thm-sm-equ2_quantization}
      \end{align}
    \end{subequations}
  \end{theorem}
  \begin{proof}
	See Appendix~\ref{appendix:thm3}.
  \end{proof}
	\begin{remark}
% Notably,
The obtained $\mathcal{O}(1/T)$ convergence rate 
 matches that of the uncompressed counterpart in distributed nonconvex optimization.
Moreover, our algorithm is more communication- and memory-efficient than existing methods that simultaneously consider relative and absolute compression errors.
	Specifically,
% Liao et al. 
\cite{liao2024robust} maintained five $d$-dimensional auxiliary variables and used two temporary compressed variables for transmission, whereas our method only stores three additional variables $v_{i,k}$, $\hat{x}_{i,k}$, and $y_{i,k}$, and transmits a single temporary compressed variable $q_{i,k}$.
In addition, 
% Michelusi et al. 
\cite{michelusi2022finite} assumed contractive relative and exponentially diminishing absolute compression error,
while 
% Nassif et al. 
\cite{10976577} required strong convexity and guaranteed convergence to a neighborhood of the optimum.
	\end{remark}

  	When the global cost function further satisfies the P--L condition, linear convergence follows.

  \begin{theorem}\label{nonconvex:thm-ft_quantization}
	Under Definition~\ref{ass:global} and Assumptions~\ref{nonconvex:ass:graph}--\ref{nonconvex:ass:optset} and \ref{nonconvex:ass:fil},  
	let $\{x_{i,k}\}$ be generated by Algorithm~\ref{nonconvex:algorithm-pdgd} 
	with the same parameters as in Theorem~\ref{nonconvex:thm-sm_quantization}.  
	Then, with $\hat{\varepsilon}\in(0,1)$,
    \begin{align}\label{nonconvex:thm-ft-equ1_quantization}
      \mathbb{E}_{\mathcal{C}}\Big[f(\bar{x}_k)\hspace{-0.1em}-\hspace{-0.1em}f^\ast \hspace{-0.1em}+\hspace{-0.1em} \frac{1}{n}\sum\nolimits_{i=1}^{n}\|x_{i,k}\hspace{-0.1em}-\hspace{-0.1em}\bar{x}_k\|^2\Big]
      =\mathcal{O}(\hat{\varepsilon}^{k}).
    \end{align}
  \end{theorem}
  \begin{proof}
		See Appendix~\ref{appendix:thm4}.
		% The proof is similar to those of Theorems~\ref{nonconvex:thm-liner} and~\ref{nonconvex:thm-sm_quantization},
		% with details provided in the supplemental material.
  \end{proof}
  \begin{remark}
		Under the P--L condition, we achieve linear convergence as in the uncompressed counterpart.
		Similar remarks after Theorem~\ref{nonconvex:thm-sm_quantization} also apply to Theorem~\ref{nonconvex:thm-ft_quantization}, since our method is more communication- and memory-efficient than existing approaches.
  \end{remark}

\section{Simulations}\label{zerosg:sec-simulation}

In this section, we evaluate the proposed unified compression algorithm through three experiments: logistic regression, language modeling, and image classification. 
All experiments use full local gradients computed over each agent's entire local dataset, consistent with our theoretical setting.
These experiments demonstrate that the proposed algorithm is applicable to various compressors and improves communication efficiency.

Since no existing method handles the locally-bounded compressors in distributed nonconvex settings, 
whereas RCPP algorithm in \cite{liao2024robust} admits the globally-bounded compressors, 
we compare our algorithm with RCPP as well as our uncompressed counterpart,
and evaluate the proposed algorithm under various compressors.
Specifically, $\mathcal{C}_1$ and $\mathcal{C}_2$ satisfy Definition~\ref{ass:local} with $C=100$, 
	and are tested with and without one initial uncompressed communication round.
Moreover, $\mathcal{C}_3$ and $\mathcal{C}_5$ satisfy Definition~\ref{ass:global}, 
	and are tested in both the noise-free and bounded-noise cases, where the noise bound is set to $100$.
For the remaining compressor parameters, we set $\Delta=1$ for $\mathcal{C}_2$, 
choose $k$ as $10\%$--$30\%$ of the model dimension for $\mathcal{C}_3$, 
and use $4$ bits for $\mathcal{C}_5$.
In the plots, the above configurations are labeled as
``1-bit'', ``1-bit-ini'',
``sat'', ``sat-ini'',
``top-k'', ``top-k-n'',
``4-bit'', and ``4-bit-n''.
For RCPP, we prepend ``R-'' to the corresponding label,
 and our uncompressed counterpart is labeled as ``uncomp''.
% Agents communicate over an Erd\H{o}s--R\'enyi graph with connection probability $0.4$.
All experiments are conducted over a network of $n=10$ agents connected by an Erd\H{o}s--R\'enyi graph with connection probability $0.4$.

\begin{figure*}[!t]
\centering
\captionsetup{skip=2pt}
\captionsetup[subfloat]{farskip=0pt}
\subfloat[Optimality gap versus iterations.]{
  \includegraphics[width=.47\textwidth,trim=2mm 2mm 2mm 2mm,clip]{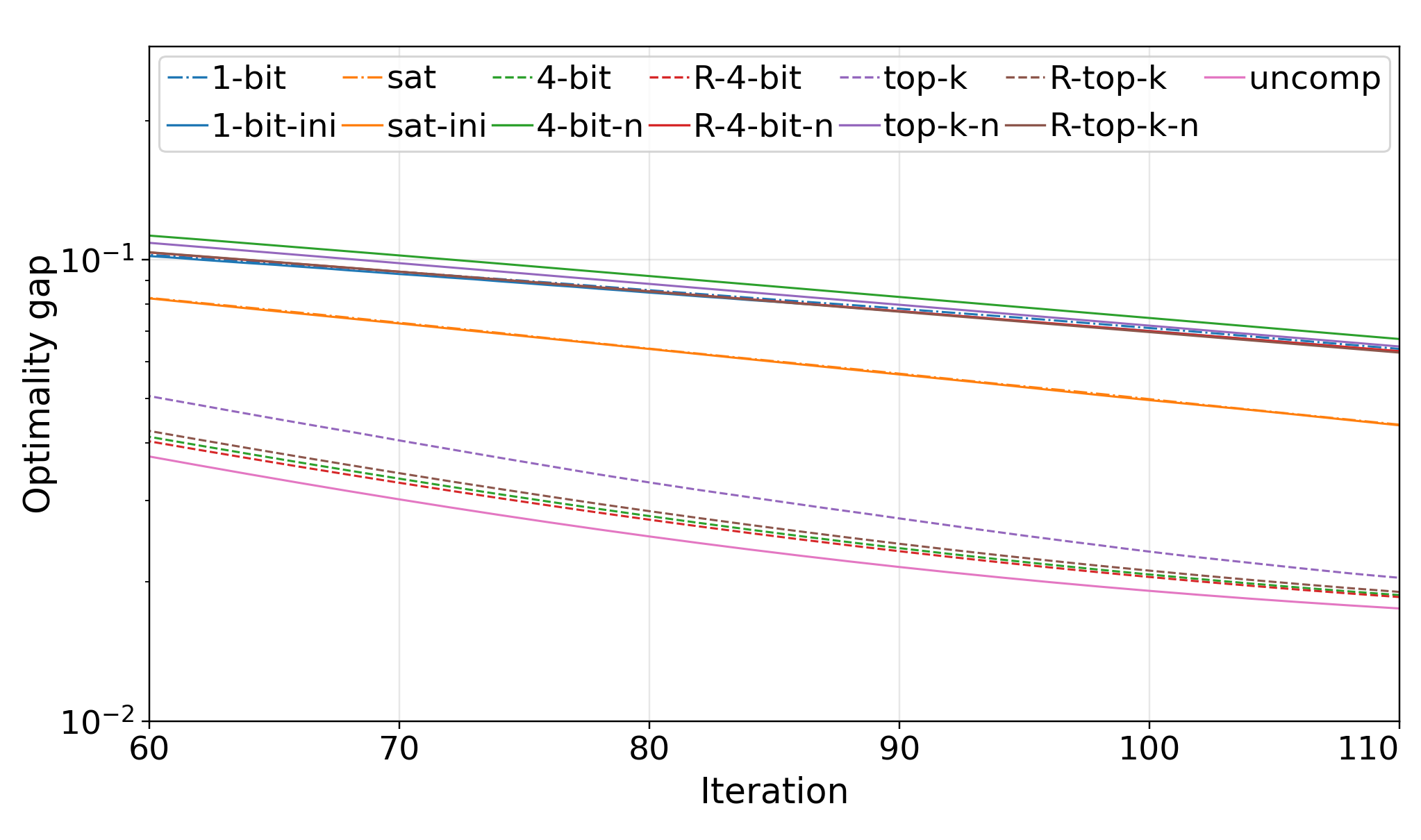}
  \label{logistic:iter}
}%
\hspace{0.03\textwidth}%
\subfloat[
Optimality gap versus communication bits per edge.]{
  \includegraphics[width=.47\textwidth,trim=2mm 2mm 2mm 2mm,clip]{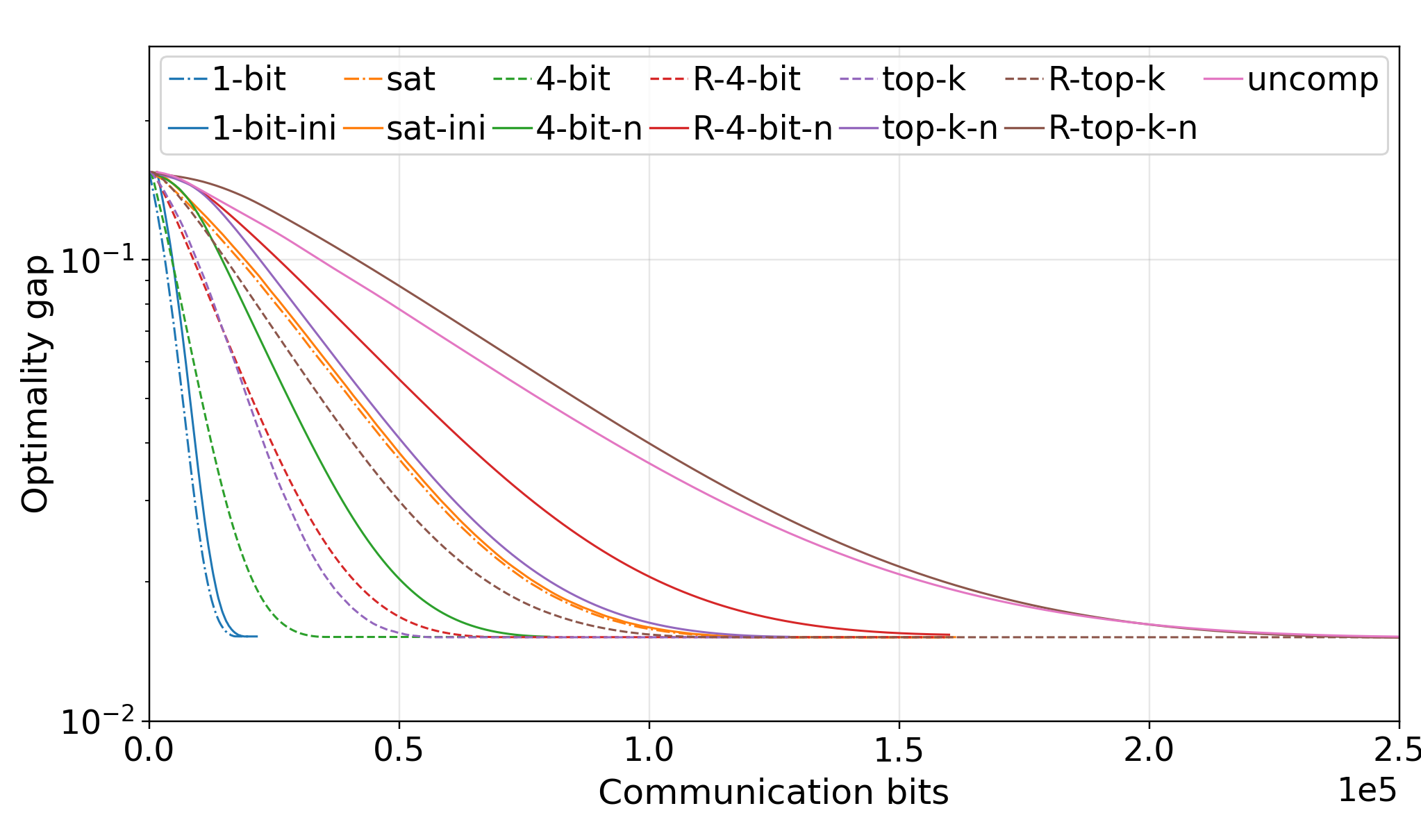}
  \label{logistic:bit}
} 
\caption{Optimality gap in logistic regression.}
\label{logistic}
\par\vspace{4pt}
\subfloat[PPL versus iterations.]{
  \includegraphics[width=.47\textwidth,trim=2mm 2mm 2mm 2mm,clip]{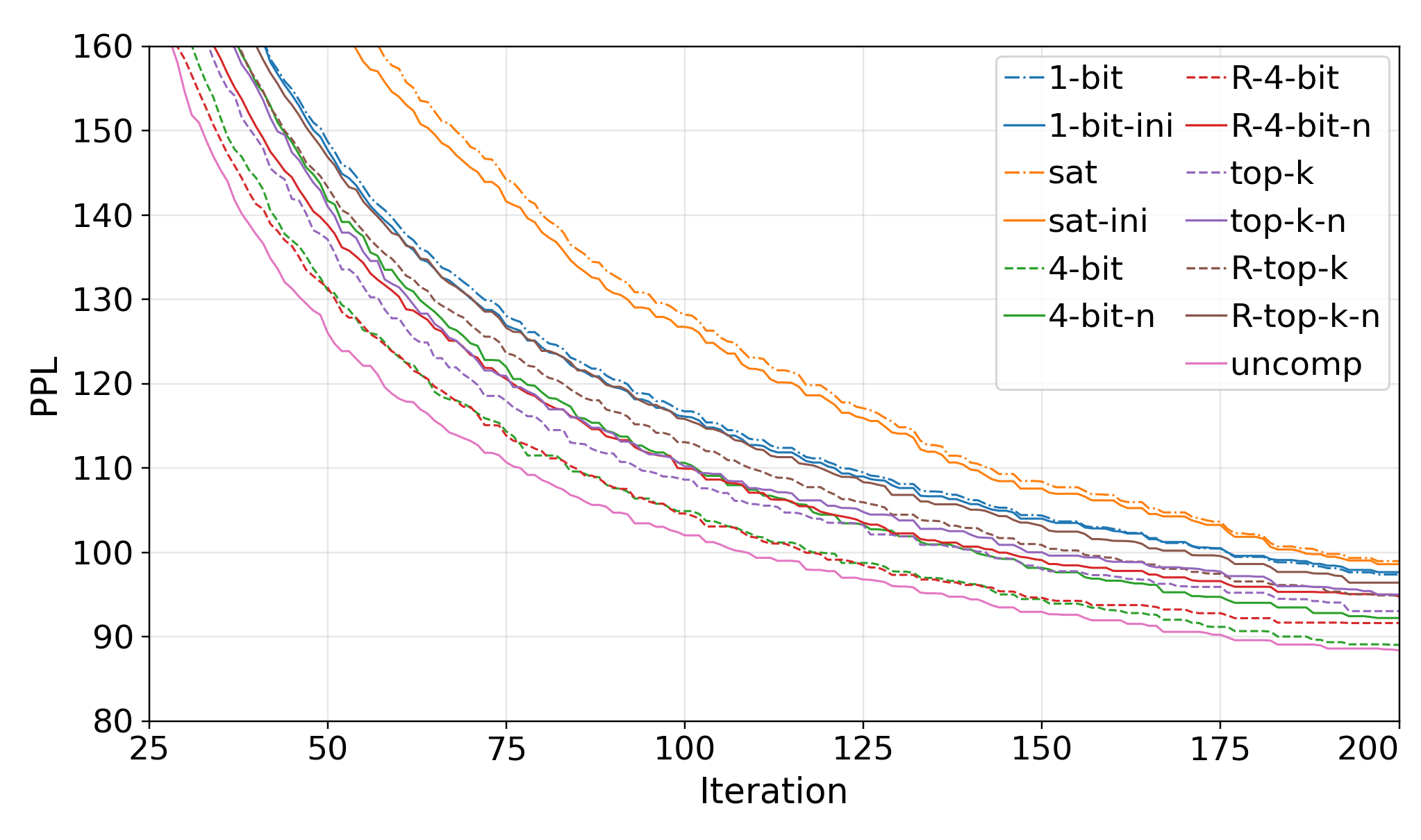}
  \label{ppl:iter}
}%
\hspace{0.03\textwidth}%
\subfloat[PPL versus communication bits per edge.]{
  \includegraphics[width=.47\textwidth,trim=2mm 2mm 2mm 2mm,clip]{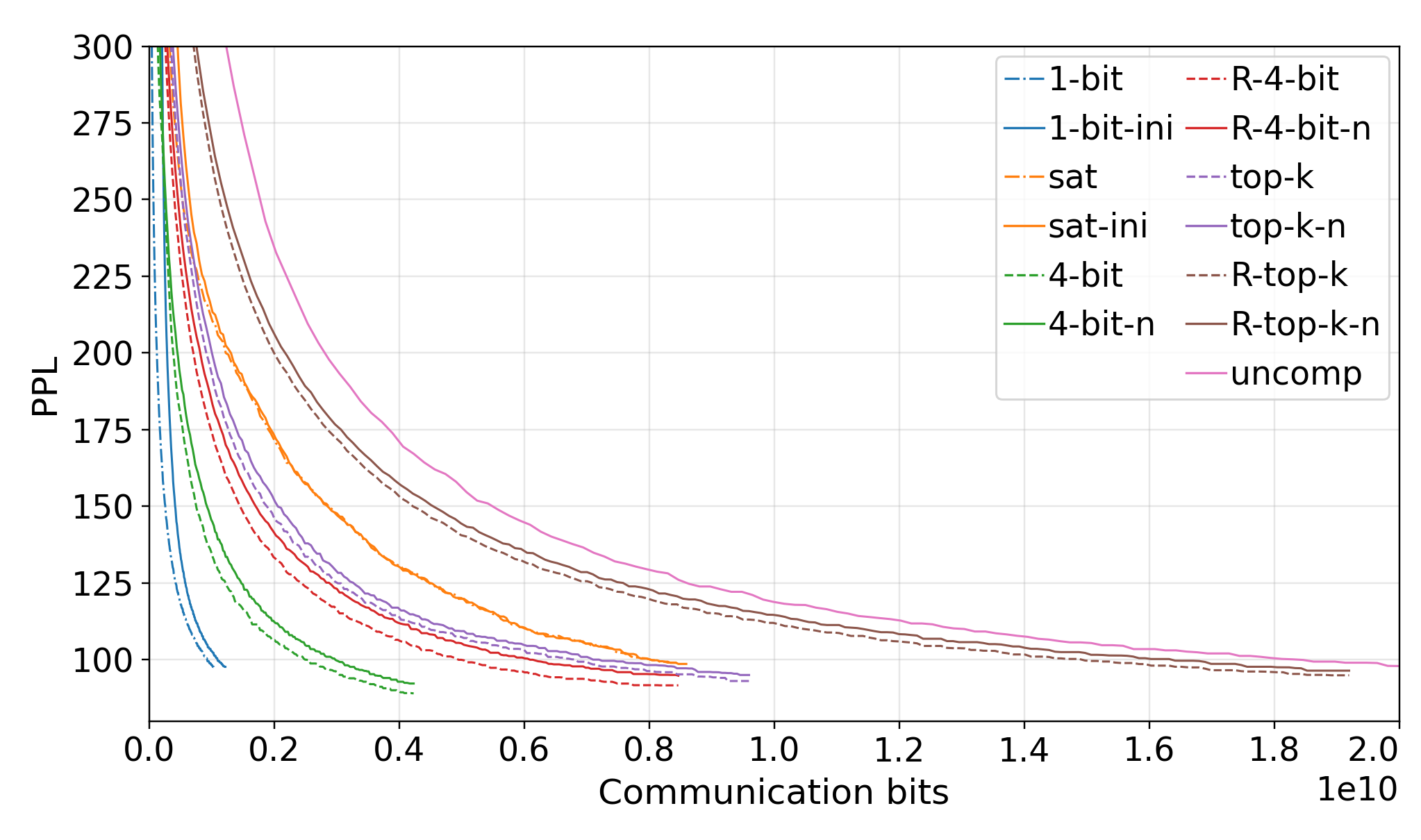}
  \label{ppl:bit}
}
\caption{PPL in LSTM language modeling.}
\label{ppl}
\par\vspace{4pt}
\subfloat[Accuracy versus iterations.]{
\includegraphics[width=.47\textwidth,trim=2mm 2mm 2mm 2mm,clip]{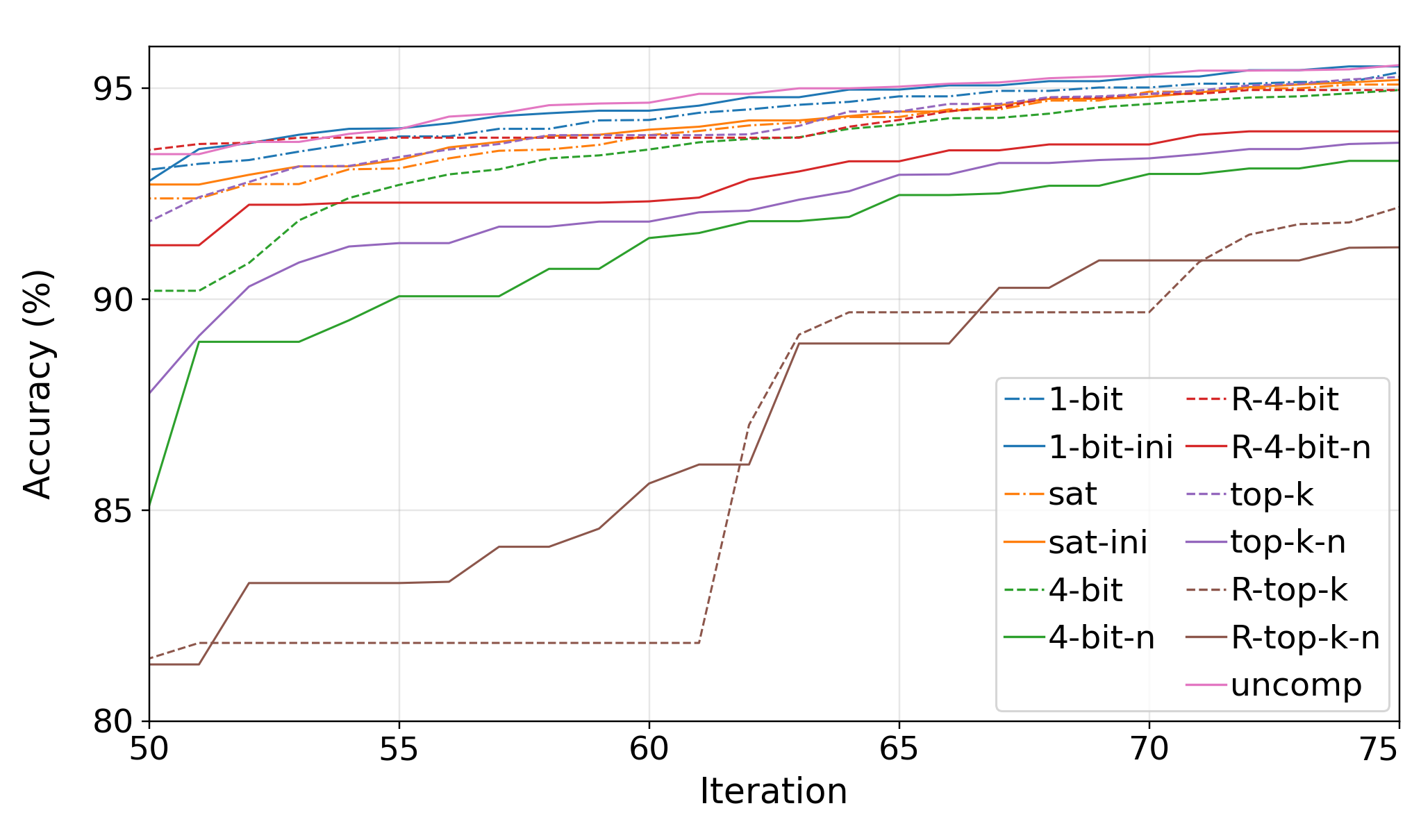}
\label{cnn:iter}
}%
\hspace{0.03\textwidth}%
\subfloat[Accuracy versus communication bits  per edge.]{
\includegraphics[width=.47\textwidth,trim=2mm 2mm 2mm 2mm,clip]{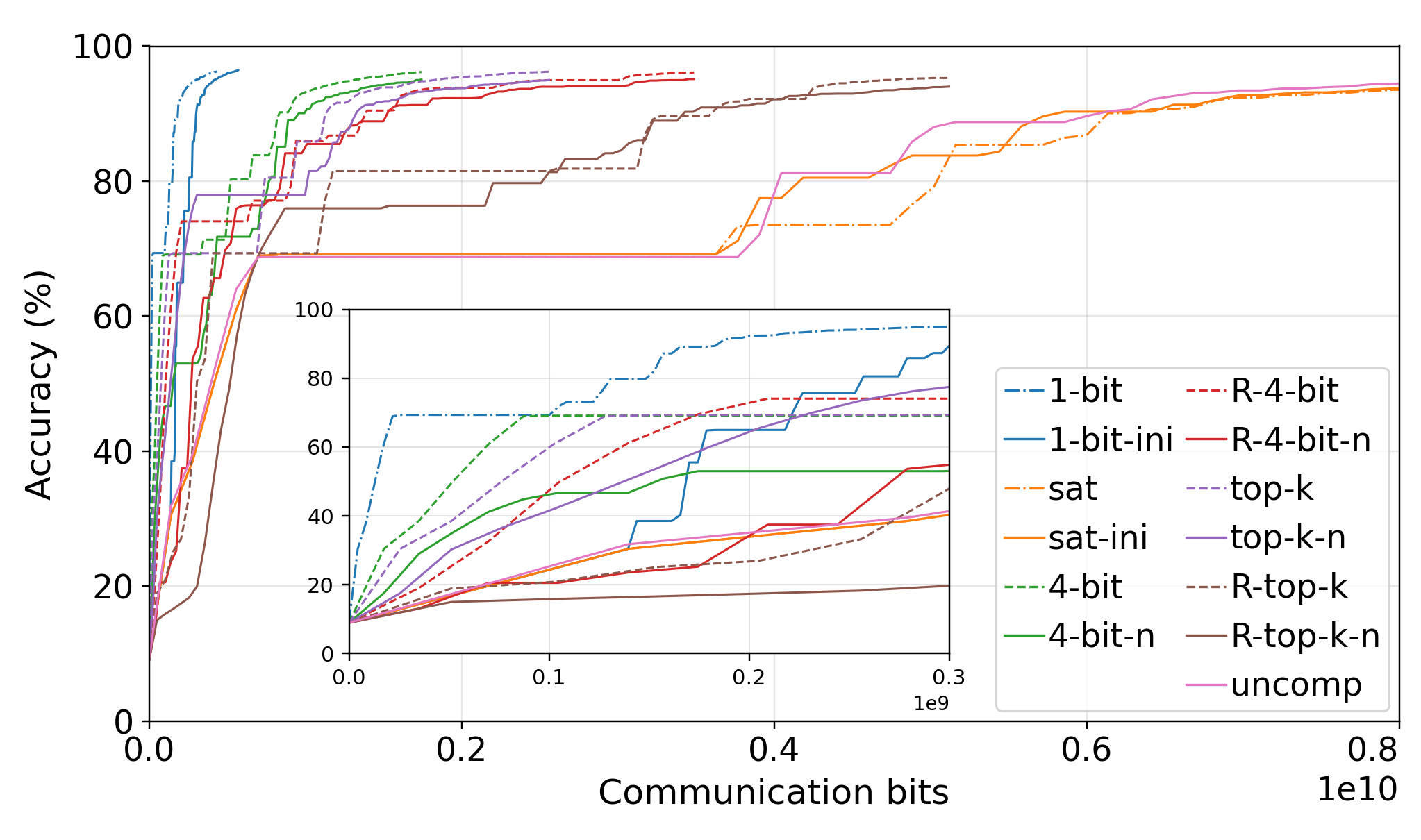}
\label{cnn:bit}
}
\caption{Classification accuracy in CNN image classification.}
\label{cnn}
\vspace{-2pt}
\end{figure*}

\subsection{Logistic Regression}

	We first construct a synthetic logistic regression problem satisfying the P--L condition, 
	where the local cost function is 
	\[
	f_i(x)
	=
	\frac{1}{m_i}
	\sum\nolimits_{j=1}^{m_i}
	\log\bigl(1+\exp(-z_{ij}h_{ij}^{\top}x)\bigr)
	+\lambda\|x\|^2,
	\]
	where $h_{ij}\in\mathbb{R}^d$ is the feature vector, 
	$z_{ij}\in\{-1,1\}$ is the corresponding label, 
	$m_i=200$, $d=50$, and $\lambda=0.001$.
	For $\mathcal{C}_3$, we set $k=10$.

Fig.~\ref{logistic} shows the evolutions of the optimality gap 
$f(\bar{x}_k)-f^\star$ with respect to iterations and communication bits.
The nearly linear decay on both semilogarithmic plots verifies the linear convergence of the proposed algorithm under the P--L condition.
% Moreover, the proposed algorithm works well with various compressors and remains effective in the presence of bounded noise.
% Fig.~\ref{logistic:bit} further shows that compression significantly improves communication efficiency.
This linear convergence behavior is preserved under different compressors and in the presence of bounded noise.

\subsection{LSTM Language Modeling}

To evaluate communication efficiency for high-dimensional problems, we consider distributed training
of a long short-term memory (LSTM) language  model
\cite{zhuang2020adabelief} ($d=5{,}293{,}200$) for next-word prediction.
Experiments are conducted on the Penn Treebank dataset, using half of the training set evenly split among the $n$ agents.
We set 
% the model dimension to $d=5{,}293{,}200$,
the truncated backpropagation through time length to $80$,
and $k=1{,}500{,}000$ for $\mathcal{C}_3$.
Performance is evaluated by perplexity (PPL), defined as
the exponential of the loss function and computed following
\cite{zhuang2020adabelief}.

Fig.~\ref{ppl:iter} shows the evolutions of PPL with respect to the number of iterations.
The proposed algorithm works well with a variety of compressors.
For the locally-bounded compressors,  one initial uncompressed communication round yields a modest speedup.
For the globally-bounded compressors, the proposed algorithm exhibits robustness to the bounded noise.
% In addition, the globally-bounded cases converge slightly faster than the locally-bounded ones, and both are comparable to the uncompressed counterpart.
Fig.~\ref{ppl:bit} demonstrates that compression can greatly improve communication efficiency compared with the uncompressed counterpart,
especially for 1-bit compressors.

\subsection{CNN Image Classification}

Finally, we consider image classification on the full MNIST dataset using a convolutional neural network (CNN) \cite{lecun1998gradient}, where the training set is evenly split among the $n$ agents and the test set is used for evaluation.
The model dimension is $d=436{,}006$.
For $\mathcal{C}_3$, we set $k=40{,}000$.
Performance is evaluated by classification accuracy.

Fig.~\ref{cnn} shows the classification accuracy with respect to the number of iterations and communication bits.
The proposed algorithm achieves comparable accuracy under different compressors, while compression substantially reduces the communication cost.
Fig.~\ref{cnn:C} presents an ablation study on the local region size for the locally-bounded compressors and the bounded-noise level for the globally-bounded compressors, showing that the proposed algorithm is robust to both factors in terms of final classification accuracy.

% \begin{figure*}[!t]
% \centering
% \subfloat[Accuracy versus iterations.]{
% \includegraphics[width=.47\textwidth]{wang52.png}
% \label{cnn:iter}
% }
% \hfill
% \subfloat[Accuracy versus communication bits.]{
% \includegraphics[width=.47\textwidth]{wang63.png}
% \label{cnn:bit}
% }
% \caption{Classification accuracy in CNN image classification.}
% \label{cnn}
% \end{figure*}
\begin{figure}[t]
\centering
\includegraphics[width=.45\textwidth]{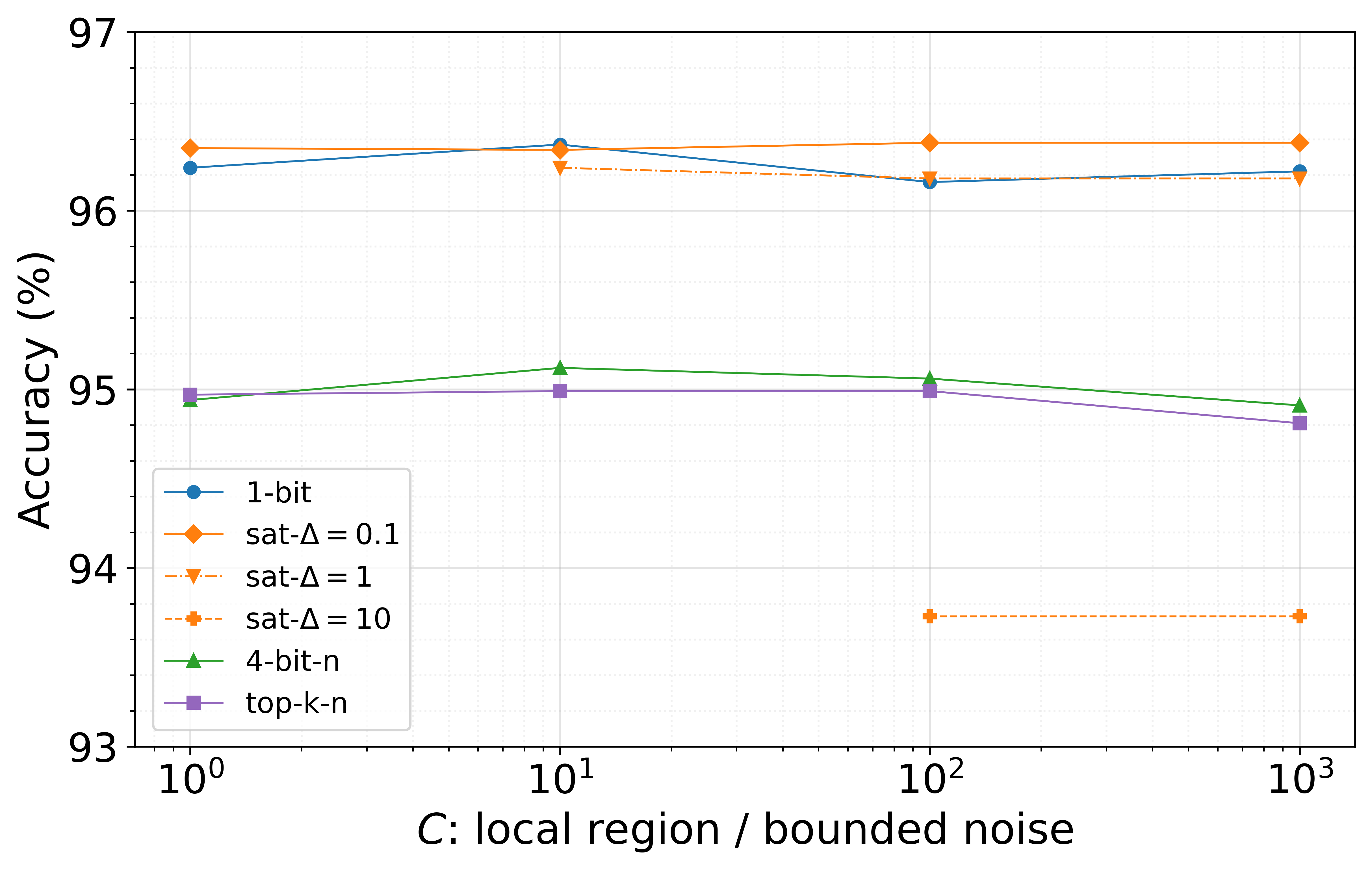}
\caption{Sensitivity to local region size and bounded-noise level in CNN image classification.}
\label{cnn:C}
\end{figure}

\section{Conclusions} \label{zerosg:sec-conclusion}

In this paper, we proposed a unified  compression algorithm for distributed nonconvex optimization to reduce communication overhead.
The proposed algorithm applies to broader classes of compressors than existing methods.
Specifically, in the local sense, we provided the convergence for local compressors in the nonconvex setting, covering low-information compressors such as 1-bit compressors and saturating quantizers.
In the global sense, we accommodated global compressors with both relative and absolute compression errors, as well as additional arbitrary bounded noise,
 thereby covering a broader range of practical compressors.
%  into our framework.
% In particular, it covers the globally-bounded compressors with both relative and absolute compression errors
%  and the locally-bounded compressors including 1-bit compressors and saturating quantizers,
%  and establishes rigorous convergence guarantees in nonconvex settings.
Future work includes time-varying sparse topologies, communication errors, and stochastic-gradient extensions.

\appendix

% \section{Proofs of Lemmas~\ref{lemma:noise} and \ref{lemma:composition}}

\subsection{Proof of Lemma~\ref{lemma:noise}} \label{appendix:bn_co:1}
Denote $c_1=\delta_r/(2-2\delta_r)$. 
For global compressors with relative compression error,
\begin{align}\label{equ:relative}
	&\mathbb{E}_{\mathcal{C}_{r}}\Big[\Big\|\frac{\tilde{\mathcal{C}}_r(x)}{r_r}-x\Big\|^2\Big]
	=\mathbb{E}_{\mathcal{C}_{r}}\Big[\Big\|\frac{\mathcal{C}_r(x)+\xi}{r_r}-x\Big\|^2\Big] \nonumber\\
	&\le\mathbb{E}_{\mathcal{C}_{r}}\Big[(1+c_1)\Big\|\frac{\mathcal{C}_r(x)}{r_r}-x\Big\|^2 
		+ (1+c_1^{-1})\frac{C_\xi^2}{r_r^2}\Big] \nonumber\\
	% &\le\mathbb{E}_{\mathcal{C}_{r}}\Big[\Big(1+\frac{\delta_r}{2(1-\delta_r)}\Big)\Big\|\frac{\mathcal{C}_r(x)}{r_r}-x\Big\|^2 \nonumber\\
	% 	&\quad + \Big(1+\frac{2(1-\delta_r)}{\delta_r}\Big)\frac{C_\xi^2}{r_r^2}\Big] \nonumber\\
	&\le(1-\frac{\delta_r}{2})\|x\|^2 + C_{\tilde{r}}.
	% &\le(1-\frac{\delta_r}{2})\|x\|^2 +  \frac{(2-\delta_r)C_\xi^2}{\delta_r r_r^2},
\end{align}
\looseness=-1
where the first inequality holds due to the Cauchy-Schwarz inequality and $\|\xi\|\le C_\xi$, and the last inequality holds due to \eqref{ass:relative}.

For global compressors with absolute compression error, 
% setting $r=1$ yields
\begin{align}\label{equ:absolute}
	&\mathbb{E}_{\mathcal{C}_{a}}\Big[\Big\|\frac{\tilde{\mathcal{C}}_a(x)}{r_a}-x\Big\|^2\Big]
	=\mathbb{E}_{\mathcal{C}_{a}}\Big[\Big\|\frac{\mathcal{C}_a(x)+\xi}{r_a}-x\Big\|^2\Big] \nonumber\\
	&\le\mathbb{E}_{\mathcal{C}_{a}}\Big[2\Big\|\frac{\mathcal{C}_a(x)}{r_a}-x\Big\|^2 
		+ 2\frac{C_\xi^2}{r_a^2}\Big] \nonumber\\
	% &\le\mathbb{E}_{\mathcal{C}_{a}}[2\|\mathcal{C}_a(x)-x\|^2 +2C_\xi^2] \nonumber\\
	% &\le\mathbb{E}_{\mathcal{C}_{a}}[2\tilde{d}^2\|\mathcal{C}_a(x)-x\|^2_p +2C_\xi^2] \nonumber\\
	&\le C_{\tilde{a}},
	% &\le2\tilde{d}^2C_a + 2C_\xi^2,
\end{align}
\looseness=-1
where the first inequality holds due to the Cauchy--Schwarz inequality and $\|\xi\|\le C_\xi$;
% the second inequality holds due to the norm equivalence;
and the last inequality holds due to \eqref{ass:absolute}.

\subsection{Proof of Lemma~\ref{lemma:composition}} \label{appendix:bn_co:2}
Without ambiguity, $\mathbb{E}_{\mathcal{C}}[\cdot]$ is taken only over the internal randomness of $\mathcal{C}_{r}$ and $\mathcal{C}_{a}$, and not over the bounded noise. Equivalently, the results hold uniformly for any bounded noise realization.
Denote $c_2=\delta_r/(4-2\delta_r)$ and $c_3=\delta_r/(8-2\delta_r)$.
 Consider the composition $\tilde{\mathcal{C}}_r(\tilde{\mathcal{C}}_a(x)/r_a)$ with $r=r_r$, 
\begin{align}
	&\mathbb{E}_{\mathcal{C}}\hspace{-0.25em}\Big[\hspace{-0.05em}\Big\|\hspace{-0.1em}\frac{\tilde{\mathcal{C}}_r\hspace{-0.1em}(\tilde{\mathcal{C}}_a\hspace{-0.1em}(x)\hspace{-0.1em}/r_a)}{r_r}\hspace{-0.2em}-\hspace{-0.15em}x\hspace{-0.05em}\Big\|^2\hspace{-0.05em}\Big]
	\hspace{-0.3em}=\hspace{-0.1em}\mathbb{E}_{\mathcal{C}_{a}}\hspace{-0.3em}\Big[\hspace{-0.05em}\mathbb{E}_{\mathcal{C}_{r}}\hspace{-0.3em}\Big[\Big\|\frac{\tilde{\mathcal{C}}_r\hspace{-0.1em}(\tilde{\mathcal{C}}_a(x)\hspace{-0.1em}/r_a)}{r_r}\hspace{-0.2em}-\hspace{-0.15em}x\Big\|^2\hspace{-0.05em}\Big]\hspace{-0.1em}\Big] \nonumber\\
	&=\hspace{-0.1em}\mathbb{E}_{\mathcal{C}_{a}}\hspace{-0.2em}\Big[\mathbb{E}_{\mathcal{C}_{r}}\hspace{-0.2em}\Big[\Big\|\frac{\tilde{\mathcal{C}}_r(\tilde{\mathcal{C}}_a(x)/r_a)}{r_r}-\frac{\tilde{\mathcal{C}}_a(x)}{r_a}+\frac{\tilde{\mathcal{C}}_a(x)}{r_a}-x\Big\|^2\Big]\Big] \nonumber\\
	&\le\hspace{-0.1em}\mathbb{E}_{\mathcal{C}_{a}}\hspace{-0.2em}\Big[\mathbb{E}_{\mathcal{C}_{r}}\hspace{-0.2em}\Big[
		(1+c_2)\Big\|\frac{\tilde{\mathcal{C}}_r(\tilde{\mathcal{C}}_a(x)/r_a)}{r_r}-\frac{\tilde{\mathcal{C}}_a(x)}{r_a}\Big\|^2 \nonumber\\
		&\quad +(1+c_2^{-1})\Big\|\frac{\tilde{\mathcal{C}}_a(x)}{r_a}-x\Big\|^2\Big]\Big] \nonumber\\
	&\le\hspace{-0.1em}\mathbb{E}_{\mathcal{C}_{a}}\hspace{-0.2em}\Big[
		(1-\frac{\delta_r}{4})\Big\|\frac{\tilde{\mathcal{C}}_a(x)}{r_a}-x+x\Big\|^2 + (1+c_2)C_{\tilde{r}} \nonumber\\
		&\quad +(1+c_2^{-1})\Big\|\frac{\tilde{\mathcal{C}}_a(x)}{r_a}-x\Big\|^2\Big] \nonumber\\
	&\le\hspace{-0.1em}\mathbb{E}_{\mathcal{C}_{a}}\hspace{-0.2em}\Big[
		(1+c_3)(1-\frac{\delta_r}{4})\|x\|^2 + \Big((1+c_3^{-1})(1-\frac{\delta_r}{4}) \nonumber\\
		 &\quad +(1+c_2^{-1})\Big)\Big\|\frac{\tilde{\mathcal{C}}_a(x)}{r_a}-x\Big\|^2 + (1+c_2)C_{\tilde{r}}\Big] \nonumber\\
	&\le
		(1-\frac{\delta_r}{8})\|x\|^2 + \Big((1+c_3^{-1})(1-\frac{\delta_r}{4}) \nonumber\\
		&\quad +(1+c_2^{-1})\Big)C_{\tilde{a}} + (1+c_2)C_{\tilde{r}},
	% &\le\mathbb{E}_{\mathcal{C}_{r}}\Big[\Big(1+\frac{\delta_r}{2(1-\delta_r)}\Big)\Big\|\frac{\tilde{\mathcal{C}}_r(x)}{r_r}-x\Big\|^2 \nonumber\\
	% 	&\quad + \Big(1+\frac{2(1-\delta_r)}{\delta_r}\Big)\frac{C_\xi^2}{r_r^2}\Big] \nonumber\\
	% &\le(1-\frac{\delta_r}{8})\|x\|^2 +  \frac{(2-\delta_r)(6-\delta_r)\tilde{d}^2C_a}{2\delta_r},
\end{align}
\looseness=-1
where the first equality holds due to the independence of $\mathcal{C}_{r}$ and $\mathcal{C}_{a}$;
the first inequality holds due to the Cauchy--Schwarz inequality;
the second inequality holds due to \eqref{equ:relative} and the independence of $\mathcal{C}_{r}$ from both $\mathcal{C}_{a}$ and $x$;
the third inequality holds due to the Cauchy--Schwarz inequality;
and the last inequality holds due to \eqref{equ:absolute} and the independence of $x$ and $\mathcal{C}_{a}$.

Consider the composition $\tilde{\mathcal{C}}_a(\tilde{\mathcal{C}}_r(x))$ with $r=r_rr_a$,
\begin{align}
	&\mathbb{E}_{\mathcal{C}}\Big[\Big\|\frac{\tilde{\mathcal{C}}_a(\tilde{\mathcal{C}}_r(x))}{r_rr_a}\hspace{-0.2em}-\hspace{-0.15em}x\Big\|^2\Big]
		\hspace{-0.3em}=\hspace{-0.1em}\mathbb{E}_{\mathcal{C}_{r}}\hspace{-0.2em}\Big[\mathbb{E}_{\mathcal{C}_{a}}\hspace{-0.2em}\Big[\Big\|\frac{\tilde{\mathcal{C}}_a(\tilde{\mathcal{C}}_r(x))}{r_rr_a}\hspace{-0.2em}-\hspace{-0.15em}x\Big\|^2\Big]\Big] \nonumber\\
	&=\hspace{-0.1em}\mathbb{E}_{\mathcal{C}_{r}}\hspace{-0.2em}\Big[\mathbb{E}_{\mathcal{C}_{a}}\hspace{-0.2em}\Big[
		\Big\|\frac{1}{r_r}\Big(\frac{\tilde{\mathcal{C}}_a(\tilde{\mathcal{C}}_r(x))}{r_a} -\tilde{\mathcal{C}}_r(x) \Big) + \frac{\tilde{\mathcal{C}}_r(x)}{r_r}-x\Big\|^2\Big]\Big] \nonumber\\
	&\le\hspace{-0.1em}\mathbb{E}_{\mathcal{C}_{r}}\hspace{-0.2em}\Big[\mathbb{E}_{\mathcal{C}_{a}}\hspace{-0.2em}\Big[
		\frac{(1+c_2^{-1})}{r_r^2}\Big\|\frac{\tilde{\mathcal{C}}_a(\tilde{\mathcal{C}}_r(x))}{r_a}-\tilde{\mathcal{C}}_r(x)\Big\|^2\Big]\Big]  \nonumber\\
		&\quad +\mathbb{E}_{\mathcal{C}_{r}}\hspace{-0.2em}\Big[(1+c_2)\Big\|\frac{\tilde{\mathcal{C}}_r(x)}{r_r}-x\Big\|^2\Big] \nonumber\\
	&\le (1-\frac{\delta_r}{4})\|x\|^2 + (1+c_2)C_{\tilde{r}} + \frac{(1+c_2^{-1})C_{\tilde{a}}}{r_r^2},
	% 		\nonumber\\
	% &\le\hspace{-0.1em}\mathbb{E}_{\mathcal{C}_{r}}\hspace{-0.2em}\Big[\mathbb{E}_{\mathcal{C}_{a}}\hspace{-0.2em}\Big[
	% 	\frac{(1+c_2^{-1})\tilde{d}^2}{r_r^2}\|\tilde{\mathcal{C}}_a(\tilde{\mathcal{C}}_r(x))-\tilde{\mathcal{C}}_r(x)\|^2_p\Big]\Big] \nonumber\\ 
	% 	&\quad +(1-\frac{\delta_r}{2})\|x\|^2 \nonumber\\
	% &\le(1-\frac{\delta_r}{2})\|x\|^2 +  \frac{(2-\delta_r)\tilde{d}^2C_a}{\delta_r r_r^2},
\end{align}
% \begin{align}
% 	&\mathbb{E}_{\mathcal{C}}\Big[\Big\|\frac{\mathcal{C}_a(\mathcal{C}_r(x))}{r_r}\hspace{-0.2em}-\hspace{-0.15em}x\Big\|^2\Big]
% 		\hspace{-0.3em}=\hspace{-0.1em}\mathbb{E}_{\mathcal{C}_{r}}\hspace{-0.2em}\Big[\mathbb{E}_{\mathcal{C}_{a}}\hspace{-0.2em}\Big[\Big\|\frac{\mathcal{C}_a(\mathcal{C}_r(x))}{r_r}\hspace{-0.2em}-\hspace{-0.15em}x\Big\|^2\Big]\Big] \nonumber\\
% 	&=\hspace{-0.1em}\mathbb{E}_{\mathcal{C}_{r}}\hspace{-0.2em}\Big[\mathbb{E}_{\mathcal{C}_{a}}\hspace{-0.2em}\Big[
% 		\Big\|\frac{\mathcal{C}_a(\mathcal{C}_r(x))-\mathcal{C}_r(x)}{r_r} + \frac{\mathcal{C}_r(x)}{r_r}-x\Big\|^2\Big]\Big] \nonumber\\
% 	&\le\hspace{-0.1em}\mathbb{E}_{\mathcal{C}_{r}}\hspace{-0.2em}\Big[\mathbb{E}_{\mathcal{C}_{a}}\hspace{-0.2em}\Big[
% 		\frac{(1+c_1^{-1})}{r_r^2}\|\mathcal{C}_a(\mathcal{C}_r(x))-\mathcal{C}_a(x)\|^2\Big]\Big]  \nonumber\\
% 		&\quad +\mathbb{E}_{\mathcal{C}_{r}}\hspace{-0.2em}\Big[(1+c_1)\Big\|\frac{\mathcal{C}_r(x)}{r_r}-x\Big\|^2\Big] \nonumber\\
% 	&\le\hspace{-0.1em}\mathbb{E}_{\mathcal{C}_{r}}\hspace{-0.2em}\Big[\mathbb{E}_{\mathcal{C}_{a}}\hspace{-0.2em}\Big[
% 		\frac{(1+c_1^{-1})\tilde{d}^2}{r_r^2}\|\mathcal{C}_a(\mathcal{C}_r(x))-\mathcal{C}_a(x)\|^2_p\Big]\Big] \nonumber\\ 
% 		&\quad +(1-\frac{\delta_r}{2})\|x\|^2 \nonumber\\
% 	&\le(1-\frac{\delta_r}{2})\|x\|^2 +  \frac{(2-\delta_r)\tilde{d}^2C_a}{\delta_r r_r^2},
% \end{align}
where the first equality holds due to the independence of $\mathcal{C}_r$ and $\mathcal{C}_a$;
the first inequality holds due to the Cauchy--Schwarz inequality and the independence of $\mathcal{C}_{a}$ from both $\mathcal{C}_{r}$ and $x$;
and the last inequality holds due to \eqref{equ:absolute} and \eqref{equ:relative}.

\subsection{Technical Preliminaries}\label{zero:app-lemmas}

	\subsubsection{Properties of Graph Matrices}
 \begin{lemma}[Lemma 3,~\cite{Yi_CommunicationCompression_2023}]
  The matrices $L$ and $E$ are positive semi-definite, while $F$ is positive definite. 
  They satisfy
  \begin{subequations}		
		\begin{align}
			&EL=LE=L,\label{nonconvex:KL-L-eq}\\ 
			%&E^2=E,\label{nonconvex:KL-L-eq1}\\
			&0\le\rho_2(L)E\le L\le\rho(L)E,\label{nonconvex:KL-L-eq2}\\
			&FL=LF=E,\label{nonconvex:lemma-eq3}\\
			%&PE=EP=Q,\label{nonconvex:lemma-eq4}\\
			&\rho^{-1}(L){\bf I}_n\leq F \le\rho_2^{-1}(L){\bf I}_n.\label{nonconvex:lemma-eq5}
		\end{align}
	\end{subequations}
 \end{lemma}

\subsubsection{Smoothness}

\begin{lemma}[Lemma~3.4,~\cite{bubeck2015convex}]
Let $f:\mathbb{R}^p$$\to\mathbb{R}$ be an $\ell$-smooth differentiable function with $\ell>0$. 
Then for any $x,y\in\mathbb{R}^p$,
\begin{subequations}\label{nonconvex:lemma:smoothness}
\begin{align}
&|f(y)-f(x)-(y-x)^{\!\top}\nabla f(x)|
\le \tfrac{\ell}{2}\|y-x\|^2, \label{nonconvex:lemma:lipschitz}\\
&\tfrac{1}{2}\|\nabla f(x)\|^2
\le \ell\,(f(x)-f^*). \label{rePL}
\end{align}
\end{subequations}
\end{lemma}

\subsection{Proof of Lemma~\ref{Lemma: Lyapunov Difference}}\label{appendix:lemmas:c}
We denote the following constants:
\begin{align*}
	&\hat{\kappa}_0=\min\Big\{\frac{\varphi_1}{\varphi_2},~\frac{\varphi_3}{\varphi_4},
	~\frac{\varphi_5}{\varphi_6}\Big\},\\
	&\kappa_1=\frac{4}{\rho_2(L)},\\
	&\kappa_2=\max\Big\{2+2\ell^2,~\frac{5}{\rho_2(L)},~\sqrt[3]{\frac{16\ell^2(\tau_1+1)^2}{\rho_2(L)}},~\frac{2\sqrt{2}\ell}{\rho_2(L)}\Big\},\\
	% &\kappa_7=\min\Big\{\hat{\kappa}_0,~\frac{1}{2\ell}\Big\},\\
	&\psi_1=2(\varphi_5+\alpha\varphi_6),\\
	&\psi_2=\varphi_7+\alpha\varphi_8,\\
	&\psi_3=4(1+\varepsilon_5^{-1})\hat{d}^2\beta^2\rho^2(L),\\
	&\psi_4=4(1+\varepsilon_5^{-1})\hat{d}^2\max\{(\beta^2\rho^2(L)+\ell^2),~\gamma^2\rho(L)\}/\varepsilon_1,\\
	&\varepsilon_1=\frac{\tau_1\rho_2(L)-1}{2\tau_1\rho_2(L)},\\
	&\varepsilon_2=\max\Big\{\frac{1+\tau_1\rho_2(L)}{2},
		~\frac{1+\tau_1}{2}+\frac{1}{2\tau_1\rho_2^2(L)}\Big\},\\
	&\varepsilon_3=\min\{\varphi_1-\alpha\varphi_2,~\varphi_3-\alpha\varphi_4\},\\
	% &\varepsilon_4=\min\{\varphi_1-\alpha\varphi_2,~\alpha\varphi_3-\alpha\varphi_4,~\frac{1}{4}\},\\
	% &\varepsilon_4=\min\{\varepsilon_3,~\frac{1}{4}\},\\
	&\varepsilon_5=\omega r(\delta-\delta^2/2),\\
	&\varphi_1=\frac{1}{2}\big(\rho_2(L)\beta-(3\gamma+2+3\ell^2)\big),\\
	&\varphi_2= 3\rho^2(L)\beta^2  - \rho_2(L)\beta\gamma + \big(\rho(L)+2\big)\gamma^2
		+ 1 + \frac{5}{2}\ell^2,\\
	% &\varphi_1^\prime=\frac{1}{2}\big(\rho_2(L)\beta-(3\gamma+2+2\ell^2+2\varphi_5\ell^2)\big),\\
	% &\varphi_2^\prime= 3\rho^2(L)\beta^2  - \rho_2(L)\beta\gamma + \big(\rho(L)+2\big)\gamma^2\\
	% 	&\quad + 1 + \frac{5}{2}\ell^2 + 2\varphi_6\ell^2,\\
	% &\varphi_1=\frac{1}{2}\big(\rho_2(L)\beta-(3\gamma+2+3\ell^2)\big),\\
	% &\varphi_2= 3\rho^2(L)\beta^2 - \rho_2(L)\beta\gamma + \big(\rho(L)+2\big)\gamma^2 + 1 + \frac{5}{2}\ell^2,\\
	&\varphi_3=\frac{\gamma}{2}-\frac{5}{2}\rho^{-1}_2(L),\\
	&\varphi_4=2\rho(L)\gamma^2+\frac{1}{2}\rho(L),\\
	% &\varphi_5=\frac{1}{8} + \frac{(\beta+\gamma)^2}{\gamma^5}\rho_2^{-1}(L)\ell^2
	% 	+ \frac{1}{2\gamma^2}\rho_2^{-2}(L)\ell^2,\\
	% &\varphi_6=\frac{(\beta+\gamma)}{2\gamma^3}\rho_2^{-1}(L)\ell^2 + \frac{\rho_2^{-2}(L)}{2}\ell^2 + \frac{\ell^2}{2},\\
	&\varphi_5=\frac{1}{8}-\frac{(\beta+\gamma)^2}{\gamma^5}\rho_2^{-1}(L)\ell^2
		-\frac{1}{2\gamma^2}\rho_2^{-2}(L)\ell^2,\\
	&\varphi_6=\frac{(\beta+\gamma)}{2\gamma^3}\rho_2^{-1}(L)\ell^2 + \frac{1}{2\gamma^2}\rho_2^{-2}(L)\ell^2 + \frac{\rho_2^{-2}(L)}{2}\ell^2  \\
		&\quad + \frac{\ell^2}{2}
		+\frac{\ell}{2},\\
	&\varphi_7=\big( \rho(L) + 1 \big)\gamma + \frac{1}{2}\rho(L)\beta,\\
	&\varphi_8=3\rho^2(L)\beta^2 \hspace{-0.1em}+\hspace{-0.1em} \big(\rho(L) \hspace{-0.1em}-\hspace{-0.1em} 2\rho_2(L) \big)\beta\gamma \hspace{-0.1em}+\hspace{-0.1em} \big(\rho(L) \hspace{-0.1em}+\hspace{-0.1em} 2 \big)\gamma^2 \hspace{-0.1em}+\hspace{-0.1em} 1.
\end{align*}

\subsubsection{Compact Form of the Algorithm}

	Denote $\bsL=L\otimes{\bf I}_d$ and $\bar{\bsg}_k=\bsH\bsg_k$. To simplify the analysis, we first introduce the compact form of the algorithm \eqref{nonconvex:kia-algo-dc-a}--\eqref{zerosg:algorithm-random-pd}: 
	% and present several useful equations. 
	% Without ambiguity, denote $\mathcal{C}(\bsx)=\col(\mathcal{C}(x_1),\dots,\mathcal{C}(x_n))$, then
	\begin{subequations}\label{nonconvex:alg-compress-compact}
		\begin{align}
			&\bsq_{k}=\mathcal{C}((\bsx_{k}-\hat{\bsx}_{k-1})/s_k),\label{nonconvex:kia-algo-dc-compact-q}\\
			&\hat{\bsx}_{k}=\hat{\bsx}_{k-1}+\omega s_k\bsq_k,\label{nonconvex:kia-algo-dc-compact-a}\\
			&\bsy_{k}=\bsy_{k-1}+\omega s_k\bsL\bsq_k,\label{nonconvex:alg-compress-compact-b}\\
			&\bm{x}_{k+1}=\bm{x}_k-\alpha(\beta\bsL\hat{\bsx}_k+\gamma\bm{v}_k + \bsg_k),\label{nonconvex:kia-algo-dc-compact-x}\\
			&\bm{v}_{k+1}=\bm{v}_k+\alpha\gamma\bsL\hat{\bsx}_k.\label{nonconvex:kia-algo-dc-compact-v}
		\end{align}
	\end{subequations}
	Noting $\sum_{i=1}^{n}L_{ij}=0$ and $\sum_{i=1}^{n}v_{i,0}={\bf0}_d$, \eqref{nonconvex:kia-algo-dc-compact-v} implies 
	% it follows that
	\begin{align}
		\bar{v}_k={\bm 0}_d.\label{nonconvex:vkn}
	\end{align}
	Then combining \eqref{nonconvex:vkn} and \eqref{nonconvex:kia-algo-dc-compact-x} yields
	\begin{align}
		&\bar{\bsx}_{k+1}=\bar{\bsx}_{k}-\alpha\bar{\bsg}_k.\label{nonconvex:xbardynamic}
	\end{align}
	From Assumption~\ref{nonconvex:ass:fiu}, one has the smoothness of $\tilde{f}$ and
	\begin{align}
		\|\bsg^0_{k}-\bsg_{k}\|^2\le \ell^2\|\bar{\bsx}_{k}-\bsx_{k}\|^2=\ell^2\|\bsx_{k}\|^2_{\bsE}.\label{nonconvex:gg1}
	\end{align}
	Moreover, $\rho(\bsH)=1$ gives
	\begin{align}
		&\|\bar{\bsg}^0_k-\bar{\bsg}_k\|^2=\|\bsH(\bsg^0_{k}-\bsg_{k})\|^2 \nonumber\\
		&\quad \le\|\bsg^0_{k}-\bsg_{k}\|^2\le \ell^2\|\bsx_k\|^2_{\bsE}.\label{nonconvex:gg2}
	\end{align}
	Similarly, the smoothness of $\tilde{f}$ and \eqref{nonconvex:xbardynamic} gives
	\begin{align}
		&\|\bsg^0_{k+1}-\bsg^0_{k}\|^2\le \ell^2\|\bar{\bsx}_{k+1}-\bar{\bsx}_{k}\|^2=\alpha^2\ell^2\|\bar{\bsg}_k\|^2.\label{nonconvex:gg}
	\end{align}
	
\subsubsection{One-Step Difference of Lyapunov Components}
We denote the following constants:
\begin{align*}
	&\tau_2=\Big(\frac{\beta+\gamma}{2\gamma^3} + \frac{(\beta+\gamma)^2}{\alpha\gamma^5} \Big)
		\rho_2^{-1}(L)+\frac{1}{2},\\
	&\tau_3=(\frac{\alpha+1}{2\alpha\gamma^2}+\frac{1}{2})\rho_2^{-2}(L).
\end{align*}
% We begin by analyzing the one-step difference of the Lyapunov function $\mathcal{L}_{1,k}$.
% To this end, we bound the components $e_{1,k}$--$e_{4,k}$ separately via Lemmas~\ref{Lemma Consensus}--\ref{Lemma Optimality 2}.
% Denote some constants
To analyze $\mathcal{L}_{1,k}$,
 we bound the one-step differences of $e_{1,k}$--$e_{4,k}$ separately in Lemmas~\ref{Lemma Consensus}--\ref{Lemma Optimality 2}.
 For readability, the proofs are provided in Appendix~\ref{imed}.
% We analyze the Lyapunov function $\mathcal{L}_{1,k}$ term by term, i.e., 
% we bound the one-step differences of $e_{1,k}$--$e_{4,k}$ separately in Lemmas~\ref{Lemma Consensus}--\ref{Lemma Optimality 2}.
% Before proceeding, we denote
% \begin{align*}
% 	&\tau_2=\Big(\frac{\beta+\gamma}{2\gamma^3} + \frac{(\beta+\gamma)^2}{\alpha\gamma^5} \Big)
% 		\rho_2^{-1}(L)+\frac{1}{2},\\
% 	&\tau_3=(\frac{\alpha+1}{2\alpha\gamma^2}+\frac{1}{2})\rho_2^{-2}(L).
% \end{align*}

\begin{lemma}\label{Lemma Consensus}
			Suppose Assumptions~\ref{nonconvex:ass:graph}--\ref{nonconvex:ass:fiu} holds. 
			Let $\{x_{i,k}\}$ be the sequence generated by Algorithm~\ref{nonconvex:algorithm-pdgd}.
			Then,
			% Then, the following holds for Algorithm~\ref{nonconvex:algorithm-pdgd}:
			\begin{align}
				&e_{1,k+1}\le e_{1,k}-\|\bsx_k\|^2_{\frac{\alpha\beta}{2}\bsL-\frac{\alpha}{2}\bsE
				-\frac{\alpha}{2}(1+3\alpha)\ell^2\bsE - 3\alpha^2\beta^2\bsL^2}   \nonumber\\
			&\quad-\alpha\gamma\hat{\bsx}^\top_k\bsE\Big(\bm{v}_k+\frac{1}{\gamma}\bsg_k^0\Big) +\Big\|\bm{v}_k+\frac{1}{\gamma}\bsg_k^0\Big\|^2_{\frac{6\alpha^2\gamma^2\rho(L)+\alpha\gamma}{4}\bsF}
			\nonumber\\
			&\quad+ \big(\frac{\alpha}{2}(\beta+2\gamma)\rho(L) + 3\alpha^2\beta^2\rho^2(L)\big)\|\bsx_k-\hat{\bsx}_k\|^2. \hspace{-0.5em}
			% &\quad +\Big\|\bm{v}_k+\frac{1}{\gamma}\bsg_k^0\Big\|^2_{\frac{6\alpha^2\gamma^2\rho(L)+\alpha\gamma}{4}\bsF}. 
			\label{nonconvex:v1k}
		\end{align}
	\end{lemma}

	\begin{lemma}\label{Lemma Optimality 1}
		Suppose Assumptions~\ref{nonconvex:ass:graph}--\ref{nonconvex:ass:fiu} holds.
		Let $\{x_{i,k}\}$ be the sequence generated by Algorithm~\ref{nonconvex:algorithm-pdgd}.
		Then,
		% Then, the following holds for Algorithm~\ref{nonconvex:algorithm-pdgd}:% with $\Delta_k=\frac{1}{\gamma}-\frac{1}{\gamma_{k+1}}$, and $\epsilon_5, \epsilon_6, \varepsilon_1$ are given in Appendix~\ref{appendix:lemmas}.:
		\begin{align}
			& e_{2,k+1} 
			\le e_{2,k} 
			+\alpha(\beta+\gamma)\hat{\bsx}^\top_k\bsE\Big(\bm{v}_k+\frac{1}{\gamma}\bsg_k^0\Big)\nonumber\\
			&\quad+ \big(\alpha^2\gamma(\beta+\gamma)\rho(L)+\alpha^2\big)\|\bm{x}_k-\hat{\bm{x}}_k\|^2 \nonumber\\
			&\quad +\Big\|\bm{v}_k+\frac{1}{\gamma}\bsg_{k}^0\Big\|^2_{\frac{\alpha\gamma}{4}\bsF}
				+\|\bm{x}_k\|^2_{\alpha^2\gamma(\beta+\gamma)\bsL + \alpha^2\bsE} \nonumber\\
			&\quad	+ \alpha^2 \tau_2\ell^2\|\bar{\bsg}_k\|^2
				 + \frac{\alpha\beta}{\gamma}\hat{\bm{x}}_k^\top\bsE(\bsg_{k+1}^0-\bsg_{k}^0).\label{nonconvex:v2k}
		\end{align}
	% where $\tau_2$ is given in Appendix~\ref{appendix:lemmas}.
	\end{lemma}%=\frac{1}{\epsilon_2}(\alpha-\alpha_{k+1})

	\begin{lemma}\label{Lemma Cross}
		Suppose Assumptions~\ref{nonconvex:ass:graph}--\ref{nonconvex:ass:fiu} holds. 
		Let $\{x_{i,k}\}$ be the sequence generated by Algorithm~\ref{nonconvex:algorithm-pdgd}.
		Then,
		% Then, the following holds for Algorithm~\ref{nonconvex:algorithm-pdgd}:% with $\epsilon_7$ and $\epsilon_8$ are given in Appendix~\ref{appendix:lemmas}:
		\begin{align}
			& e_{3,k+1}
			\le e_{3,k} 
		-\alpha\beta\hat{\bsx}_k^\top\bsE\Big(\bm{v}_k+\frac{1}{\gamma}\bsg_{k}^0\Big)\nonumber\\
		&\quad +\|\bm{x}_k\|^2_{\frac{\alpha(3\gamma+1)}{2}\bsE + 2\alpha^2(\gamma^2\bsE-\beta\gamma\bsL)+\frac{\alpha}{2}(1+2\alpha)\ell^2\bsE} \nonumber\\
		&\quad +\big(\alpha\gamma + 2\alpha^2(\gamma^2-\beta\gamma\rho_2(L))\big)\|\bm{x}_k-\hat{\bm{x}}_k\|^2 \nonumber\\
		&\quad-\frac{\alpha\beta}{\gamma}\hat{\bm{x}}_k^\top
		\bsE(\bsg_{k+1}^0-\bsg_{k}^0)+\Big(\alpha^2\tau_3\ell^2+\frac{\alpha}{8}\Big)\|\bar{\bsg}_k\|^2\nonumber\\
		&\quad-\Big\|\bm{v}_k+\frac{1}{\gamma}\bsg_{k}^0\Big\|^2_{\alpha(\gamma-\frac{5}{2}\rho^{-1}_2(L))\bsF
			-\frac{\alpha^2}{2}(\gamma^2+1)\rho(L)\bsF}.\label{nonconvex:v3k}
			\end{align}
			% where $\tau_3$ is given in Appendix~\ref{appendix:lemmas}.
	\end{lemma}
	\begin{lemma}\label{Lemma Optimality 2}
		Suppose Assumptions~\ref{nonconvex:ass:graph}--\ref{nonconvex:ass:fiu} holds . 
		Let $\{x_{i,k}\}$ be the sequence generated by Algorithm~\ref{nonconvex:algorithm-pdgd}.
		Then,
		% Then, the following holds for Algorithm~\ref{nonconvex:algorithm-pdgd}:
		\begin{align}
			&e_{4,k+1}
			\le e_{4,k}-\frac{\alpha}{4}(1-2\alpha \ell)\|\bar{\bsg}_{k}\|^2
		+\|\bsx_k\|^2_{\frac{\alpha}{2}\ell^2\bsE}\nonumber\\
		&\quad-\frac{\alpha}{4}\|\bar{\bsg}_{k}^0\|^2.\label{nonconvex:v4k}
		\end{align}
	\end{lemma}

\subsubsection{Main Poof}
% In addition to the notations defined in Appendix~\ref{appendix:lemmas},
% 	we also denote the following notations.
	Now it is ready to prove Lemma~\ref{Lemma: Lyapunov Difference}.	
	Denote
	\begin{align*}
	% &\varphi_1=\frac{1}{2}\big(\rho_2(L)\beta-(3\gamma+2+2\ell^2)\big),\\
	% &\varphi_2= 3\rho^2(L)\beta^2  - \rho_2(L)\beta\gamma + \big(\rho(L)+2\big)\gamma^2
	% 	+ 1 + \frac{5}{2}\ell^2,\\
	% % &\varphi_1^\prime=\frac{1}{2}\big(\rho_2(L)\beta-(3\gamma+2+2\ell^2+2\varphi_5\ell^2)\big),\\
	% % &\varphi_2^\prime= 3\rho^2(L)\beta^2  - \rho_2(L)\beta\gamma + \big(\rho(L)+2\big)\gamma^2\\
	% % 	&\quad + 1 + \frac{5}{2}\ell^2 + 2\varphi_6\ell^2,\\
	% % &\varphi_1=\frac{1}{2}\big(\rho_2(L)\beta-(3\gamma+2+3\ell^2)\big),\\
	% % &\varphi_2= 3\rho^2(L)\beta^2 - \rho_2(L)\beta\gamma + \big(\rho(L)+2\big)\gamma^2 + 1 + \frac{5}{2}\ell^2,\\
	% &\varphi_3=\frac{\gamma}{2}-\frac{5}{2}\rho^{-1}_2(L),\\
	% &\varphi_4=2\rho(L)\gamma^2+\frac{1}{2}\rho(L),\\
	% % &\varphi_5=\frac{1}{8} + \frac{(\beta+\gamma)^2}{\gamma^5}\rho_2^{-1}(L)\ell^2
	% % 	+ \frac{1}{2\gamma^2}\rho_2^{-2}(L)\ell^2,\\
	% % &\varphi_6=\frac{(\beta+\gamma)}{2\gamma^3}\rho_2^{-1}(L)\ell^2 + \frac{\rho_2^{-2}(L)}{2}\ell^2 + \frac{\ell^2}{2},\\
	% &\varphi_5=\frac{1}{4}-\frac{(\beta+\gamma)^2}{\gamma^5}\rho_2^{-1}(L)\ell^2
	% 	-\frac{1}{2\gamma^2}\rho_2^{-2}(L)\ell^2,\\
	% &\varphi_6=\frac{(\beta+\gamma)}{2\gamma^3}\rho_2^{-1}(L)\ell^2 + \frac{1}{2\gamma^2}\rho_2^{-2}(L)\ell^2 + \frac{\rho_2^{-2}(L)}{2}\ell^2  \\
	% 	&\quad + \frac{\ell^2}{2}
	% 	+\frac{\ell}{2},\\
	&\bsM_1=\frac{1}{2}\big(\beta\bsL - (3\gamma+2+3\ell^2)\bsE\big),\\
	&\bsM_2=3\beta^2\bsL^2 - (\beta\gamma - \gamma^2)\bsL 
		+ \Big(2\gamma^2 + 1 + \frac{5}{2}\ell^2 \Big)\bsE.
	% &\bsM_2^\prime=3\beta^2\bsL^2 - (\beta\gamma - \gamma^2)\bsL \nonumber\\
	%  	&\quad + \Big(2\gamma^2 + 1 + \frac{5}{2}\ell^2 + 2\varphi_6\ell^2 \Big)\bsE.
		\end{align*}
% Now it is ready to prove Lemma~\ref{Lemma: Lyapunov Difference}.
 
\noindent {\bf (i)}  
This step shows the relation between $\mathcal{L}_{1,k+1}$ and $\mathcal{L}_{1,k}$.
	% \begin{lemma}
	% 	\begin{align}\label{L1:iteration}
	% 		&\mathcal{L}_{1,k+1}\le \mathcal{L}_{1,k} - \frac{\alpha}{4}\|\bar{\bsg}_{k}^0\|^2 + \alpha n\tilde{d}^2\psi_2\max_{i\in[n]}\|x_{i,k}-\hat{x}_{i,k}\|^2_p
	% 	\end{align}
	% \end{lemma}
		% Similar to the proof of \eqref{nonconvex:vkLya4_determin}, by 
		Combining \eqref{nonconvex:v1k}--\eqref{nonconvex:v4k} yields
		\begin{align}\label{nonconvex:vkLya_xhat_determin:L2}
			&\mathcal{L}_{1,k+1}
			\le  \mathcal{L}_{1,k}-\|\bsx_k\|^2_{\frac{\alpha\beta}{2}\bsL-\frac{\alpha}{2}\bsE
					-\frac{\alpha}{2}(1+3\alpha)\ell^2\bsE - 3\alpha^2\beta^2\bsL^2}
					\nonumber\\
				&\quad	+\Big\|\bm{v}_k+\frac{1}{\gamma}\bsg_k^0\Big\|^2_{\frac{6\alpha^2\gamma^2\rho(L)+\alpha\gamma}{4}\bsF}\nonumber\\
				&\quad+ \big(\frac{\alpha}{2}(\beta+2\gamma)\rho(L) + 3\alpha^2\beta^2\rho^2(L)\big)\|\bsx_k-\hat{\bsx}_k\|^2 \nonumber\\
			&\quad + \big(\alpha^2\gamma(\beta+\gamma)\rho(L)+\alpha^2\big)\|\bm{x}_k-\hat{\bm{x}}_k\|^2
				\nonumber\\
				&\quad +\Big\|\bm{v}_k+\frac{1}{\gamma}\bsg_{k}^0\Big\|^2_{\frac{\alpha\gamma}{4}\bsF} + \|\bm{x}_k\|^2_{\alpha^2\gamma(\beta+\gamma)\bsL + \alpha^2\bsE} \nonumber\\
				&\quad + \alpha^2 \tau_2\ell^2\|\bar{\bsg}_k\|^2  + \Big(\alpha^2\tau_3\ell^2+\frac{\alpha}{8}\Big)\|\bar{\bsg}_k\|^2\nonumber\\
			&\quad +\|\bm{x}_k\|^2_{\frac{\alpha(3\gamma+1)}{2}\bsE + 2\alpha^2(\gamma^2\bsE-\beta\gamma\bsL)+\frac{\alpha}{2}(1+2\alpha)\ell^2\bsE} \nonumber\\
				&\quad +\big(\alpha\gamma + 2\alpha^2(\gamma^2-\beta\gamma\rho_2(L))\big)\|\bm{x}_k-\hat{\bm{x}}_k\|^2 \nonumber\\
				&\quad-\Big\|\bm{v}_k+\frac{1}{\gamma}\bsg_{k}^0\Big\|^2_{\alpha(\gamma-\frac{5}{2}\rho^{-1}_2(L))\bsF
					-\frac{\alpha^2}{2}(\gamma^2+1)\rho(L)\bsF} \nonumber\\
			&\quad -\frac{\alpha}{4}(1-2\alpha \ell)\|\bar{\bsg}_{k}\|^2
				+\|\bsx_k\|^2_{\frac{\alpha}{2}\ell^2\bsE} 
				-\frac{\alpha}{4}\|\bar{\bsg}_{k}^0\|^2 \nonumber\\
			&=\mathcal{L}_{1,k}-\|\bsx_k\|^2_{\alpha\bsM_1-\alpha^2\bsM_2} 
				- \Big\|\bm{v}_k+\frac{1}{\gamma}\bsg_k^0\Big\|^2_{\alpha(\varphi_3-\alpha\varphi_4)\bsF} \nonumber\\
				&\quad 	-\alpha(\varphi_5-\alpha\varphi_6)\|\bar{\bsg}_{k}\|^2 + \alpha(\varphi_7+\alpha\varphi_8)\|\bsx_{k}-\hat{\bsx}_{k}\|^2 \nonumber\\
				&\quad 	-\frac{\alpha}{4}\|\bar{\bsg}_{k}^0\|^2 \nonumber\\
			&\le \mathcal{L}_{1,k}-\|\bsx_k\|^2_{\alpha(\varphi_1-\alpha\varphi_2)\bsE}
				- \Big\|\bm{v}_k+\frac{1}{\gamma}\bsg_k^0\Big\|^2_{\alpha(\varphi_3-\alpha\varphi_4)\bsF} \nonumber\\
				&\quad 	-\alpha(\varphi_5-\alpha\varphi_6)\|\bar{\bsg}_{k}\|^2 + \alpha(\varphi_7+\alpha\varphi_8)\|\bsx_{k}-\hat{\bsx}_{k}\|^2 \nonumber\\
				&\quad 	-\frac{\alpha}{4}\|\bar{\bsg}_{k}^0\|^2, 
		\end{align}
		where the last inequality holds due to \eqref{nonconvex:KL-L-eq2}.

		Next, we show that the constants appearing on the right-hand side of \eqref{nonconvex:vkLya_xhat_determin:L2} are positive.
		% From $\beta=\tau_1\gamma$, $\tau_1>\kappa_1\ge\frac{4}{\rho_2(L)}$, and $\gamma>\kappa_2\ge\max\{2+2\ell^2,~\sqrt[3]{\frac{8\ell^2(\kappa_1+1)^2}{\rho_2(L)}},~\frac{2\ell}{\rho_2(L)}\}$, we have
		Recalling $\beta=\tau_1\gamma$, $\tau_1>\kappa_1\ge4/\rho_2(L)$, and $\gamma>\kappa_2\ge 2+2\ell^2$, we have
		\begin{align}\label{nonconvex:beta1}
			&\varphi_1=\frac{1}{2}\big(\rho_2(L)\beta-(3\gamma+2+2\ell^2)\big) \nonumber\\
			&\quad >\frac{1}{2}\rho_2(L)\kappa_1\gamma-\frac{1}{2}(3\gamma+2+2\ell^2)\nonumber\\
			&\quad >\frac{1}{2}\rho_2(L)\kappa_1\gamma-2\gamma \ge0.
		\end{align}
		% From $\beta=\tau_1\gamma$, $\tau_1>\kappa_1\ge\frac{4}{\rho_2(L)}$, and $\gamma>\kappa_2\ge2+3\ell^2$, we have
		% \begin{align}\label{nonconvex:beta1:L2}
		% 	&\varphi_1=\frac{1}{2}\big(\rho_2(L)\beta-(3\gamma+2+3\ell^2)\big) \nonumber\\
		% 	&\quad>\frac{1}{2}\gamma\big( \rho_2(L)\kappa_1-4 \big) \ge0.
		% \end{align}
		Similarly, from $\gamma>\kappa_2\ge5/\rho_2(L)$,
		\begin{align}\label{nonconvex:beta2}
			\varphi_3=\frac{\gamma}{2}-\frac{5}{2}\rho^{-1}_2(L)>0.
		\end{align}
		From $\gamma>\kappa_2\ge\max\{\sqrt[3]{\frac{16\ell^2(\kappa_1+1)^2}{\rho_2(L)}},~\frac{2\sqrt{2}\ell}{\rho_2(L)}\}$,
		\begin{align}
			&\varphi_5=\frac{1}{8}-\frac{(\beta+\gamma)^2}{\gamma^5}\rho_2^{-1}(L)\ell^2
				-\frac{1}{2\gamma^2}\rho_2^{-2}(L)\ell^2 \nonumber\\
			&> \frac{1}{8}-\frac{(\kappa_1+1)^2}{\gamma^3}\rho_2^{-1}(L)\ell^2
				-\frac{1}{2\gamma^2}\rho_2^{-2}(L)\ell^2 > 0.
		\end{align}
		% Given $\alpha=\frac{1}{n^{\frac{1}{4}}\tilde{d}\sqrt{T}}$ and $T>\tilde{\kappa}_3\ge\frac{1}{n^{\frac{1}{2}}\tilde{d}^2\kappa_7^2}$,
		%  and $\kappa_7\le\min\{\frac{\varphi_1}{\varphi_2},\frac{\varphi_3}{\varphi_4},\frac{\varphi_5}{\varphi_6}\}$, 
		Given $\alpha<\hat{\kappa}_0=\min\{\frac{\varphi_1}{\varphi_2},\frac{\varphi_3}{\varphi_4},\frac{\varphi_5}{\varphi_6}\}$,
		\begin{subequations}
			\begin{align}
				&\alpha(\varphi_1-\alpha\varphi_2)>0,\label{nonconvex:vkLya1.1p_quantization}\\
				&\alpha(\varphi_3-\alpha\varphi_4)>0,\label{nonconvex:vkLya1.3p_quantization}\\
				&\alpha(\varphi_5-\alpha\varphi_6)>0,\label{nonconvex:vkLya1.2p_quantization}
			\end{align}
		\end{subequations}
	
		% From Lemma~\ref{nonconvex:lemma:pnorm}, we have
		By norm equivalence, we obtain
		\begin{align}\label{nonconvex:xminusxhat_determin}
			&\|\bsx_{k}-\hat{\bsx}_{k}\|^2=\sum_{i=1}^{n}\|x_{i,k}-\hat{x}_{i,k}\|^2 \nonumber\\
			&\le\sum_{i=1}^{n}\tilde{d}^2\|x_{i,k}-\hat{x}_{i,k}\|^2_p\le n\tilde{d}^2\max_{i\in[n]}\|x_{i,k}-\hat{x}_{i,k}\|^2_p.
		\end{align}		
		Combining \eqref{nonconvex:vkLya_xhat_determin:L2}--\eqref{nonconvex:xminusxhat_determin} gives \eqref{L1:iteration}.

\noindent {\bf (ii)}  
	In~\eqref{L1:iteration}, the compression error $\|x_{i,k}-\hat{x}_{i,k}\|_p^2$ depends on $\|x_{i,k}-\hat{x}_{i,k-1}\|_p^2$ for both the locally- and globally-bounded compressors, as shown in~\eqref{local:contraction} and~\eqref{global:contraction}.
	% Therefore, in~\eqref{c:iteration} we further analyze the one-step difference of $\|x_{i,k}-\hat{x}_{i,k-1}\|_p^2$.
% As shown in~\eqref{local:contraction}, the compression error term $\max_{i\in[n]}\|x_{i,k}-\hat{x}_{i,k}\|_p^2$ in~\eqref{L1:iteration}
% is controllable under the locally-bounded compressors only when the input stays in the local region, i.e., $\max_{i\in[n]}\|x_{i,k}-\hat{x}_{i,k-1}\|_p\le Cs_k$.
Accordingly, we also track the one-step difference of $\max_{i\in[n]}\|x_{i,k}-\hat{x}_{i,k-1}\|_p$. 
% This step aims to establish an upper bound for 
% $\max_{i\in[n]}\|x_{i,k}-\hat{x}_{i,k}\|_p^2$.
% Noting that 
% $\max_{i\in[n]}\|x_{i,k}-\hat{x}_{i,k}\|_p^2
% = \max_{i\in[n]}\|x_{i,k}-\hat{x}_{i,k-1}
% -\omega s_k \mathcal{C}(x_{i,k}-\hat{x}_{i,k-1})\|_p^2$,
% we therefore focus on analyzing the iteration of 
% \(\max_{i\in[n]}\|x_{i,k}-\hat{x}_{i,k-1}\|_p^2\). We have
			\begin{align}\label{nonconvex:xminush_compress_determin}
			&\|x_{i,k+1}-\hat{x}_{i,k}\|^2_p=\|x_{i,k+1}-x_{i,k}+x_{i,k}-\hat{x}_{i,k}\|^2_p\nonumber\\
			&\le(\|x_{i,k+1}-x_{i,k}\|_p+\|x_{i,k}-\hat{x}_{i,k}\|_p)^2\nonumber\\
			&\le(1+\varepsilon_5^{-1})\|x_{i,k+1}-x_{i,k}\|_p^2+(1+\varepsilon_5)\|x_{i,k}-\hat{x}_{i,k}\|_p^2\nonumber\\
			&\le(1+\varepsilon_5^{-1})\hat{d}^2\|x_{i,k+1}-x_{i,k}\|^2+(1+\varepsilon_5)\|x_{i,k}-\hat{x}_{i,k}\|_p^2\nonumber\\
			&\le(1+\varepsilon_5^{-1})\hat{d}^2\|\bsx_{k+1}-\bsx_{k}\|^2+(1+\varepsilon_5)\|x_{i,k}-\hat{x}_{i,k}\|_p^2,
			\end{align}
			% where the first inequality holds due to the Minkowski inequality; the second inequality holds due to the Cauchy--Schwarz inequality and $\varepsilon_5>0$; and the third inequality holds due to Lemma~\ref{nonconvex:lemma:pnorm}.
			where the first inequality holds due to the Minkowski inequality; the second inequality holds due to the Cauchy--Schwarz inequality and $\varepsilon_5>0$; and the third inequality holds due to norm equivalence.
			
			% For the first term on the right-hand side of \eqref{nonconvex:xminush_compress_determin}, we have
			% \begin{align}\label{nonconvex:xkoneminusx}
			% 	&\|\bsx_{k+1}-\bsx_k\|^2 = \alpha^2\|\beta\bsL\hat{\bsx}_k+\gamma\bm{v}_k+\bsg_k\|^2\nonumber\\
			% 	&=\alpha^2\|\beta\bsL(\hat{\bsx}_k-\bsx_k)+\beta\bsL\bsx_k+\gamma\bm{v}_k
			% 	+\bsg_k^0+\bsg_k-\bsg_k^0\|^2\nonumber\\
			% 	&\le4\alpha^2(\beta^2\|\bsx_k-\hat{\bsx}_k\|^2_{\bsL^2}+\beta^2\|\bsx_k\|^2_{\bsL^2} \nonumber\\
			% 		&\quad + \|\gamma\bm{v}_k +\bsg_k^0\|^2+\|\bsg_k-\bsg_k^0\|^2)\nonumber\\
			% 	&\le4\alpha^2\Big(\beta^2\rho^2(L)\|\bsx_k-\hat{\bsx}_k\|^2+\|\bsx_k\|^2_{(\beta^2\rho^2(L)+\ell^2)\bsE} + \nonumber\\
			% 		&\quad \Big\|\bm{v}_k+\frac{1}{\gamma}\bsg_k^0\Big\|^2_{\gamma^2\rho(L)\bsF}\Big),
			% \end{align}
			% where the first equality holds due to \eqref{nonconvex:kia-algo-dc-compact-x}; the first inequality holds due to the Cauchy--Schwarz inequality; and the second inequality holds due to \eqref{nonconvex:KL-L-eq2}, \eqref{nonconvex:lemma-eq5}, and \eqref{nonconvex:gg1}.

			Combining \eqref{nonconvex:xminusxhat_determin}, \eqref{nonconvex:xminush_compress_determin}, and \eqref{nonconvex:xkoneminusx} 
			% and \eqref{nonconvex:xminusxhat_determin}
			% and \eqref{hatx:le} 
			leads to
			\begin{align}\label{nonconvex:xminush_compress_determin2:0}
				&\max_{i\in[n]}\|x_{i,k+1}-\hat{x}_{i,k}\|^2_p \nonumber\\
				&\le(1+\varepsilon_5+\alpha^2 n\tilde{d}^2\psi_3)
				\max_{i\in[n]}\|x_{i,k}-\hat{x}_{i,k}\|_p^2 \nonumber\\
				&\quad +\alpha^2\psi_4\varepsilon_1\Big(\|\bsx_k\|^2_{\bsE}
				+\Big\|\bm{v}_k+\frac{1}{\gamma}\bsg_k^0\Big\|^2_{\bsF}\Big).
				% &\le(1+\varepsilon_5+\alpha^2 n\tilde{d}^2\psi_3)
				% \max_{i\in[n]}\|x_{i,k}-\hat{x}_{i,k}\|_p^2+\alpha^2\psi_4\mathcal{L}_{1,k}.
			\end{align}
			% where the first inequality holds due to \eqref{nonconvex:xkoneminusx}, \eqref{nonconvex:xminusxhat_determin}, and \eqref{nonconvex:xminush_compress_determin}; and the second inequality holds due to \eqref{nonconvex:vkLya3.2_determin}.
			Denote	$\hat{\mathcal{L}}_{1,k}=\|\bsx_k\|^2_{\bsE}+\Big\|\bsv_k+\frac{1}{\gamma}\bsg_k^0\Big\|^2_{\bsF} + \tilde{f}(\bar{\bsx}_{k})-nf^*$,
			then
			\begin{align}
				&\mathcal{L}_{1,k}
				\ge\frac{1}{2}\|\bsx_{k}\|^2_{\bsE}
					+\frac{1}{2}\Big(1+\frac{\beta}{\gamma}\Big)\Big\|\bsv+\frac{1}{\gamma}\bsg^0\Big\|^2_{\bsF} \nonumber\\
					&\quad-\frac{\gamma}{2\beta\rho_2(L)}\|\bsx_{k}\|^2_{\bsE}
					-\frac{\beta}{2\gamma}\Big\|\bsv+\frac{1}{\gamma}\bsg^0\Big\|^2_{\bsF} + \tilde{f}(\bar{\bsx}_{k})-nf^* \nonumber\\
				% &\ge\varepsilon_6\Big(\|\bsx_{k}\|^2_{\bsE}
				% 	+\Big\|\bsv+\frac{1}{\gamma}\bsg^0\Big\|^2_{\bsF}\Big) \nonumber\\
				&\ge\varepsilon_1\hat{\mathcal{L}}_{1,k}\ge0,\label{hatx:le}
			\end{align}
			where the first inequality holds due to the definition of $\mathcal{L}_{1,k}$, 
			the Cauchy--Schwarz inequality, and \eqref{nonconvex:lemma-eq5}; and the last inequality holds due to $0<\varepsilon_1<\frac{1}{2}$. Similarly, one has
			\begin{align}\label{hatx:ge}
				\mathcal{L}_{1,k}\le\varepsilon_2\hat{\mathcal{L}}_{1,k},\quad \varepsilon_2>1.
			\end{align}
			Finally, \eqref{c:iteration} follows form \eqref{nonconvex:xminush_compress_determin2:0} and \eqref{hatx:le}.

\subsection{Proof of Lemma~\ref{Lemma: PL}} \label{appendix:lemma:localpl}

We denote the following constants:
	\begin{align*}
&\tilde{\kappa}_0=\min\Big\{\hat{\kappa}_0,
	~\kappa_5,~\frac{\kappa_6}{\sqrt{n}\tilde{d}}\Big\},\\
&\kappa_5=\begin{cases}
	\text{the smaller positive root of }~1=\\
	\quad \alpha\min\{\varphi_1-\alpha\varphi_2,\;
	\varphi_3-\alpha\varphi_4\}, & \text{if it exists},\\[4pt]
	+\infty, & \text{otherwise,}
	\end{cases}\\
&\kappa_6=\sqrt{\frac{\varepsilon_5+2\varepsilon_5^2}{(1-2\varepsilon_5)\psi_3}},\\
% &\kappa_\nu=\frac{1}{2}\|\bm{x}_0 \|^2_{\bsK}+\frac{1}{2}\Big\|\bsv_0
% 	+\frac{1}{\beta}\bsg_0^0\Big\|^2_{\frac{\alpha+\beta}{\beta}\bsF} \\
% 	&\quad +\bsx_0^\top\bsK\bsF\Big(\bm{v}_0+\frac{1}{\beta}\bsg_0^0\Big)
% 	+\frac{1}{2\nu}\|\bar{\bsg}_0^0\|^2,\\
% &\psi_5>\frac{ (1-2\varepsilon_5)\psi_2}{\varepsilon_6},\\
% &\epsilon\in(\max\{\kappa_{9},~\kappa_{10}\},1),\\
% &\kappa_{9}=\sqrt{1-\alpha\big(\varepsilon_6- (1-2\varepsilon_5)\psi_2\frac{1}{\psi_5}\big)},\\
% &\kappa_{10}=\sqrt{1-(\varepsilon_5+2\varepsilon_5^2)+\alpha^2n\tilde{d}^2\big((1-2\varepsilon_5)\psi_3+\psi_4\psi_5\big)},\\
&\varepsilon_6=\min\{\nu/2,~\varepsilon_3\}/\varepsilon_2,\\
&\varepsilon_7=\varepsilon_5+2\varepsilon_5^2
		-\alpha^2n\tilde{d}^2(1-2\varepsilon_5)\psi_3.
% &\psi_5>\frac{ (1-2\varepsilon_5)\psi_2}{\varepsilon_6}.
\end{align*}

From Assumption~\ref{nonconvex:ass:fil}, we have
\begin{align} \label{pl}
	\|\bar{\bsg}_k^0\|^2=n\|\nabla f(\bar{x}_k)\|^2\ge2n\nu(f(\bar{x}_k)-f^*).
\end{align}
% and  $\mathcal{L}_{1,0}\le\kappa_\nu$.
% Then for $k=0$, the inequalities hold since $s_0\ge\max\{\sqrt{\frac{\kappa_\nu}{n\tilde{d}^2\psi_5C^2}},~\frac{\max_{i\in[n]}\|x_{i,0}\|_p}{C}\}$.
% 	Suppose that the statement holds for $k = 0,1,\ldots,\tau$, 
% 	we next show that the inequality also holds for $k = \tau+1$.
Noting $\alpha<\tilde{\kappa}_0\le\kappa_5$ and $\varepsilon_2>1$,
\begin{align} \label{epsilon:01}
	\hspace{-0.4em} \alpha\varepsilon_6\le \frac{\alpha\varepsilon_3}{\varepsilon_2}=\frac{\alpha\min{\{\varphi_1-\alpha\varphi_2,~\alpha\varphi_3-\alpha\varphi_4}\}}{\varepsilon_2}<1.
\end{align}
% From Assumption~\ref{nonconvex:ass:fil}, we have
% \begin{align} \label{pl}
% 	\|\bar{\bsg}_k^0\|^2=n\|\nabla f(\bar{x}_k)\|^2\ge2\nu(f(\bar{x}_k)-f^*).
% \end{align}
Substituting \eqref{pl}, \eqref{local:contraction} into \eqref{L1:iteration} 
	and using \eqref{hatx:ge}, \eqref{epsilon:01} 
	with the assumption \eqref{ass:induction}, we obtain
\begin{align} \label{induction:cor:1:0}
&\mathcal{L}_{1,k+1} \le \mathcal{L}_{1,k} - \frac{\nu\alpha}{2}n(f(\bar{x}_k)-f^*) \nonumber\\
	&\quad + \alpha n\tilde{d}^2(1-2\varepsilon_5)\psi_2C^2s_k^2 \nonumber\\
	&\quad -\alpha\varepsilon_3\big(\|\bsx_{k}\|^2_{\bsE} + \big\|\bsv_{k}+\tfrac{1}{\gamma}\bsg_{k}^0\big\|^2_{\bsF}\big) \nonumber\\
&\quad \le (1-\alpha\varepsilon_6)\mathcal{L}_{1,k} + \alpha n\tilde{d}^2(1-2\varepsilon_5)\psi_2C^2s_k^2. 
% &\quad \le (1-\alpha\varepsilon_6)n\tilde{d}^2\psi_5C^2s_k^2 + \alpha n\tilde{d}^2(1-2\varepsilon_5)\psi_2C^2s_k^2 \nonumber\\
% &\quad \le \Big(1-\alpha\big(\varepsilon_6- (1-2\varepsilon_5)\psi_2\frac{1}{\psi_5}\big)\Big)\frac{n\tilde{d}^2\psi_5C^2s_{k+1}^2}{\epsilon^2} \nonumber\\
% &\quad \le n\tilde{d}^2\psi_5C^2s_{k+1}^2,
\end{align}
% where the last inequality holds due to $\psi_5>\frac{ (1-2\varepsilon_5)\psi_2}{\varepsilon_6}$ and $\epsilon>\kappa_{9}$.
% Next, we prove the second part of the induction.
Combining \eqref{c:iteration}, \eqref{local:contraction} and \eqref{ass:induction} yields
\begin{align} \label{induction:cor:2:0}
	&\max_{i\in[n]}\|x_{i,k+1}-\hat{x}_{i,k}\|_p^2\nonumber\\
	&\le (1+\varepsilon_5  +\alpha^2 n\tilde{d}^2\psi_3)(1-2\varepsilon_5)C^2s_k^2 
	 +\alpha^2\psi_4\mathcal{L}_{1,k} \nonumber\\
	&=\big(1-(\varepsilon_5+2\varepsilon_5^2) 
		+\alpha^2n\tilde{d}^2(1-2\varepsilon_5)\psi_3)\big)C^2s_{k}^2 \nonumber\\
	&\quad + \alpha^2\psi_4\mathcal{L}_{1,k}, 
\end{align}
which gives \eqref{c:iteration:induction}
 due to $\alpha<\tilde{\kappa}_0\le\kappa_6/(\sqrt{n}\tilde{d})$.

\subsection{Constants Used in Lemma~\ref{Lemma:s0}} \label{appendix:lemma2}

% In addition to those denoted in Appendix\ref{appendix:lemmas}, 
% We  denote the following constants:
\begin{align*}
	% &\tilde{\kappa}^{\prime}_0(T)=\min\Big\{\frac{\varphi_1}{\varphi_2},~\frac{\varphi_3}{\varphi_4},
	% 		~\kappa_6,~\frac{\varphi_1}{\varphi_2},~\frac{\varphi_5}{\varphi_6},~\kappa_{8}(T)\Big\},\\
	&\tilde{\kappa}_3=\max\Big\{\frac{1}{n^{\frac{1}{2}}\tilde{d}^2\kappa_7^2},
		~\frac{(1-2\varepsilon_5)\psi_3}{\varepsilon_8}n^{\frac{1}{2}}, \nonumber\\
		&\quad \frac{2\psi_4\mathcal{L}_{1,0}}{\varepsilon_8C^2s_0^2n}\frac{n^{\frac{1}{2}}}{\tilde{d}^2},
		~\frac{4(1-2\varepsilon_5)^2\psi_2^2\psi_4^2}{\varepsilon_8^2 }\frac{n^{\frac{1}{2}}}{\tilde{d}^2},
		~\frac{\hat{\kappa}_3\tilde{d}^2}{n^\frac{1}{2}}\Big\},\\
	&\hat{\kappa}_3>0,\\
	&\tilde{\kappa}_4(T) =   \psi_4\mathcal{L}_{1,0}  \alpha^2/C^2   
		+  (1-2\varepsilon_5)\psi_2\psi_4n\tilde{d}^2s_0^2 \alpha^3T, \\
	&\tilde{\kappa}^{\prime}_0(T)=\min\Big\{\kappa_7,~\kappa_{8}(T)\Big\},\\
	&\kappa_7=\min\Big\{\hat{\kappa}_0,~1/(2\ell)\Big\},\\
	% &\tilde{\kappa}^{\prime}_0(T)=\min\Big\{\frac{\varphi_1}{\varphi_2},~\frac{\varphi_3}{\varphi_4},
	% 	~\frac{\varphi_5}{\varphi_6},~\frac{1}{2\ell},~\kappa_{8}(T)\Big\},\\
	% &\tilde{\kappa}^{\prime}_0(T)=\min\Big\{\kappa_7,~\kappa_{8}(T)\Big\},\\
	% &\kappa_6=\begin{cases}
	% 	\text{the smaller positive root of }~1=\\
	% 	\quad \alpha\min\{\varphi_1-\alpha\varphi_2,\;
	% 	\alpha\varphi_3-\alpha\varphi_4\}, & \text{if it exists},\\[4pt]
	% 	+\infty, & \text{otherwise,}
	% 	\end{cases}\\
	&\kappa_{8}(T)=\min\Big\{\Big(\frac{\varepsilon_8}{(1-2\varepsilon_5)\psi_3 n\tilde{d}^2}\Big)^{\frac{1}{2}}, \Big(\frac{\varepsilon_8C^2s_0^2}{2\psi_4\mathcal{L}_{1,0} }\Big)^{\frac{1}{2}},\\
				&\quad	\Big(\frac{\varepsilon_8}{2(1-2\varepsilon_5)\psi_2\psi_4 n\tilde{d}^2}\Big)^{1/3}\frac{1}{T^{1/3}}	\Big\},\\
	&\varepsilon_8=(\varepsilon_5 + 2\varepsilon_5^2)/2.
\end{align*}

\subsection{Proof of Theorem~\ref{nonconvex:thm-R1}} \label{appendix:thm1}

			Denote $\varepsilon_4=\min\{\varepsilon_3,~1/4\}$.
			Under the setting of Theorem~\ref{nonconvex:thm-R1}, 
			it is clear that all the conditions in Lemmas~\ref{Lemma: Lyapunov Difference} and \ref{Lemma:s0} are satisﬁed.
			% Finally, we are ready to establish the convergence rate.
			Summing \eqref{L1:iteration} over $k \in [0,T-1]$, combined with \eqref{c:1}, \eqref{bound:sk} and \eqref{thm1:set}, yields	
			\begin{align}
				&\frac{1}{T}\sum_{k=0}^{T-1}\Big(\|\nabla f(\bar{x}_k)\|^2+\frac{1}{n}\sum_{i=1}^{n}\|x_{i,k}-\bar{x}_k\|^2\Big)\nonumber\\
				&=\frac{1}{nT}\sum_{k=0}^{T-1}[\|\bar{\bsg}_{k}^0\|^2 + \|\bsx_k\|_{\bsE}^2]\nonumber\\
				&\le \frac{\mathcal{L}_{1,0}}{\varepsilon_4nT\alpha}
				 	+ \frac{(1-2\varepsilon_5)\psi_2\tilde{d}^2}{\varepsilon_4T}\sum_{k=0}^{T-1}\Big( (1-\varepsilon_8)^k C^2s_0^2 \nonumber\\
					&\quad + \frac{\psi_4\mathcal{L}_{1,0}}{\varepsilon_8n} \frac{n^{\frac{1}{2}}}{\tilde{d}^2T}   	+ \frac{(1-2\varepsilon_5)\psi_2\psi_4}{\varepsilon_8 }C^2s_0^2  \frac{n^{\frac{1}{4}}}{\tilde{d}\sqrt{T}}
							\Big)  \nonumber\\
				&\le \Big(\frac{\mathcal{L}_{1,0}}{\varepsilon_4n} + \frac{(1-2\varepsilon_5)^2\psi_2^2\psi_4}{\varepsilon_4\varepsilon_8}C^2s_0^2\Big)\frac{n^{\frac{1}{4}}\tilde{d}}{\sqrt{T}} \nonumber\\
					&\quad +\mathcal{O}\Big(\frac{n^{\frac{1}{2}}}{T}\Big)
					+ \mathcal{O}\Big(\frac{\tilde{d}^2}{T}\Big), \label{zerosg:thm-sg-sm-equ4p}
			\end{align}
			where the first inequality follows from the definition of $\varepsilon_4$.
			Recalling $\mathcal{L}_{1,0}=\mathcal{O}(n)$, \eqref{zerosg:thm-sg-sm-equ4p} gives \eqref{nonconvex:thm-sm-equ1}.

			Since $\alpha=\frac{1}{n^{1/4}\tilde{d}\sqrt{T}}$ and 
			$T>\tilde{\kappa}_3\ge\frac{1}{n^{1/2}\tilde{d}^2\kappa_7^2}$ imply 
			$\alpha \le \kappa_7 \le \frac{1}{2\ell}$, 
			we obtain $\frac{\alpha}{4}(1-2\alpha\ell)>0$. 
			Summing~\eqref{nonconvex:v4k} over $k \in [0,T-1]$ and using \eqref{zerosg:thm-sg-sm-equ4p} gives
			\begin{align}\label{nonconvex:thm-sm-equ2:0}
				&f(\bar{x}_T)-f^*=\frac{e_{4,T}}{n}\le \frac{e_{4,0}}{n} + \frac{\ell^2}{2}\frac{1}{n^{\frac{1}{4}}\tilde{d}\sqrt{T}}\Big( \big(\frac{\mathcal{L}_{1,0}}{\varepsilon_4n}+ \nonumber\\
				&\quad  \frac{(1-2\varepsilon_5)^2\psi_2^2\psi_4}{\varepsilon_4\varepsilon_8}C^2s_0^2\big)n^{\frac{1}{4}}\tilde{d}\sqrt{T}
					+ \mathcal{O}\big(n^{\frac{1}{2}}\big) + \mathcal{O}\big(\tilde{d}^2\big)\Big),
			\end{align}
			which futher implies \eqref{nonconvex:thm-sm-equ2} since $e_{4,0}=\mathcal{O}(n)$ and $T>\tilde{\kappa}_3=\max\big\{\mathcal{O}\big(\frac{n^{\frac{1}{2}}}{\tilde{d}^2}\big),~\mathcal{O}\big(\frac{\tilde{d}^2}{n^{\frac{1}{2}}}\big)\big\}$.
		
		\subsection{Proof of Theorem~\ref{nonconvex:cor-R1}} \label{appendix:cor1}
			% In addition to the notations defined in Appendix~\ref{appendix:lemmas} and \ref{appendix:thm1},
			We denote the following constants:
			\begin{align*}
				&\kappa_0=\Big(\frac{\varepsilon_8}{2(1-2\varepsilon_5)\psi_2\psi_4 }\Big)^{1/3},\\
				&\kappa_3=\max\Big\{\frac{\tau_0^3}{n\tilde{d}^2\kappa_7^3},
					~\Big(\frac{(1-2\varepsilon_5)\psi_3\tau_0^2}{\varepsilon_8}\Big)^{\frac{3}{2}}n^{\frac{1}{2}}\tilde{d} \Big\},\\
				&\kappa_4= \frac{2\psi_4\mathcal{L}_{1,0} }{C^2\varepsilon_8n}.
			\end{align*}
			% If the first communication round is uncompressed, 
			% i.e., 

			If the first-round compression error is controllable and can approach zero, 
			% $s_0$ is controllable and can approach zero, 
			the overall convergence rate can be further enhanced.
			In the analysis, an uncompressed first communication round can be equivalently viewed as the initialization
			% Define the  sequence $\{s_k\}$ by
			% \begin{align}
			% 	&s_{k+1}^2 = \big(1-\frac{\varepsilon_5+2\varepsilon_5^2}{2}\big)s_k^2  + \psi_4\mathcal{L}_{1,0}  \alpha^2  \nonumber\\
			% 		&+  (1-2\varepsilon_5)\psi_2\psi_4n\tilde{d}^2s_0^2 \alpha^3T,
			% 		~\forall k\in[0,T-1], \label{set:sk:cor}\\
			% 	&s_0=\tau_4n\alpha^2,~\alpha=\frac{\tau_0}{n^{\frac{1}{3}}\tilde{d}^{\frac{2}{3}}T^{\frac{1}{3}}},~\tau_4\le \kappa_4,~\tau_0\le\kappa_0, 		
			% \end{align}
			% For simplicity of analysis, we equivalently treat an uncompressed first communication round as the initialization
			\begin{align}
			\hat{x}_{i,-1}=x_{i,0},\qquad y_{i,-1}=\sum\nolimits_{j=1}^{n}L_{ij}x_{j,0}.
			\end{align}
			Based on the parameter setting $\alpha=\tau_0/(n^{\frac{1}{3}}\tilde{d}^{\frac{2}{3}}T^{\frac{1}{3}})$, $T>\kappa_3$, $\tau_0\le\kappa_0$, $s_0^2=\tau_4 n\alpha^2$ and $\tau_4\ge \kappa_4$, we have $\alpha\in(0,\tilde{\kappa}^{\prime}_0(T))$.
			% Based on the parameter setting $\alpha=\tau_0/(n^{\frac{1}{3}}\tilde{d}^{\frac{2}{3}}T^{\frac{1}{3}})$, $T>\kappa_3$, $\tau_0\le\kappa_0$ and $\tau_4\ge \kappa_4$,
			% we have $\alpha\in(0,\tilde{\kappa}^{\prime}_0(T))$.
			Then, similar to the proof of Lemma~\ref{Lemma:s0}, we know that \eqref{bound:sk0} holds.
			% \begin{align} \label{bound:sk:max:cor}
			% 		&\max_{i\in[n]}\|x_{i,k}-\hat{x}_{i,k-1}\|_p^2\le s_k^2 \nonumber\\
			% 	&\le (1-\varepsilon_8)^k C^2s_0^2  	
			% 		+ \frac{\psi_4\mathcal{L}_{1,0}}{\varepsilon_8n} n\alpha^2 \nonumber\\
			% 		&\quad 	+ \frac{(1-2\varepsilon_5)\psi_2\psi_4}{\varepsilon_8 } n\tilde{d}^2 C^2s_0^2 \alpha^3 T.  
			% \end{align}

			Similar to the proof of \eqref{zerosg:thm-sg-sm-equ4p}, 
			% taking the expectation of \eqref{L1:iteration} 
			% with respect to $\mathcal{F}_T$ and 
			summing \eqref{L1:iteration} over $k \in [0,T-1]$, 
			together with \eqref{c:1}, \eqref{bound:sk0} and \eqref{cor1:set}, yields
			\begin{align} \label{nonconvex:cor-sm-eq1:0}
				&\frac{1}{T}\sum_{k=0}^{T-1}\Big(\|\nabla f(\bar{x}_k)\|^2+\frac{1}{n}\sum_{i=1}^{n}\|x_{i,k}-\bar{x}_k\|^2\Big)\nonumber\\
				&=\frac{1}{nT}\sum_{k=0}^{T-1}[\|\bar{\bsg}_{k}^0\|^2 + \|\bsx_k\|_{\bsE}^2]\nonumber\\
				&\le \frac{\mathcal{L}_{1,0}}{\varepsilon_4nT\alpha}
				 	+ \frac{(1-2\varepsilon_5)\psi_2\tilde{d}^2}{\varepsilon_4T}\sum_{k=0}^{T-1}\Big( (1-\varepsilon_8)^k C^2s_0^2 \nonumber\\
					&\quad + \frac{\psi_4\mathcal{L}_{1,0}}{\varepsilon_8n} n\alpha^2 
					 	+ \frac{(1-2\varepsilon_5)\psi_2\psi_4}{\varepsilon_8 } n\tilde{d}^2 C^2s_0^2 \alpha^3 T  \Big) \nonumber\\
				&\le \Big(\frac{\mathcal{L}_{1,0}}{\varepsilon_4\tau_0 n} 
					+  \frac{(1-2\varepsilon_5)\psi_2\psi_4\tau_0^2\mathcal{L}_{1,0}}{\varepsilon_4\varepsilon_8n}    \nonumber\\
					&  \quad + \frac{C^2(1-2\varepsilon_5)^2\psi_2^2\psi_4\tau_4\tau_0^5}{\varepsilon_4\varepsilon_8}\Big)\frac{n^{\frac{1}{3}}\tilde{d}^\frac{2}{3}}{T^{\frac{2}{3}}} \nonumber\\
					&  \quad + \frac{C^2(1-2\varepsilon_5)\psi_2\tau_4\tau_0^2}{\varepsilon_4\varepsilon_8}\frac{n^{\frac{1}{3}}\tilde{d}^\frac{2}{3}}{T^{\frac{5}{3}}}, 
			\end{align}
			where the last inequality holds due to $\alpha=\frac{\tau_0}{n^{\frac{1}{3}}\tilde{d}^{\frac{2}{3}}T^{\frac{1}{3}}}$ and $s_0^2=\tau_4n\alpha^2$.
			Recalling $\mathcal{L}_{1,0}=\mathcal{O}(n)$, \eqref{nonconvex:cor-sm-eq1:0} yields \eqref{nonconvex:cor-sm-eq1}.

			Similar to the proof of \eqref{nonconvex:thm-sm-equ2:0}, 
			summing~\eqref{nonconvex:v4k} over $k \in [0,T-1]$ and using \eqref{nonconvex:cor-sm-eq1:0} gives
			\begin{align}\label{nonconvex:cor-sm-eq2:0}
				&f(\bar{x}_T)-f^*=\frac{e_{4,T}}{n}\le \frac{e_{4,0}}{n} + \frac{\ell^2\tau_0}{2}\frac{1}{n^{\frac{1}{3}}\tilde{d}^{\frac{2}{3}}T^{\frac{1}{3}}}\Big(  \nonumber\\
				&\quad \mathcal{O}\big(n^{\frac{1}{3}}\tilde{d}^{\frac{2}{3}}T^{\frac{1}{3}}\big) + \mathcal{O}\big(\frac{n^{\frac{1}{3}}\tilde{d}^{\frac{2}{3}}}{T^{\frac{2}{3}}}\big) \Big),
			\end{align}
			which implies \eqref{nonconvex:cor-sm-eq2}.

\subsection{Proof of Theorem~\ref{nonconvex:thm-liner}} \label{appendix:thm2}

	% In addition to the notations defined in Appendix~\ref{appendix:lemmas}--\ref{appendix:cor1},
	% we also denote the following notations.
	We denote the following constants:
	\begin{align*}
&\kappa_0^\prime=\min\Big\{\hat{\kappa}_0,
	~\kappa_5,~\frac{\kappa_6^\prime}{\sqrt{n}\tilde{d}}\Big\},\\
% &\kappa_5=\begin{cases}
% 	\text{the smaller positive root of }~1=\\
% 	\quad \alpha\min\{\varphi_1-\alpha\varphi_2,\;
% 	\alpha\varphi_3-\alpha\varphi_4\}, & \text{if it exists},\\[4pt]
% 	+\infty, & \text{otherwise,}
% 	\end{cases}\\
&\kappa_6^\prime=\sqrt{\frac{\varepsilon_5+2\varepsilon_5^2}{(1-2\varepsilon_5)\psi_3+\psi_4\psi_5}},\\
&\psi_5>\frac{ (1-2\varepsilon_5)\psi_2}{\varepsilon_6},\\
&\kappa_\nu=\frac{1}{2}\|\bm{x}_0 \|^2_{\bsE}+\frac{1}{2}\Big\|\bsv_0
	+\frac{1}{\gamma}\bsg_0^0\Big\|^2_{\frac{\beta+\gamma}{\gamma}\bsF} \\
	&\quad +\bsx_0^\top\bsE\bsF\Big(\bm{v}_0+\frac{1}{\gamma}\bsg_0^0\Big)
	+\frac{1}{2\nu}\|\bar{\bsg}_0^0\|^2,\\
% &\psi_5>\frac{ (1-2\varepsilon_5)\psi_2}{\varepsilon_6},\\
&\epsilon\in(\max\{\kappa_{9},~\kappa_{10}\},1),\\
&\kappa_{9}=\sqrt{1-\alpha\big(\varepsilon_6- (1-2\varepsilon_5)\psi_2\frac{1}{\psi_5}\big)},\\
&\kappa_{10}=\sqrt{1-(\varepsilon_5+2\varepsilon_5^2)+\alpha^2n\tilde{d}^2\big((1-2\varepsilon_5)\psi_3+\psi_4\psi_5\big)}.
% % &\varepsilon_6=\min\{\frac{\nu}{2},~\varphi_1-\alpha\varphi_2,~\alpha\varphi_3-\alpha\varphi_4\}/\varepsilon_2,\\
% &\varepsilon_6=\min\{\nu/2,~\varepsilon_3\}/\varepsilon_2.
\end{align*}

We use mathematical induction to prove
\begin{align}
	&\mathcal{L}_{1,k}\le n\tilde{d}^2\psi_5C^2s^2_k, \label{induction:cor:11}\\
	&\max_{i\in[n]}\|x_{i,k}-\hat{x}_{i,k-1}\|_p^2\le C^2s_k^2. \label{induction:cor:22}
\end{align}
From Assumption~\ref{nonconvex:ass:fil} and \eqref{pl}, we have
% \begin{align} \label{pl}
% 	\|\bar{\bsg}_k^0\|^2=n\|\nabla f(\bar{x}_k)\|^2\ge2n\nu(f(\bar{x}_k)-f^*),
% \end{align}
% and  
$\mathcal{L}_{1,0}\le\kappa_\nu$.
Then for $k=0$, the inequalities hold since $s_0\ge\max\{\sqrt{\frac{\kappa_\nu}{n\tilde{d}^2\psi_5C^2}},~\frac{\max_{i\in[n]}\|x_{i,0}\|_p}{C}\}$.
	Suppose that the statement holds for $k = 0,1,\ldots,\tau$, 
	we next show that the inequality also holds for $k = \tau+1$.
% For $k=0$, \eqref{induction:cor:11} holds due to $\mathcal{L}_{1,k}\le$
% since $s_0\ge\max_{i\in[n]}\|x_{i,0}\|_p/C$, \eqref{induction:cor:22} holds trivially. As for \eqref{induction:cor:11},
% We have
% \begin{align}\label{c:1}
% \|x_{i,k}-\hat{x}_{i,k}\|_p
% &=\|x_{i,k}-\hat{x}_{i,k-1}-s_k\mathcal{C}((x_{i,k}-\hat{x}_{i,k-1})/s_k)\|_p\nonumber\\
% &=s_k\|(x_{i,k}-\hat{x}_{i,k-1})/s_k-\mathcal{C}((x_{i,k}-\hat{x}_{i,k-1})/s_k)\|_p\nonumber\\
% &\le (1-\varphi)s_k,
% \end{align}
% where the first equality holds due to \eqref{nonconvex:kia-algo-dc-xhat_determin} and \eqref{nonconvex:kia-algo-dc-q_determin}; and the inequality holds due to \eqref{nonconvex:Ukandx} and \eqref{nonconvex:ass:compression_equ_determin}.

It is clear that Lemma~\ref{Lemma: PL} applies at iteration $\tau$.
% all the condition in Lemma~\ref{Lemma: PL} are satisfied.
% Noting $\alpha<\tilde{\kappa}_0\le\kappa_5$ and $\varepsilon_2>1$,
% \begin{align} \label{epsilon:01}
% 	\hspace{-0.4em} \alpha\varepsilon_6\le \frac{\alpha\varepsilon_3}{\varepsilon_2}=\frac{\alpha\min{\{\varphi_1-\alpha\varphi_2,~\alpha\varphi_3-\alpha\varphi_4}\}}{\varepsilon_2}<1.
% \end{align}
% From Assumption~\ref{nonconvex:ass:fil}, we have
% \begin{align} \label{pl}
% 	\|\bar{\bsg}_k^0\|^2=n\|\nabla f(\bar{x}_k)\|^2\ge2\nu(f(\bar{x}_k)-f^*).
% \end{align}
% Substituting \eqref{pl}, \eqref{c:1} into \eqref{L1:iteration} and using \eqref{hatx:ge}, \eqref{epsilon:01} with the induction hypothesis, we obtain
Substituting the induction hypothesis into \eqref{L1:iteration:induction} yields 
\begin{align} \label{induction:cor:1}
&\mathcal{L}_{1,\tau+1} 
% \mathcal{L}_{1,\tau} - \frac{\nu\alpha}{2}n(f(\bar{x}_\tau)-f^*) \nonumber\\
% 	&\quad + \alpha n\tilde{d}^2(1-2\varepsilon_5)\psi_2C^2s_\tau^2 \nonumber\\
% 	&\quad -\alpha\varepsilon_3\big(\|\bsx_{\tau}\|^2_{\bsE} + \big\|\bsv_{\tau}+\tfrac{1}{\gamma}\bsg_{\tau}^0\big\|^2_{\bsF}\big) \nonumber\\
% &\quad \le (1-\alpha\varepsilon_6)\mathcal{L}_{1,\tau} + \alpha n\tilde{d}^2(1-2\varepsilon_5)\psi_2C^2s_\tau^2 \nonumber\\
 \le (1-\alpha\varepsilon_6)n\tilde{d}^2\psi_5C^2s_\tau^2 + \alpha n\tilde{d}^2(1-2\varepsilon_5)\psi_2C^2s_\tau^2 \nonumber\\
&\quad \le \Big(1-\alpha\big(\varepsilon_6- (1-2\varepsilon_5)\psi_2\frac{1}{\psi_5}\big)\Big)\frac{n\tilde{d}^2\psi_5C^2s_{\tau+1}^2}{\epsilon^2} \nonumber\\
&\quad \le n\tilde{d}^2\psi_5C^2s_{\tau+1}^2,
\end{align}
where the last inequality holds due to $\psi_5>(1-2\varepsilon_5)\psi_2/\varepsilon_6$ and $\epsilon>\kappa_{9}$.

Next, we prove the second part of the induction.
Substituting the induction hypothesis into \eqref{c:iteration:induction} yields
\begin{align} \label{induction:cor:2}
	&\max_{i\in[n]}\|x_{i,\tau+1}-\hat{x}_{i,\tau}\|_p^2 
	% &\hspace{-0.2em}\le (1+\varepsilon_5  +\alpha^2 n\tilde{d}^2\psi_3)(1-2\varepsilon_5)C^2s_\tau^2 
	%  +\alpha^2\psi_4n\tilde{d}^2\psi_5C^2s_\tau^2  \nonumber\\
	\le \Big(1-(\varepsilon_5+2\varepsilon_5^2) \nonumber\\
		&\quad +\alpha^2n\tilde{d}^2\big((1-2\varepsilon_5)\psi_3+\psi_4\psi_5\big)\Big)\frac{C^2s_{\tau+1}^2}{\epsilon^2}\nonumber\\
	&\hspace{-0.2em} \le C^2s_{\tau+1}^2,
\end{align}
where the last inequality holds due to $\alpha<\kappa_0^\prime\le\kappa_6^\prime/(\sqrt{n}\tilde{d})$ and $\epsilon>\kappa_{10}$.
Then the combination of \eqref{induction:cor:1} and \eqref{induction:cor:2} completes the induction.

Finally, from  \eqref{induction:cor:11} and \eqref{hatx:le}, we know that 
\begin{align}
	 &f(\bar{x}_k)-f^\ast + \frac{1}{n}\sum_{i=1}^{n}\|x_{i,k}-\bar{x}_k\|^2 \nonumber\\
	 &\le \frac{\hat{\mathcal{L}}_{1,k}}{n}\le \frac{\mathcal{L}_{1,k}}{n\varepsilon_1}
     \le \frac{\psi_5C^2s_0^2}{\varepsilon_1}\tilde{d}^2\epsilon^{2k},
\end{align}
which gives \eqref{nonconvex:thm-ft-equ1_determin}.

\subsection{Constants Used in Lemma~\ref{lemma:5}} \label{appendix:lemma3}
	% We denote the following constant: 
	\begin{align*}
		&\varepsilon_9=\omega r\delta/2,\\
		&\varepsilon_{10}=(1-2\varepsilon_9)(1+\varepsilon_9^{-1}),\\
		&\varepsilon_{11}=\varepsilon_9 + 2\varepsilon_9^2 - 4\varepsilon_{10}\alpha^2\beta^2\rho^2(L).
	\end{align*}	
	% To analyze the one-step difference of $\mathcal{L}_{2,k}$, we study its additional component term $\|\bsx_{k}-\hat{\bsx}_{k}\|^2$.	
	% 	From \eqref{global:contraction},  
	% 		\begin{align}
	% 		&\mathbb{E}_{\mathcal{C}}[\|\bsx_{k+1}-\hat{\bsx}_{k+1}\|^2] 
	% 			=\sum\nolimits_{i=1}^{n} \mathbb{E}_{\mathcal{C}}[\|x_{i,k+1}-\hat{x}_{i,k+1}\|^2] \nonumber\\
	% 		&\le (1-2\varepsilon_9)\mathbb{E}_{\mathcal{C}}[\|\bsx_{k+1}-\hat{\bsx}_{k}\|^2] + n\omega rCs_k^2 \nonumber\\
	% 		&= (1-2\varepsilon_9)\mathbb{E}_{\mathcal{C}}[\|\bsx_{k+1}-\bsx_{k}+\bsx_{k}-\hat{\bsx}_{k}\|^2] + n\omega rCs_k^2 \nonumber\\
	% 		&\le \varepsilon_{10}\mathbb{E}_{\mathcal{C}}[\|\bsx_{k+1}-\bsx_{k}\|^2] \nonumber\\
	% 			& \quad + (1-\varepsilon_9-2\varepsilon_9^2)\mathbb{E}_{\mathcal{C}}[\|\bsx_{k}-\hat{\bsx}_{k-1}\|^2] +  n\omega rCs_k^2.
	% 			\label{nonconvex:xminush_compress}
	% 	\end{align}
	% 	Then substituting \eqref{nonconvex:xkoneminusx} into \eqref{nonconvex:xminush_compress} yields \eqref{nonconvex:xminush}.

\subsection{Proof of Lemma~\ref{lemma:L2}} \label{appendix:lemma4}

	We denote the following notations:
	\begin{align*}
		&\hat{\kappa}^{\prime}_0=\min\Big\{\frac{\varphi_1}{\varphi_2^{\prime}},~\frac{\varphi_3}{\varphi_4^{\prime}},~\frac{\varphi_5}{\varphi_6},
			~\frac{\sqrt{\varphi_7^2+8\varepsilon_{12}\varphi_8^{\prime}}-\varphi_7}{2\varphi_8^{\prime}}
			% ~\kappa_5,\frac{\kappa_6}{\sqrt{n}\tilde{d}}
			\Big\},\\
		&\varphi_2^{\prime}= (3+4\varepsilon_{10})\rho^2(L)\beta^2 - \rho_2(L)\beta\gamma + \big(\rho(L)+2\big)\gamma^2 \nonumber\\
			&\quad  + 1 + (\frac{5}{2}+4\varepsilon_{10})\ell^2,\\
		% &\varphi_2^{\prime}= (7+4\varepsilon_9^{-1})\rho^2(L)\beta^2 - \rho_2(L)\beta\gamma + \big(\rho(L)+2\big)\gamma^2 \nonumber\\
		% 	&\quad  + 1 + (\frac{13}{2}+4\varepsilon_9^{-1})\ell^2 + 2\varphi_6\ell^2,\\
		&\varphi_4^{\prime}=(2+4\varepsilon_{10})\rho(L)\gamma^2  + \frac{1}{2}\rho(L),\\
		&\varphi_8^{\prime}=(3+4\varepsilon_{10})\rho^2(L)\beta^2 + \big(\rho(L) - 2\rho_2(L) \big)\beta\gamma \nonumber\\
			&\quad + \big(\rho(L) + 2 \big)\gamma^2 + 1,\\
		&\varepsilon_{12}=(\varepsilon_9 + 2\varepsilon_9^2)/2,\\
		&\varepsilon_3^{\prime}=\min\{\varphi_1-\alpha\varphi_2^{\prime},~\varphi_3-\alpha\varphi_4^{\prime}\},\\
		% &\varepsilon_4^{\prime}=\min\{\varphi_1-\alpha\varphi_2^{\prime},~\alpha\varphi_3-\alpha\varphi_4^{\prime},~1/4\},\\
		% &\bsM_2^{\prime}=3\beta^2\bsL^2 - (\beta\gamma - \gamma^2)\bsL + \Big( 4\varepsilon_{10}\rho^2(L)\beta^2  \nonumber\\
		% 	&\quad + 2\gamma^2 + 1 + \frac{5}{2}\ell^2 + 2\varphi_6\ell^2 + 4\varepsilon_{10}\ell^2 \Big)\bsE.\\
		&\bsM_2^{\prime}=3\beta^2\bsL^2 - (\beta\gamma - \gamma^2)\bsL + \Big( 4\varepsilon_{10}\rho^2(L)\beta^2  \nonumber\\
			&\quad + 2\gamma^2 + 1 + \frac{5}{2}\ell^2 +  4\varepsilon_{10}\ell^2 \Big)\bsE.\\
	\end{align*}

		% We show the relation between $\mathcal{L}_{2,k+1}$ and $\mathcal{L}_{2,k}$.
		Combining \eqref{nonconvex:v1k}--\eqref{nonconvex:v4k} and \eqref{nonconvex:xminush} yields
		\begin{align} \label{L2:0}
				&\mathbb{E}_{\mathcal{C}}[\mathcal{L}_{2,k+1}]
			\le  \mathbb{E}_{\mathcal{C}}\Big[\mathcal{L}_{2,k} +\Big\|\bm{v}_k+\frac{1}{\gamma}\bsg_k^0\Big\|^2_{\frac{6\alpha^2\gamma^2\rho(L)+\alpha\gamma}{4}\bsF}
					\nonumber\\
				&\quad	-\|\bsx_k\|^2_{\frac{\alpha\beta}{2}\bsL-\frac{\alpha}{2}\bsE
					-\frac{\alpha}{2}(1+3\alpha)\ell^2\bsE - 3\alpha^2\beta^2\bsL^2}\nonumber\\
				&\quad+ \big(\frac{\alpha}{2}(\beta+2\gamma)\rho(L) + 3\alpha^2\beta^2\rho^2(L)\big)\|\bsx_k-\hat{\bsx}_k\|^2, \nonumber\\
			&\quad + \big(\alpha^2\gamma(\beta+\gamma)\rho(L)+\alpha^2\big)\|\bm{x}_k-\hat{\bm{x}}_k\|^2
				\nonumber\\
				&\quad +\Big\|\bm{v}_k+\frac{1}{\gamma}\bsg_{k}^0\Big\|^2_{\frac{\alpha\gamma}{4}\bsF} + \|\bm{x}_k\|^2_{\alpha^2\gamma(\beta+\gamma)\bsL + \alpha^2\bsE} \nonumber\\
				&\quad + \alpha^2 \tau_2\ell^2\|\bar{\bsg}_k\|^2  + \Big(\alpha^2\tau_3\ell^2+\frac{\alpha}{8}\Big)\|\bar{\bsg}_k\|^2\nonumber\\
			&\quad +\|\bm{x}_k\|^2_{\frac{\alpha(3\gamma+1)}{2}\bsE + 2\alpha^2(\gamma^2\bsE-\beta\gamma\bsL)+\frac{\alpha}{2}(1+2\alpha)\ell^2\bsE} \nonumber\\
				&\quad +\big(\alpha\gamma + 2\alpha^2(\gamma^2-\beta\gamma\rho_2(L))\big)\|\bm{x}_k-\hat{\bm{x}}_k\|^2 \nonumber\\
				&\quad-\Big\|\bm{v}_k+\frac{1}{\gamma}\bsg_{k}^0\Big\|^2_{\alpha(\gamma-\frac{5}{2}\rho^{-1}_2(L))\bsF
					-\frac{\alpha^2}{2}(\gamma^2+1)\rho(L)\bsF}, \nonumber\\
			&\quad -\frac{\alpha}{4}(1-2\alpha \ell)\|\bar{\bsg}_{k}\|^2
				+\|\bsx_k\|^2_{\frac{\alpha}{2}\ell^2\bsE} 
				-\frac{\alpha}{4}\|\bar{\bsg}_{k}^0\|^2, \nonumber\\
			&\quad -\varepsilon_{11}\|\bsx_{k}-\hat{\bsx}_{k}\|^2 + n\omega r C s_k^2 \nonumber\\
				&\quad	+\|\bsx_k\|^2_{4\varepsilon_{10}\alpha^2(\beta^2\rho^2(L)+\ell^2)\bsE}\nonumber\\
				&\quad+\Big\|\bm{v}_k+\frac{1}{\gamma}\bsg_k^0\Big\|^2_{4\varepsilon_{10}\alpha^2\gamma^2\rho(L)\bsF} \Big] \nonumber\\
			&=\mathbb{E}_{\mathcal{C}}\Big[ \mathcal{L}_{2,k}-\|\bsx_k\|^2_{\alpha\bsM_1-\alpha^2\bsM_2^{\prime}} 
				+ n\omega r C s_k^2  \nonumber\\
				&\quad 	- \Big\|\bm{v}_k+\frac{1}{\gamma}\bsg_k^0\Big\|^2_{\alpha(\varphi_3-\alpha\varphi_4^{\prime})\bsF} -\alpha(\varphi_5-\alpha\varphi_6)\|\bar{\bsg}_{k}\|^2  \nonumber\\
				&\quad 	- (2\varepsilon_{12} - \alpha\varphi_7- \alpha^2\varphi_8^{\prime})\|\bsx_{k}-\hat{\bsx}_{k}\|^2 -\frac{\alpha}{4}\|\bar{\bsg}_{k}^0\|^2 \Big]\nonumber\\
			&\le\mathbb{E}_{\mathcal{C}}\Big[ \mathcal{L}_{2,k}-\|\bsx_k\|^2_{\alpha(\varphi_1-\alpha\varphi_2^{\prime})\bsE}+ n\omega r C s_k^2  \nonumber\\
				&\quad 	- \Big\|\bm{v}_k+\frac{1}{\gamma}\bsg_k^0\Big\|^2_{\alpha(\varphi_3-\alpha\varphi_4^{\prime})\bsF} -\alpha(\varphi_5-\alpha\varphi_6)\|\bar{\bsg}_{k}\|^2  \nonumber\\
				&\quad 	- (2\varepsilon_{12} - \alpha\varphi_7- \alpha^2\varphi_8^{\prime})\|\bsx_{k}-\hat{\bsx}_{k}\|^2 -\frac{\alpha}{4}\|\bar{\bsg}_{k}^0\|^2 \Big] , 
		\end{align}
		where the last inequality holds due to \eqref{nonconvex:KL-L-eq2}.

		Next, we show that the constants appearing on the right-hand side of \eqref{L2:0} are positive.

		Given  $\alpha<\hat{\kappa}^{\prime}_0\le\min\{\frac{\varphi_1}{\varphi_2^{\prime}},~\frac{\varphi_3}{\varphi_4^{\prime}},~\frac{\varphi_5}{\varphi_6},
		~\frac{\sqrt{\varphi_7^2+8\varepsilon_{12}\varphi_8^{\prime}}-\varphi_7}{2\varphi_8^{\prime}}\}$,
		% and $\alpha\in(0,\kappa_7)$,
		we have 
		\begin{subequations}
			\begin{align}
				&\alpha(\varphi_1-\alpha\varphi_2^{\prime})>0,\label{nonconvex:vkLya1.1p_thm3}\\
				&\alpha(\varphi_3-\alpha\varphi_4^{\prime})>0,\label{nonconvex:vkLya1.3p_thm3}\\
				&\alpha(\varphi_5-\alpha\varphi_6)>0,\label{nonconvex:vkLya1.2p_thm3}\\
				&2\varepsilon_{12} - \alpha\varphi_7- \alpha^2\varphi_8^{\prime}>0. \label{nonconvex:vkLya1.4p_thm3}
			\end{align}
		\end{subequations}
		Then the combination of \eqref{L2:0}--\eqref{nonconvex:vkLya1.4p_thm3} gives \eqref{L2:iteration}.

\subsection{Proof of Theorem~\ref{nonconvex:thm-sm_quantization}} \label{appendix:thm3}

	Denote $\varepsilon_4^{\prime}=\min\{\varepsilon_3^{\prime},~1/4\}$.
	Under the setting of Theorem~\ref{nonconvex:thm-sm_quantization}, 
			it is clear that all the conditions in Lemmas~\ref{lemma:5} and \ref{lemma:L2} are satisﬁed.
	% Now we are ready to prove Theorem~\ref{nonconvex:thm-sm_quantization}.
	Summing \eqref{L2:iteration} over $k\in[0,T-1]$ yields	
	\begin{align}
		&\frac{1}{T}\sum_{k=0}^{T-1}\mathbb{E}_{\mathcal{C}}\Big[\|\nabla f(\bar{x}_k)\|^2+\frac{1}{n}\sum_{i=1}^{n}\|x_{i,k}-\bar{x}_k\|^2\Big] \nonumber\\
		&=\frac{1}{nT}\sum_{k=0}^{T-1}\mathbb{E}_{\mathcal{C}}[\|\bar{\bsg}_{k}^0\|^2 + \|\bsx_k\|_{\bsE}^2]\nonumber\\
		&\le \frac{1}{\varepsilon_4^{\prime}\alpha nT}\Big(\mathcal{L}_{2,0} + \sum_{k=0}^{T-1}n\omega rCs_0^2\varepsilon^{2k}\Big) \nonumber\\
		&\le \frac{1}{T}\Big(\frac{\mathcal{L}_{2,0}}{\varepsilon_4^{\prime}\alpha n} + \frac{\omega rCs_0^2}{(1-\varepsilon^2)\varepsilon_4^{\prime}\alpha}\Big), \label{thm3:10}\\
		&\mathbb{E}_{\mathcal{C}}[n(f(\bar{x}_T)\hspace{-0.05em}-\hspace{-0.05em}f^*)]\hspace{-0.1em}\le\hspace{-0.1em}\mathbb{E}_{\mathcal{C}}[\mathcal{L}_{2,T}]
			\hspace{-0.1em}\le\hspace{-0.1em} \mathcal{L}_{2,0} \hspace{-0.1em}+\hspace{-0.15em} \frac{n\omega rCs_0^2}{(1-\varepsilon^2)}. \label{thm3:20}
	\end{align}
	Noting $\mathcal{L}_{2,0}=\mathcal{O}(n)$, \eqref{nonconvex:thm-sm-equ1_quantization} follows from \eqref{thm3:10} and \eqref{nonconvex:thm-sm-equ2_quantization} follows from \eqref{thm3:20}.

\subsection{Proof of Theorem~\ref{nonconvex:thm-ft_quantization}}  \label{appendix:thm4}
	% 	In addition to the notations defined in Appendix\ref{appendix:lemmas}--\ref{appendix:thm3},
	% we also denote the following notations.
	We denote the following notations:
	\begin{align*}
	&\hat{\varepsilon}\in(\max\{1-\alpha\varepsilon_6^{\prime},~\varepsilon^2\},~1),\\
	&\varepsilon_6^{\prime}=\min\Big\{\nu/2,~\varepsilon_3^{\prime},
		~(2\varepsilon_{12} - \alpha\varphi_7- \alpha^2\varphi_8^{\prime})/\alpha\Big\}/\varepsilon_2,\\
	&\hat{\mathcal{L}}_{2,k}=\|\bsx_k\|^2_{\bsE}+\Big\|\bsv_k+\frac{1}{\gamma}\bsg_k^0\Big\|^2_{\bsF} + \tilde{f}(\bar{\bsx}_{k})-nf^* \\
		&\quad + \|\bsx_{k}-\hat{\bsx}_{k}\|^2.
	\end{align*}
	Similar to the proof \eqref{hatx:le} and \eqref{hatx:ge}, we know that
	\begin{subequations}
	\begin{align}
		&\mathcal{L}_{2,k}
		\ge\varepsilon_1\hat{\mathcal{L}}_{2,k}\ge0,\label{hatx2:le}\\
		&\mathcal{L}_{2,k}\le\varepsilon_2\hat{\mathcal{L}}_{2,k}. \label{hatx2:ge}
	\end{align}
	\end{subequations}
	From $\omega\le1/r$ and $\delta\in(0,1]$, we have $\varepsilon_9\le 1/2$ and $2\varepsilon_{12}\le1$.
	Recalling $\varepsilon_2>1$, we get 
	\begin{align} \label{le1}
			\alpha\varepsilon_6^{\prime}\le\frac{2\varepsilon_{12} - \alpha\varphi_7- \alpha^2\varphi_8^{\prime}}{\varepsilon_2}\le\frac{2\varepsilon_{12}}{\varepsilon_2}<1.
	\end{align}
	Based on Assumption~\ref{nonconvex:ass:fil},
	we substitute~\eqref{pl} into~\eqref{L2:iteration} and apply~\eqref{hatx2:ge} and~\eqref{le1} to obtain 
	% substituting \eqref{pl} into \eqref{L2:iteration} 
	% and using \eqref{hatx2:ge} together with \eqref{le1} yield	
	\begin{align}
		&\mathbb{E}_{\mathcal{C}}[n(f(\bar{x}_k)-f^*)]\le\mathbb{E}_{\mathcal{C}}[\mathcal{L}_{2,k}] \nonumber\\
		&\le 
		\mathbb{E}_{\mathcal{C}}\Big[(1-\alpha\varepsilon_6^{\prime})\mathcal{L}_{2,k-1}  + n\omega r C s_{k-1}^2 \Big] \nonumber\\
		&= 
		\mathbb{E}_{\mathcal{C}}\Big[(1-\alpha\varepsilon_6^{\prime})\mathcal{L}_{2,k-1}  + n\omega r C s_0^2\varepsilon^{2k-2} \Big] \nonumber\\
		&\le (1-\alpha\varepsilon_6^{\prime})^k \mathcal{L}_{2,0} \nonumber\\
	&\quad + n\omega r C s_0^2 
	\sum_{t=0}^{k-1}(1-\alpha\varepsilon_6^{\prime})^{k-1-t}\varepsilon^{2t}, 
	\end{align}
  	which gives \eqref{nonconvex:thm-ft-equ1_quantization}.

\subsection{Proofs of Lemmas \ref{Lemma Consensus}--\ref{Lemma Optimality 2}}\label{imed}
% \begin{proof} \vspace{-0.5em}

\noindent {\bf (i)}  This step shows the relation between $e_{1,k+1}$ and $e_{1,k}$. \vspace{-0.5em}
	\begin{align}
		&e_{1,k+1}=\frac{1}{2}\|\bm{x}_{k+1} \|^2_{\bsE}
		=\frac{1}{2}\|\bm{x}_k-\alpha(\beta\bsL\hat{\bm{x}}_k+\gamma\bm{v}_k+\bsg_k) \|^2_{\bsE}\nonumber\\
		&=\frac{1}{2}\|\bm{x}_k\|^2_{\bsE}-\alpha\beta\bsx^\top_k\bsL\hat{\bm{x}}_k
		+\|\hat{\bm{x}}_k\|^2_{\frac{\alpha^2\beta^2}{2}\bsL^2}\nonumber\\
		&\quad-\alpha\gamma(\bsx^\top_k-\alpha\beta\hat{\bm{x}}_k^\top\bsL)\bsE
		\Big(\bm{v}_k+\frac{1}{\gamma}\bsg_k\Big)\nonumber\\
		&\quad+\Big\|\bm{v}_k+\frac{1}{\gamma}\bsg_k\Big\|^2_{\frac{\alpha^2\gamma^2}{2}\bsE}\nonumber\\
		&=\frac{1}{2}\|\bm{x}_k\|^2_{\bsE}-\alpha\beta\bsx^\top_k\bsL(\bm{x}_k+\hat{\bm{x}}_k-\bm{x}_k)
		+\|\hat{\bm{x}}_k\|^2_{\frac{\alpha^2\beta^2}{2}\bsL^2}\nonumber\\
		&\quad-\alpha\gamma(\bsx^\top_k-\alpha\beta\hat{\bm{x}}_k^\top\bsL)\bsE\Big(\bm{v}_k
		+\frac{1}{\gamma}\bsg_k^0
		+\frac{1}{\gamma}\bsg_k-\frac{1}{\gamma}\bsg_k^0\Big)\nonumber\\
		&\quad+\Big\|\bm{v}_k+\frac{1}{\gamma}\bsg_k^0
		+\frac{1}{\gamma}\bsg_k-\frac{1}{\gamma}\bsg_k^0\Big\|^2_{\frac{\alpha^2\gamma^2}{2}\bsE}\nonumber\\
		&\le\frac{1}{2}\|\bm{x}_k\|^2_{\bsE}-\|\bsx_k\|^2_{\alpha\beta\bsL}
		+\|\bsx_k\|^2_{\frac{\alpha\beta}{2}\bsL}+\|\bm{x}_k-\hat{\bm{x}}_k\|^2_{\frac{\alpha\beta}{2}\bsL}
		\nonumber\\
		&\quad+\|\hat{\bm{x}}_k\|^2_{\frac{\alpha^2\beta^2}{2}\bsL^2}
		-\alpha\gamma\bsx^\top_k\bsE\Big(\bm{v}_k+\frac{1}{\gamma}\bsg_k^0\Big)
		+\frac{\alpha}{2}\|\bm{x}_k\|^2_{\bsE}\nonumber\\
		&\quad+\frac{\alpha}{2}\|\bsg_k-\bsg_k^0\|^2+\|\hat{\bm{x}}_k\|^2_{\frac{\alpha^2\beta^2}{2}\bsL^2}
		+\frac{\alpha^2\gamma^2}{2}\Big\|\bm{v}_k+\frac{1}{\gamma}\bsg_k^0\Big\|^2\nonumber\\
		&\quad+\|\hat{\bm{x}}_k\|^2_{\frac{\alpha^2\beta^2}{2}\bsL^2}
		+\frac{\alpha^2}{2}\|\bsg_k-\bsg_k^0\|^2\nonumber\\
		&\quad+\alpha^2\gamma^2\Big\|\bm{v}_k+\frac{1}{\gamma}\bsg_k^0\Big\|^2
		+\alpha^2\|\bsg_k-\bsg_k^0\|^2\nonumber\\
		&=\frac{1}{2}\|\bm{x}_k\|^2_{\bsE}-\|\bsx_k\|^2_{\frac{\alpha\beta}{2}\bsL-\frac{\alpha}{2}\bsE}
		+\|\hat{\bm{x}}_k-\bm{x}_k+\bm{x}_k\|^2_{\frac{3\alpha^2\beta^2}{2}\bsL^2}\nonumber\\
		&\quad+\frac{\alpha}{2}(1+3\alpha)\|\bsg_k-\bsg_k^0\|^2
		+\|\bm{x}_k-\hat{\bm{x}}_k\|^2_{\frac{\alpha\beta}{2}\bsL}\nonumber\\
		&\quad-\alpha\gamma(\hat{\bsx}_k+\bsx_k-\hat{\bsx}_k)^\top\bsE
		\Big(\bm{v}_k+\frac{1}{\gamma}\bsg_k^0\Big)\nonumber\\
		&\quad+\frac{3\alpha^2\gamma^2}{2}\Big\|\bm{v}_k+\frac{1}{\gamma}\bsg_k^0\Big\|^2\nonumber\\
		&\le\frac{1}{2}\|\bm{x}_k\|^2_{\bsE}-\|\bsx_k\|^2_{\frac{\alpha\beta}{2}\bsL-\frac{\alpha}{2}\bsE-3\alpha^2\beta^2\bsL^2}
		% +\|\hat{\bm{x}}_k\|^2_{\frac{3\alpha^2\beta^2}{2}\bsL^2}
		\nonumber\\
		&\quad+\frac{\alpha}{2}(1+3\alpha)\|\bsg_k-\bsg_k^0\|^2 -\alpha\gamma\hat{\bsx}^\top_k\bsE\Big(\bm{v}_k+\frac{1}{\gamma}\bsg_k^0\Big)
		\nonumber\\
		&\quad +\|\bm{x}_k-\hat{\bm{x}}_k\|^2_{\frac{\alpha}{2}(\beta\bsL+2\gamma\rho(L)\bsE)+3\alpha^2\beta^2\bsL^2} \nonumber\\
		&\quad
		+\frac{6\alpha^2\gamma^2+\alpha\gamma\rho^{-1}(L)}{4}\Big\|\bm{v}_k+\frac{1}{\gamma}\bsg_k^0\Big\|^2
		\nonumber\\%\label{nonconvex:v1k-zero}\\
		&\le\frac{1}{2}\|\bm{x}_k\|^2_{\bsE}-\|\bsx_k\|^2_{\frac{\alpha\beta}{2}\bsL-\frac{\alpha}{2}\bsE
			-\frac{\alpha}{2}(1+3\alpha)\ell^2\bsE - 3\alpha^2\beta^2\bsL^2}
			% +\|\hat{\bm{x}}_k\|^2_{\frac{3\alpha^2\beta^2}{2}\bsL^2}
			\nonumber\\
		&\quad-\alpha\gamma\hat{\bsx}^\top_k\bsE\Big(\bm{v}_k+\frac{1}{\gamma}\bsg_k^0\Big)
		+\Big\|\bm{v}_k+\frac{1}{\gamma}\bsg_k^0\Big\|^2_{\frac{6\alpha^2\gamma^2\rho(L)+\alpha\gamma}{4}\bsF}\nonumber\\
		&\quad+ \big(\frac{\alpha}{2}(\beta+2\gamma)\rho(L) + 3\alpha^2\beta^2\rho^2(L)\big)\|\bsx_k-\hat{\bsx}_k\|^2, \hspace{-0.04em}\label{nonconvex:v1k:deterministic}
	\end{align}
	where the second and third equalities hold due to \eqref{nonconvex:kia-algo-dc-compact-x} and \eqref{nonconvex:KL-L-eq}, respectively; 
	the first and second inequalities hold due to the Cauchy--Schwarz inequality and $\rho(\bsE)=1$; 
	and the last  inequality holds due to \eqref{nonconvex:gg1}, \eqref{nonconvex:lemma-eq5}, \eqref{nonconvex:KL-L-eq2} and $\rho(\bsE)=1$.
	% Then, since  $\mathcal{C}_k$ is independent of $\bsx_k$ and $\bsv_k$, taking the expectation with respect to $\mathcal{C}_k$ on both sides of \eqref{nonconvex:v1k:deterministic} yields \eqref{nonconvex:v1k}.

	\noindent {\bf (ii)}  This step shows the relation between $e_{2,k+1}$ and $e_{2,k}$.
	\begin{align}
		&e_{2,k+1}=\frac{1}{2}\Big\|\bsv_{k+1}+\frac{1}{\gamma}\bsg_{k+1}^0\Big\|^2_{\frac{\beta+\gamma}{\gamma}\bsF}\nonumber\\
		&=\frac{1}{2}\Big\|\bm{v}_k+\frac{1}{\gamma}\bsg_{k}^0+\alpha\gamma\bsL\hat{\bm{x}}_k
		+\frac{1}{\gamma}(\bsg_{k+1}^0-\bsg_{k}^0) \Big\|^2_{\frac{\beta+\gamma}{\gamma}\bsF}\nonumber\\
		&=\frac{1}{2}\Big\|\bm{v}_k+\frac{1}{\gamma}\bsg_{k}^0\Big\|^2_{\frac{\beta+\gamma}{\gamma}\bsF}
		+\alpha(\beta+\gamma)\hat{\bsx}^\top_k\bsE\Big(\bm{v}_k+\frac{1}{\gamma}\bsg_k^0\Big)\nonumber\\
		&\quad+\|\hat{\bsx}_k\|^2_{\frac{\alpha^2\gamma}{2}(\beta+\gamma)\bsL}
		+\frac{1}{2\gamma^2}\|\bsg_{k+1}^0-\bsg_{k}^0\|^2_{\frac{\beta+\gamma}{\gamma}\bsF}\nonumber\\
		&\quad+\frac{1}{\gamma}\Big(\bm{v}_k+\frac{1}{\gamma}\bsg_{k}^0
		\Big)^\top\Big(\frac{\beta+\gamma}{\gamma}\bsF\Big)(\bsg_{k+1}^0-\bsg_{k}^0)\nonumber\\
		&\quad+\alpha\hat{\bm{x}}_k^\top\Big(\bsE+\frac{\beta}{\gamma}\bsE\Big)(\bsg_{k+1}^0-\bsg_{k}^0)\nonumber\\
		&\le\frac{1}{2}\Big\|\bm{v}_k+\frac{1}{\gamma}\bsg_{k}^0\Big\|^2_{\frac{\beta+\gamma}{\gamma}\bsF}
		+\alpha(\beta+\gamma)\hat{\bsx}^\top_k\bsE\Big(\bm{v}_k+\frac{1}{\gamma}\bsg_k^0\Big)\nonumber\\
		&\quad+\|\hat{\bsx}_k\|^2_{\frac{\alpha^2\gamma}{2}(\beta+\gamma)\bsL}
		+\|\bsg_{k+1}^0-\bsg_{k}^0\|^2_{\frac{\beta+\gamma}{2\gamma^3}\bsF}\nonumber\\
		&\quad+\Big\|\bm{v}_k+\frac{1}{\gamma}\bsg_{k}^0\Big\|^2_{\frac{\alpha\gamma}{4}\bsF}
		+\|\bsg_{k+1}^0-\bsg_{k}^0\|^2_{\frac{(\beta+\gamma)^2}{\alpha\gamma^5}\bsF}\nonumber\\
		&\quad+\|\hat{\bm{x}}_k\|^2_{\frac{\alpha^2}{2}\bsE}
		+\frac{1}{2}\|\bsg_{k+1}^0-\bsg_{k}^0\|^2
		+\frac{\alpha\beta}{\gamma}\hat{\bm{x}}_k^\top\bsE(\bsg_{k+1}^0-\bsg_{k}^0)\nonumber\\
		&=\frac{1}{2}\Big\|\bm{v}_k+\frac{1}{\gamma}\bsg_{k}^0\Big\|^2_{\frac{\beta+\gamma}{\gamma}\bsF}
		+\alpha(\beta+\gamma)\hat{\bsx}^\top_k\bsE\Big(\bm{v}_k+\frac{1}{\gamma}\bsg_k^0\Big)\nonumber\\
		&\quad+\|\hat{\bm{x}}_k-\bm{x}_k+\bm{x}_k\|^2_{\frac{\alpha^2\gamma}{2}(\beta+\gamma)\bsL+\frac{\alpha^2}{2}\bsE}
		\nonumber\\
		&\quad +\Big\|\bm{v}_k+\frac{1}{\gamma}\bsg_{k}^0\Big\|^2_{\frac{\alpha\gamma}{4}\bsF} +\|\bsg_{k+1}^0-\bsg_{k}^0\|^2_{
			\frac{\beta+\gamma}{2\gamma^3}\bsF + \frac{(\beta+\gamma)^2}{\alpha\gamma^5}\bsF}\nonumber\\%%%%%%%%%
		&\quad+\frac{1}{2}\|\bsg_{k+1}^0-\bsg_{k}^0\|^2
		+\frac{\alpha\beta}{\gamma}\hat{\bm{x}}_k^\top\bsE(\bsg_{k+1}^0-\bsg_{k}^0)\nonumber\\
		&\le\frac{1}{2}\Big\|\bm{v}_k+\frac{1}{\gamma}\bsg_{k}^0\Big\|^2_{\frac{\beta+\gamma}{\gamma}\bsF}
		+\alpha(\beta+\gamma)\hat{\bsx}^\top_k\bsE\Big(\bm{v}_k+\frac{1}{\gamma}\bsg_k^0\Big)\nonumber\\
		&\quad+\|\bm{x}_k - \hat{\bm{x}}_k\|^2_{\alpha^2\gamma(\beta+\gamma)\bsL + \alpha^2\bsE}
		+\Big\|\bm{v}_k+\frac{1}{\gamma}\bsg_{k}^0\Big\|^2_{\frac{\alpha\gamma}{4}\bsF}\nonumber\\
		&\quad+\|\bm{x}_k\|^2_{\alpha^2\gamma(\beta+\gamma)\bsL + \alpha^2\bsE} + \tau_2\|\bsg_{k+1}^0-\bsg_{k}^0\|^2 \nonumber\\
		&\quad +\frac{\alpha\beta}{\gamma}
		\hat{\bm{x}}_k^\top\bsE(\bsg_{k+1}^0-\bsg_{k}^0)
		\nonumber\\%\label{nonconvex:v2k-zero}\\
		&\le\frac{1}{2}\Big\|\bm{v}_k+\frac{1}{\gamma}\bsg_{k}^0\Big\|^2_{\frac{\beta+\gamma}{\gamma}\bsF}
		+\alpha(\beta+\gamma)\hat{\bsx}^\top_k\bsE\Big(\bm{v}_k+\frac{1}{\gamma}\bsg_k^0\Big)\nonumber\\
		&\quad + \big(\alpha^2\gamma(\beta+\gamma)\rho(L)+\alpha^2\big)\|\bm{x}_k-\hat{\bm{x}}_k\|^2
		\nonumber\\
		&\quad +\Big\|\bm{v}_k+\frac{1}{\gamma}\bsg_{k}^0\Big\|^2_{\frac{\alpha\gamma}{4}\bsF} +\|\bm{x}_k\|^2_{\alpha^2\gamma(\beta+\gamma)\bsL + \alpha^2\bsE} \nonumber\\
		&\quad + \alpha^2 \tau_2\ell^2\|\bar{\bsg}_k\|^2 + \frac{\alpha\beta}{\gamma}\hat{\bm{x}}_k^\top\bsE(\bsg_{k+1}^0-\bsg_{k}^0),\label{nonconvex:v2k:deterministic}
	\end{align}
	where the second and third equalities hold due to \eqref{nonconvex:kia-algo-dc-compact-v} and \eqref{nonconvex:lemma-eq3}, respectively;  
	the first and second inequalities hold due to the Cauchy--Schwarz inequality, $\rho(\bsE)=1$ and \eqref{nonconvex:lemma-eq5};
	and  the last inequality holds due to \eqref{nonconvex:gg}, \eqref{nonconvex:KL-L-eq2}, and $\rho(\bsE)=1$.
	% Then, since  $\mathcal{C}_k$ is independent of $\bsx_k$ and $\bsv_k$, taking the expectation with respect to $\mathcal{C}_k$ on both sides of \eqref{nonconvex:v2k:deterministic} yields \eqref{nonconvex:v2k}.

	\noindent {\bf (iii)}  This step shows the relation between $e_{3,k+1}$ and $e_{3,k}$.	
	\begin{align}
		&e_{3,k+1}=\bsx_{k+1}^\top\bsE\bsF\Big(\bm{v}_{k+1}+\frac{1}{\gamma}\bsg_{k+1}^0\Big)\nonumber\\
		&=(\bm{x}_k-\alpha(\beta\bsL\hat{\bsx}_k+\gamma\bm{v}_k+\bsg_k^0+\bsg_k-\bsg_k^0))^\top
		\bsE\bsF\nonumber\\
		&\quad\times\Big(\bm{v}_k+\frac{1}{\gamma}\bsg_{k}^0+\alpha\gamma\bsL\hat{\bsx}_k+\frac{1}{\gamma}(\bsg_{k+1}^0-\bsg_{k}^0)\Big)\nonumber\\
		&=(\bm{x}_k^\top\bsE\bsF-\alpha(\beta+\alpha\gamma^2)\hat{\bsx}_k^\top\bsE)
		\Big(\bm{v}_k+\frac{1}{\gamma}\bsg_{k}^0\Big)\nonumber\\
		&\quad+\alpha\gamma\bm{x}_k^\top\bsE\hat{\bsx}_k-\|\hat{\bsx}_k\|^2_{\alpha^2\beta\gamma\bsL}\nonumber\\
		&\quad
		+\frac{1}{\gamma}(\bm{x}_k^\top\bsE\bsF-\alpha\beta\hat{\bsx}_k^\top\bsE)(\bsg_{k+1}^0-\bsg_{k}^0)\nonumber\\
		&\quad-\alpha(\gamma\bm{v}_k+\bsg_{k}^0+\bsg_k-\bsg_k^0-\bar{\bsg}_k)^\top\bsF
		\Big(\bm{v}_k+\frac{1}{\gamma}\bsg_{k}^0\Big)\nonumber\\%%%%
		&\quad
		-\alpha\Big(\bm{v}_k+\frac{1}{\gamma}\bsg_{k}^0\Big)^\top\bsE\bsF(\bsg_{k+1}^0-\bsg_{k}^0)\nonumber\\
		&\quad-\alpha(\bsg_k-\bsg_k^0)^\top
		\Big(\alpha\gamma\bsE\hat{\bm{x}}_k+\frac{1}{\gamma}\bsE\bsF(\bsg_{k+1}^0-\bsg_{k}^0)\Big)\nonumber\\
		&\le(\bm{x}_k^\top\bsE\bsF-\alpha\beta\hat{\bsx}_k^\top\bsE)
		\Big(\bm{v}_k+\frac{1}{\gamma}\bsg_{k}^0\Big)
		+\|\hat{\bsx}_k\|^2_{\frac{\alpha^2\gamma^2}{2}\bsE}\nonumber\\
		&\quad+\frac{\alpha^2\gamma^2}{2}\Big\|\bm{v}_k+\frac{1}{\gamma}\bsg_{k}^0\Big\|^2
		+\|\bm{x}_k\|^2_{\frac{\alpha\gamma}{2}\bsE}+\|\hat{\bm{x}}_k\|^2_{\frac{\alpha\gamma}{2}(\bsE-2\alpha\beta\bsL)}\nonumber\\
		&\quad+\|\bm{x}_k\|^2_{\frac{\alpha}{2}\bsE}+\|\bsg_{k+1}^0-\bsg_{k}^0\|^2_{\frac{1}{2\alpha\gamma^2}\bsF^2}
		\nonumber\\
		&\quad-\frac{\alpha\beta}{\gamma}\hat{\bm{x}}_k^\top\bsE(\bsg_{k+1}^0-\bsg_{k}^0)
		-\Big\|\bm{v}_k+\frac{1}{\gamma}\bsg_{k}^0\Big\|^2_{\alpha\gamma\bsF}\nonumber\\
		&\quad+\frac{\alpha}{2}\|\bsg_k-\bsg_k^0\|^2 + \Big\|\bm{v}_k+\frac{1}{\gamma}\bsg_{k}^0\Big\|^2_{\frac{\alpha}{2}\bsF^2} \nonumber\\
		&\quad  + \frac{\alpha}{8}\|\bar{\bsg}_k\|^2 + \Big\|\bm{v}_k+\frac{1}{\gamma}\bsg_{k}^0\Big\|^2_{2\alpha\bsF^2} + \frac{\alpha^2}{2}\Big\|\bm{v}_k+\frac{1}{\gamma}\bsg_{k}^0\Big\|^2 \nonumber\\%%%%%
		&\quad 
		+\|\bsg_{k+1}^0-\bsg_{k}^0\|^2_{\frac{1}{2}\bsF^2}
		+\frac{\alpha^2}{2}\|\bsg_k-\bsg_k^0\|^2\nonumber\\
		&\quad+\|\hat{\bm{x}}_k\|^2_{\frac{\alpha^2\gamma^2}{2}\bsE}
		+\frac{\alpha^2}{2}\|\bsg_k-\bsg_k^0\|^2
		+\|\bsg_{k+1}^0-\bsg_{k}^0\|^2_{\frac{1}{2\gamma^2}\bsF^2}\nonumber\\
		&=(\bm{x}_k^\top\bsE\bsF-\alpha\beta\hat{\bsx}_k^\top\bsE)\Big(\bm{v}_k+\frac{1}{\gamma}\bsg_{k}^0\Big)
		+\|\bm{x}_k\|^2_{\frac{\alpha(\gamma+1)}{2}\bsE}\nonumber\\
		&\quad+\|\hat{\bm{x}}_k-\bm{x}_k+\bm{x}_k\|^2_{\frac{\alpha\gamma}{2}\bsE
			+\alpha^2(\gamma^2\bsE-\beta\gamma\bsL)} \nonumber\\
		&\quad +\frac{\alpha}{2}(1+2\alpha)\|\bsg_k-\bsg_k^0\|^2 + \|\bsg_{k+1}^0-\bsg_{k}^0\|^2_{(\frac{\alpha+1}{2\alpha\gamma^2}+\frac{1}{2})\bsF^2} \nonumber\\
		&\quad -\frac{\alpha\beta}{\gamma}\hat{\bm{x}}_k^\top\bsE(\bsg_{k+1}^0-\bsg_{k}^0) + \frac{\alpha}{8}\|\bar{\bsg}_k\|^2\nonumber\\
		&\quad
		-\Big\|\bm{v}_k+\frac{1}{\gamma}\bsg_{k}^0\Big\|^2_{\alpha\gamma\bsF-\frac{5}{2}\alpha\bsF^2 - \frac{\alpha^2(\gamma^2+1)}{2}{\bf I}_{np}}\nonumber\\
		&\le(\bm{x}_k^\top\bsE\bsF-\alpha\beta\hat{\bsx}_k^\top\bsE)\Big(\bm{v}_k+\frac{1}{\gamma}\bsg_{k}^0\Big) \nonumber\\
		&\quad +\|\bm{x}_k\|^2_{\frac{\alpha(3\gamma+1)}{2}\bsE + 2\alpha^2(\gamma^2\bsE-\beta\gamma\bsL)} +\frac{\alpha}{8}\|\bar{\bsg}_k\|^2 \nonumber\\
		&\quad+\|\bm{x}_k-\hat{\bm{x}}_k\|^2_{\alpha\gamma\bsE
			+ 2\alpha^2(\gamma^2\bsE-\beta\gamma\bsL)}  + \tau_3\|\bsg_{k+1}^0-\bsg_{k}^0\|^2
		\nonumber\\
		&\quad  +\frac{\alpha}{2}(1+2\alpha)\|\bsg_k-\bsg_k^0\|^2 - \frac{\alpha\beta}{\gamma}\hat{\bm{x}}_k^\top
		\bsE(\bsg_{k+1}^0-\bsg_{k}^0)\nonumber\\
		&\quad-\Big\|\bm{v}_k+\frac{1}{\gamma}\bsg_{k}^0\Big\|^2_{\alpha(\gamma-\frac{5}{2}\rho^{-1}_2(L))\bsF
			-\frac{\alpha^2}{2}(\gamma^2+1)\rho(L)\bsF}
		\nonumber\\%\label{nonconvex:v3k-zero}\\
		&\le\bm{x}_k^\top\bsE\bsF\Big(\bm{v}_k+\frac{1}{\gamma}\bsg_{k}^0\Big)
		-\alpha\beta\hat{\bsx}_k^\top\bsE\Big(\bm{v}_k+\frac{1}{\gamma}\bsg_{k}^0\Big)\nonumber\\
		&\quad +\|\bm{x}_k\|^2_{\frac{\alpha(3\gamma+1)}{2}\bsE + 2\alpha^2(\gamma^2\bsE-\beta\gamma\bsL)+\frac{\alpha}{2}(1+2\alpha)\ell^2\bsE} \nonumber\\
		&\quad +\big(\alpha\gamma + 2\alpha^2(\gamma^2-\beta\gamma\rho_2(L))\big)\|\bm{x}_k-\hat{\bm{x}}_k\|^2 \nonumber\\
		&\quad-\frac{\alpha\beta}{\gamma}\hat{\bm{x}}_k^\top
		\bsE(\bsg_{k+1}^0-\bsg_{k}^0)+\Big(\alpha^2\tau_3\ell^2+\frac{\alpha}{8}\Big)\|\bar{\bsg}_k\|^2\nonumber\\
		&\quad-\Big\|\bm{v}_k+\frac{1}{\gamma}\bsg_{k}^0\Big\|^2_{\alpha(\gamma-\frac{5}{2}\rho^{-1}_2(L))\bsF
			-\frac{\alpha^2}{2}(\gamma^2+1)\rho(L)\bsF},\label{nonconvex:v3k:deterministic}
	\end{align}
	where the second equality holds due to \eqref{nonconvex:kia-algo-dc-compact-x} and \eqref{nonconvex:kia-algo-dc-compact-v}; 
	the third equality holds due to \eqref{nonconvex:KL-L-eq}, \eqref{nonconvex:lemma-eq3}, \eqref{nonconvex:vkn} and $\bsE={\bf I}_{nd}-\bsH$; 
	the first inequality holds due the Cauchy--Schwarz inequality and $\rho(\bsE)=1$;  
	the second inequality holds due to the Cauchy--Schwarz inequality and \eqref{nonconvex:lemma-eq5};
	and the last  inequality holds due to \eqref{nonconvex:gg1}, \eqref{nonconvex:gg}, \eqref{nonconvex:KL-L-eq2} and $\rho(\bsE)=1$.

	\noindent {\bf (iv)}  This step shows the relation between $e_{4,k+1}$ and $e_{4,k}$.	
		\begin{align}
		&e_{4,k+1}=n(f(\bar{x}_{k+1})-f^*)=\tilde{f}(\bar{\bsx}_{k+1})-nf^*\nonumber\\
		&=\tilde{f}(\bar{\bsx}_k)-nf^*+\tilde{f}(\bar{\bsx}_{k+1})-\tilde{f}(\bar{\bsx}_k)\nonumber\\
		&\le\tilde{f}(\bar{\bsx}_k)-nf^*
		-\alpha\bar{\bsg}_{k}^\top\bsg^0_k
		+\frac{\alpha^2\ell}{2}\|\bar{\bsg}_{k}\|^2\nonumber\\
		&=\tilde{f}(\bar{\bsx}_k)-nf^*
		-\alpha\bar{\bsg}_{k}^\top\bar{\bsg}^0_k
		+\frac{\alpha^2\ell}{2}\|\bar{\bsg}_{k}\|^2\nonumber\\
		&=n(f(\bar{x}_k)-f^*)
		-\frac{\alpha}{2}\bar{\bsg}_{k}^\top(\bar{\bsg}_k+\bar{\bsg}^0_k-\bar{\bsg}_k)\nonumber\\
		&\quad-\frac{\alpha}{2}(\bar{\bsg}_{k}-\bar{\bsg}^0_k+\bar{\bsg}^0_k)^\top\bar{\bsg}^0_k
		+\frac{\alpha^2\ell}{2}\|\bar{\bsg}_{k}\|^2\nonumber\\
		&\le n(f(\bar{x}_k)-f^*)-\frac{\alpha}{4}\|\bar{\bsg}_{k}\|^2
		+\frac{\alpha}{4}\|\bar{\bsg}^0_k-\bar{\bsg}_k\|^2
		-\frac{\alpha}{4}\|\bar{\bsg}_{k}^0\|^2\nonumber\\
		&\quad+\frac{\alpha}{4}\|\bar{\bsg}^0_k-\bar{\bsg}_k\|^2
		+\frac{\alpha^2\ell}{2}\|\bar{\bsg}_{k}\|^2\nonumber\\
		&=n(f(\bar{x}_k)-f^*)-\frac{\alpha}{4}(1-2\alpha \ell)\|\bar{\bsg}_{k}\|^2
		+\frac{\alpha}{2}\|\bar{\bsg}^0_k-\bar{\bsg}_k\|^2\nonumber\\
		&\quad-\frac{\alpha}{4}\|\bar{\bsg}_{k}^0\|^2\nonumber\\%\label{nonconvex:v4k-zero}\\
		&\le n(f(\bar{x}_k)-f^*)-\frac{\alpha}{4}(1-2\alpha \ell)\|\bar{\bsg}_{k}\|^2
		+\|\bsx_k\|^2_{\frac{\alpha}{2}\ell^2\bsE}\nonumber\\
		&\quad-\frac{\alpha}{4}\|\bar{\bsg}_{k}^0\|^2,\label{nonconvex:v4k:deterministic}
	\end{align}
	where the first inequality holds due to the smoothness of $\tilde{f}$, \eqref{nonconvex:lemma:lipschitz} and \eqref{nonconvex:xbardynamic};
	the fourth equality holds due to $\bar{\bsg}_{k}^\top\bsg^0_k=\bsg_{k}^\top\bsH\bsg^0_k=\bsg_{k}^\top\bsH\bsH\bsg^0_k=\bar{\bsg}_{k}^\top\bar{\bsg}^0_k$;
	the second inequality holds due to the Cauchy--Schwarz inequality;
	and the last inequality holds due to \eqref{nonconvex:gg2}.
\bibliographystyle{IEEEtran}
\bibliography{ref}

\end{document}